\documentclass[10pt]{article}

\usepackage{mathrsfs}
\usepackage{amsmath}
\usepackage{paralist}

\usepackage{amssymb}
\usepackage{color}
\usepackage{mathptmx}
\usepackage{verbatim}
\usepackage{epsfig}
\usepackage{CJK}
\usepackage{bm}
\usepackage{amsfonts}
\usepackage{amsthm}
\input{epsf.sty}
\usepackage{epsfig}
\def\be{\begin{equation}}
\def\ee{\end{equation}}
\def\ba{\begin{array}}
\def\ea{\end{array}}

\newtheorem{thm}{Theorem}[section]

\newtheorem{lem}[thm]{Lemma}

\numberwithin{equation}{section}

\newcommand{\mi}{\mathbf{i}}

\makeatletter
\newcommand{\rmnum}[1]{\romannumeral #1}
\newcommand{\Rmnum}[1]{\expandafter\@slowromancap\romannumeral #1@}
\makeatother

\newtheorem{Remark}{Remark}[section]

\def\be{\begin{equation}}
\def\ee{\end{equation}}
\def\br{\begin{eqnarray}}
\def\er{\end{eqnarray}}

\def\r{\rangle}

\title{KAM for high-dimensional nonlinear quantum harmonic oscillator\footnotetext{Supported by NNSFC (Grant Nos. 12090010, 12090013, 12101227)}}

\author{
 $\mbox{Jianjun \ Liu}$\\
 $\mbox{School of Mathematics, Sichuan University}$\\
 $\mbox{Chengdu 610065, PR China}$\\
 $\mbox{Caihong \ Qi}$\\
 $\mbox{School of Mathematics, Sichuan University}$\\
 $\mbox{ Chengdu 610065, PR China}$\\
 $\mbox{Guanghua \ Shi}$\\
 $\mbox{College of Mathematics and Statistics, Hunan Normal University}$\\
 $\mbox{Changsha 410081, PR China}$}

\date{}
\begin{document}
\maketitle

\noindent \textbf{Abstract}: In this paper, we study high-dimensional nonlinear quantum harmonic oscillator equation. We show the equation admits many time quasi-periodic solutions by establishing an abstract infinite dimensional KAM theorem with multiple normal frequencies. The proof is based on the classical KAM scheme, and the key is a decaying structure of Hessian matrices of Hamiltonian functions.


\section{Introduction}
In this paper, we consider the nonlinear Schr\"{o}dinger equation
	\begin{eqnarray}\label{2024341418}
		{\mi}u_t+(-\Delta+|x|^2)u+M_{\rho}u+\varepsilon|u|^{2p}u=0,~~~x\in\mathbb{R}^d,~~~d\geq2,
	\end{eqnarray}
	where $T:=-\Delta+|x|^2$ is the $d$-dimensional quantum harmonic oscillator, $M_{\rho}$ is a Hermitian multiplier, $\varepsilon>0$ is a small parameter, and $p$ is a positive integer. Under the standard inner product on $L^2(\mathbb{R}^d)$, the equation can be written in the form
	\begin{eqnarray}
		\dot{u}={\mi}\frac{\partial H}{\partial \bar{u}}
	\end{eqnarray}
	with the Hamiltonian
	\begin{eqnarray}
		H=\int_{\mathbb{R}^d}\Big(|\nabla u|^2+x^2|u|^2+ (M_{\rho}u)\bar{u}\Big)dx+\frac{\varepsilon}{p+1}\int_{\mathbb{R}^d}|u|^{2p+2}dx.
	\end{eqnarray}
We will construct time quasi-periodic solutions by KAM theory with the multiplier $M_{\rho}$ as parameter and the nonlinearity $\varepsilon|u|^{2p}u$ as perturbation.

Historically, KAM theory for partial differential equations was originated by
Kuksin \cite{K4} and Wayne \cite{Wayne}, where one dimensional nonlinear wave and Schr\"{o}dinger equations with Dirichlet boundary conditions were studied. Then infinite dimensional KAM theory was deeply developed with applications to more partial differential equations, including both one dimension and higher dimension. For one dimensional partial differential equations, also see \cite{C-You,Gre,K5,Poschel3,P,P2,Y3} for example. For higher dimensional partial differential equations, the second order Melnikov non-resonance conditions are seriously violated. By developing further Criag and Wayne's scheme in \cite{Criag-Wayne}, Bourgain \cite{Bourgain1} made a breakthrough for two dimensional nonlinear Schr$\ddot{\mbox{o}}$dinger equations. This method only requires the first order Melnikov non-resonance conditions and is now known as Craig-Wayne-Bourgain (C-W-B) method. See \cite{Berti7,Berti4,Berti-C-P,Bourgain2,Yuan2003} for higher dimensional nonlinear Schr$\ddot{\mbox{o}}$dinger equations and nonlinear wave equations. Due to the absence of the second order Melnikov conditions, C-W-B method generally gives no information on the linear stability of the obtained invariant tori. In contrast, the classical KAM scheme gives a normal form, which usually contains the information of linear stability. Furthermore, this normal form is a cornerstone to investigate nonlinear stability around the invariant tori.

In \cite{Eliasson}, the classical KAM scheme was developed by Eliasson and Kuksin for high dimensional nonlinear Schr$\ddot{\mbox{o}}$dinger equations on $\mathbb{T}^d$. To cope with the infinitely many second order Melnikov conditions, they made a breakthrough by exploring the T$\ddot{\mbox{o}}$plitz-Lipschitz structure of Hessian matrices of Hamiltonian functions. For more results of high dimensional partial differential equations on $\mathbb{T}^d$, also see \cite{Eliasson1,GY1,GY,procesi,Yuan2021} for example. Different from $\Delta$ on $\mathbb{T}^d$, another important type is represented by $\Delta$ on $\mathbb{S}^d$ and quantum harmonic oscillator on $\mathbb{R}^d$. Between them, the first significant difference is the clustering of eigenvalues. For $\Delta$ on $\mathbb{T}^d$, a sub-clustering with upper bound was discovered in \cite{Eliasson}, which ensures the solving of homological equations without loss of regularity; for the latter, the natural clustering is given by the resonant subsets in
$$\{(j,\iota)\in \mathbb{N}^+\times\mathbb{N}^+~|~\iota=1,\cdots,d_j\},$$
where $d_j\rightarrow+\infty$ means that there is no upper bound. The second significant difference is the analysis nature of eigenfunctions. For $\Delta$ on $\mathbb{T}^d$, the eigenfunctions are $e^{\mi\langle j,x\rangle}$, $j\in\mathbb{Z}^d$, which provide a foundation for the T$\ddot{\mbox{o}}$plitz-Lipschitz structure mentioned above; for the latter, the eigenfunctions are spherical harmonic functions or Hermite functions, which do not fit in the Fourier analysis.

Regarding these difficulties, in \cite{Gre2}, Gr\'{e}bert and Paturel introduced a matrix norm for the Hessian $A:=\nabla_{\zeta}^2f$ of perturbation $f$ as follows:
\begin{eqnarray}\label{20247182203}
	|A|_{s,\beta}:=\sup_{i,j}\big\|A_i^j\big\|(ij)^{\beta}\Big(\frac{i\wedge j+|i^2-j^2|}{i\wedge j}\Big)^\frac{s}{2},
\end{eqnarray}
where $\|A_i^j\big\|$ denotes $\ell^2$-operator norm of the block $A_i^j$, $i\wedge j$ denotes $\min\{i,j\}$, and the weight $\beta>0$ characterizes the regularizing effect and is crucial to control the number of small divisors. To maintain the structure \eqref{20247182203} under KAM iteration, the gradient $\nabla_{\zeta}f$ of perturbation $f$ must also be regular, i.e.,
\begin{eqnarray}\label{20247192023}
	\nabla_{\zeta}f: Y_s\to Y_{s+\beta},
\end{eqnarray}
where $Y_s$ are Sobolev type spaces. Under \eqref{20247182203} and \eqref{20247192023}, the authors established an abstract KAM theorem for infinite dimensional Hamiltonian systems with multiple normal frequencies, and applied this theorem to the nonlinear Klein-Gordon equation on $\mathbb{S}^d$.

For high-dimensional nonlinear quantum harmonic oscillator equation \eqref{2024341418}, there is a similar structure as \eqref{20247182203} for Hessian of perturbation, i.e.,
\begin{eqnarray}\label{20240726-1}
	|A|_{s,\beta}:=\sup_{i,j}\|A_i^j\|(ij)^{\beta}\Big(\frac{\sqrt{i\wedge j}+|i-j|}{\sqrt{i\wedge j}}\Big)^s,
\end{eqnarray}
but the gradient of perturbation is merely bounded, i.e.,
\begin{eqnarray}\label{20240726-2}
	\nabla_{\zeta}f: Y_s\to Y_{s},
\end{eqnarray}
which causes an essential obstacle of maintaining the structure \eqref{20240726-1}. Precisely, consider the Poisson bracket $\{f,g\}$, where
\begin{eqnarray}\label{20240726-3}
f:=f_{\theta}(\theta)+\langle f_r(\theta),r\rangle+\langle f_{\zeta}(\theta),\zeta\rangle+\frac{1}{2}\langle f_{\zeta\zeta}(\theta)\zeta,\zeta\rangle+\text{h.o.t.}
\end{eqnarray}
is the Hamiltonian function of perturbation and
\begin{equation}\label{20247182202}
g:=g_{\theta}(\theta)+\langle g_r(\theta),r\rangle+\langle g_{\zeta}(\theta),\zeta\rangle+\frac{1}{2}\langle g_{\zeta\zeta}(\theta)\zeta,\zeta\rangle
\end{equation}
is the Hamiltonian function of transformation. By direct calculation,
\begin{eqnarray}\label{20247182201}
\nabla_{\zeta}^2\{f,g\}=f_{\zeta\zeta}\cdot Jg_{\zeta\zeta}-\nabla_{\zeta}(\partial_rf)\otimes\nabla_{\zeta}(\partial_{\theta}g)+\cdot\cdot\cdot,
\end{eqnarray}
where $J=\begin{pmatrix}
			0 & -1\\ 1 & 0
		\end{pmatrix}$
and $\otimes$ denotes the tensor product. The obstacle is the second term, i.e., $\nabla_{\zeta}(\partial_rf)\otimes\nabla_{\zeta}(\partial_{\theta}g)$, which does not possess the structure \eqref{20240726-1}. Interestingly, there is still a very nice application of \eqref{20240726-1} to the reducibility of high-dimensional quantum harmonic oscillator in \cite{Gre3}, since the second term in (\ref{20247182201}) disappears for linear equation.

In the present paper, we devote to studying the case that the Hessian $\nabla_{\zeta}^2f$ has regularizing effect but the gradient $\nabla_{\zeta}f$ is merely bounded. Our approach is to weaken the norm defined in (\ref{20240726-1}) such that the tensor product in (\ref{20247182201}) possesses a weaker structure. Precisely, we define
\begin{eqnarray}\label{20247192024}
|A|_{s,\beta}:=\sup_{i,j}\|A_i^j\|(i\wedge j)^{\beta}\Big(\frac{\sqrt{i\wedge j}+|i-j|}{\sqrt{i\wedge j}}\Big)^s
\end{eqnarray}
with $(i\wedge j)^{\beta}$ instead of the factor $(ij)^{\beta}$ in \eqref{20240726-1}. Now, to guarantee the tensor product in (\ref{20247182201}) possesses the structure (\ref{20247192024}), it just requires that one vector is regular and the other is bounded, instead of both regular. Generally in KAM scheme, the corresponding terms of perturbation $f$ and transformation $g$ have different regularities. Let's take the equation \eqref{2024341418} and the terms $\langle f_{\zeta},\zeta\rangle$, $\langle g_{\zeta},\zeta\rangle$ in \eqref{20240726-3}, \eqref{20247182202} for example. The fact is that, the former is bounded, i.e., $f_{\zeta}\in Y_{s}$, while the latter is regularized by the first order Melnikov conditions, i.e., $g_{\zeta}\in Y_{s+1}$. Therefore, we think our decaying structure may be used in a wider range.

More precisely, in order to maintain the structure (\ref{20247192024}), we also introduce a well-matched matrix norm for the Hessian $S:=\nabla_{\zeta}^2g$ of transformation $g$ as follows:
\begin{eqnarray}\label{20240726-4}
|S|_{s,\beta+}:=\sup_{i,j}\|S_i^j\|(1+|i-j|)(i\wedge j)^{\beta}\Big(\frac{\sqrt{i\wedge j}+|i-j|}{\sqrt{i\wedge j}}\Big)^s.
\end{eqnarray}
Respectively denote $\mathcal{M}_{s,\beta}$ and $\mathcal{M}^+_{s,\beta}$ the spaces of matrices with the norm in \eqref{20247192024} and \eqref{20240726-4}. Formally the same as \cite{Gre2,Gre3}, we need to prove: if $A\in\mathcal{M}_{s,\beta}$ and $S\in\mathcal{M}^+_{s,\beta}$, then $AS,SA\in\mathcal{M}_{s,\beta}$; if $S_1,S_2\in\mathcal{M}^+_{s,\beta}$, then $S_1S_2\in\mathcal{M}^+_{s,\beta}$. Compared with the factor $(ij)^{\beta}$ in \cite{Gre2,Gre3}, the conclusions for $(i\wedge j)^{\beta}$ are less obvious and the proof is harder.

Except for the trouble caused by $(i\wedge j)^{\beta}$ above, there is another tricky problem. By (\ref{20247182202}), 
\begin{eqnarray}\label{20240726-5}
\nabla_{\zeta}g=g_{\zeta}+g_{\zeta\zeta}\zeta,
\end{eqnarray}
where the regularity of the first term $g_{\zeta}$ has been discussed above, and the second term $g_{\zeta\zeta}\zeta$ is left as a problem. For the factor $(ij)^{\beta}$, this problem is solved in \cite{Gre2,Gre3} by the following conclusion:
if $S\in\mathcal{M}_{s,\beta}^+$, then $S\in\mathcal{L}(Y_s,Y_{s+\beta})$. However, for the weaker factor $(i\wedge j)^{\beta}$, the conclusion is no longer valid. Fortunately, we find a weaker conclusion: if $S\in\mathcal{M}_{s,\beta}^+$, then $S\in\mathcal{L}(Y_s,Y_{s+\beta'})$ for any $0<\beta'<\beta$. Of course, this leads to $\nabla_{\zeta}^2\{f,g\}$ belongs to $\mathcal{M}_{s,\beta'}$ instead of $\mathcal{M}_{s,\beta}$, which means that the perturbation after each KAM step will lose some regularizing effect. Thus, we take $\beta$ as an iteration parameter and shrink it in the KAM procedure.

Finally, we also mention KAM theory for partial differential equations with their nonlinearity containing spatial derivatives such as KdV, derivative nonlinear wave and Schr$\ddot{\mbox{o}}$dinger equations, water wave equations, seeing \cite{Baldi4,Baldi2,Berti2,Berti-H-M,kappeler,Kuksin1,L-Y1} for example. 

In the present paper, for convenience, we keep fidelity with the notation and terminology from \cite{Gre2,Gre3}. Now we lay out an outline:
In Section 2, we first state an abstract infinite dimensional KAM theorem, seeing Theorem \ref{2024371454}, and then apply this theorem to the high-dimensional nonlinear quantum harmonic oscillator equation, seeing Theorem \ref{20240726-6}.
In Section 3, we prove several basic properties of the matrix spaces $\mathcal{M}_{s,\beta}$ and $\mathcal{M}_{s\,\beta}^+$. These properties are crucial to our main results and have been discussed above.
In Section 4, we first study Poisson brackets of two Hamiltonian functions, seeing Lemma \ref{20243212055}. Since the space $\mathcal{T}_{\sigma,\mu,\mathcal{D}}^{s,\beta}$ is not closed under the Poisson bracket, we introduce its subspace $\mathcal{T}_{\sigma,\mu,\mathcal{D}}^{s,\beta+}$. Then, we give a precise description of the Hamiltonian flow and some useful estimates of the transformation, seeing Lemma \ref{20243181426}. Notice that this lemma takes into account the change of the parameter $\beta$. Finally, we prove that the Hamiltonian flow preserves the regularizing effect in the sense of smaller $\beta$, seeing Lemma \ref{20243212033}.
In Section 5-7, the KAM theorem is proved. In section 5, we perform one KAM step. In section 6, we prove the iterative lemma, seeing Lemma \ref{20246101129}, where a new iterative parameter $\beta_m$ is introduced. In section 7, we estimate the measure of excluded parameters, seeing Lemma \ref{20240726-7}.
In Appendix, we give some useful lemmas. Lemma \ref{20243122043} describes a property of the weight $w(i,j)$. Lemma \ref{20243122017} provides two elementary inequalities being frequently used in Section 3. We emphasize that, when using this lemma to the proof of Lemma \ref{20243141530-3}, the parameter $\beta$ in \eqref{20240527-1} is replaced by $\beta-\beta'$, which leads to the large coefficient $\frac{1}{\beta-\beta'}$ in (\ref{2024581524}). Lemma \ref{20240527-2} shows that the operator norm can be controlled by the Frobenius norm. Lemma \ref{20240527-4} is the Schur inequality. Lemma \ref{20246251433} is quoted from \cite{Gre3} and is used to solve the homological equation (\ref{20245262206}). Lemma \ref{20243212004} is quoted from \cite{Gre2} and is devoted to estimating the new perturbation $f_{3+}$ in \eqref{2024672221}. Lemma \ref{20247212055} is quoted from \cite{Gre3} and is used to check the Hessian of perturbation $f$ satisfies the norm \eqref{20240511-1}.

\section{Main results}
In this section, we give an abstract KAM theorem with multiple normal frequencies and its application to high-dimensional nonlinear Schr\"{o}dinger equation with harmonic potential.
\subsection{An abstract KAM theorem}	
At beginning, we introduce some notations. Denote the index set
\begin{eqnarray}\label{20240721-11}
		\mathfrak{L}:=\{(j,\iota)\in \mathbb{N}^+\times\mathbb{N}^+~|~\iota=1,\cdots,d_j\},
	\end{eqnarray}
where $d_j\leq c^*j^{d^*}$ for some constants $c^*>0$ and $d^*>0$. For a complex vector valued sequence $$\zeta=(\zeta_{j,\iota}\in\mathbb{C}^2)_{(j,\iota)\in\mathfrak{L}},$$
we also write $\zeta=(\zeta_j)_{j\geq1}$ with $\zeta_j=(\zeta_{j,\iota})_{1\leq\iota\leq d_j}$. Denote the $\ell^2$-norm
$|\zeta_j|:=\sqrt{\sum_{1\leq\iota\leq d_j}|\zeta_{j,\iota}|^2}$, where $|\zeta_{j,\iota}|=\sqrt{|p_{j,\iota}|^2+|q_{j,\iota}|^2}$ for $\zeta_{j,\iota}=(p_{j,\iota},q_{j,\iota})$. Let $s\geq0$ and define the space $Y_s$ of complex vector valued sequences $\zeta$ with the norm
	\begin{eqnarray}		\|\zeta\|_s^2:=\sum_{j\geq1}|\zeta_j|^2j^{2s}=\sum_{(j,\iota)\in\mathfrak{L}}|\zeta_{j,\iota}|^2j^{2s}=\sum_{(j,\iota)\in\mathfrak{L}}(|p_{j,\iota}|^2+|q_{j,\iota}|^2)j^{2s}<+\infty.
	\end{eqnarray}

Denote by $\mathrm{M}_{m\times n}(\mathbb{C})$ the set of complex $m\times n$ matrices. For an infinite-dimensional $\mathrm{M}_{2\times 2}(\mathbb{C})$-valued matrix
$$A:~\mathfrak{L}\times\mathfrak{L}\to \mathrm{M}_{2\times 2}(\mathbb{C}),~~~~~~\big((i,\iota_1),(j,\iota_2)\big)\mapsto A_{(i,\iota_1)}^{(j,\iota_2)},$$
we write its block wise representation $A=(A_i^j)_{i,j\geq1}$ with $A_i^j\in\mathrm{M}_{2d_i\times 2d_j}(\mathbb{C})$. For any $i,j\geq1$, denote $i{\wedge}j=\min\{i,j\}$ and
\begin{equation}\label{20240509-1}		
w(i,j)=\frac{\sqrt{i{\wedge}j}+|i-j|}{\sqrt{i{\wedge}j}}.
\end{equation}
Fix $\beta>0$ and define the space $\mathcal{M}_{s,\beta}$ of infinite matrices $A$ with the norm
	\begin{eqnarray}\label{20240511-1}		|A|_{s,\beta}:=\sup_{i,j\geq1}(i{\wedge}j)^{\beta}w(i,j)^s\big\|A_i^j\big\|<+\infty,
	\end{eqnarray}
where $\|A_i^j\big\|$ denotes the $\ell^2$-operator norm from $\mathbb{C}^{2d_j}$ to $\mathbb{C}^{2d_i}$.

For $\sigma>0$, denote
	\begin{eqnarray*}
		\mathbb{T}_{\sigma}^n=\{z\in \mathbb{C}^n:|\Im z|<\sigma\}/ {2\pi\mathbb{Z}^n};
	\end{eqnarray*}
for $\sigma,\mu>0$ and $s\geq 0$, denote
	\begin{eqnarray*}
		\mathcal{O}^s(\sigma,\mu)=\mathbb{T}_{\sigma}^n\times\{r\in \mathbb{C}^n:|r|<\mu^2\}\times\{\zeta\in Y_s:\|\zeta\|_s<\mu\},
	\end{eqnarray*}
where $|\Im z|$ and $|r|$ are maximum norm of $n$-dimensional vectors. Let $\mathcal{D}\subset \mathbb{R}^n$ be a compact set of positive Lebesgue measure. This is the set of parameters upon which will depend our objects. Differentiability of functions on $\mathcal{D}$ is understood in the sense of Whitney.

Let $f:\mathcal{O}^s(\sigma,\mu)\times\mathcal{D}\to \mathbb{C}$ be a function, real holomorphic in $(\theta,r,\zeta)\in\mathcal{O}^s(\sigma,\mu)$ and $C^1$ in $\rho\in \mathcal{D}$, such that for all $\rho\in \mathcal{D}$,
	\begin{eqnarray*}
		\mathcal{O}^s(\sigma,\mu)\ni (\theta,r,\zeta)\mapsto \nabla_{\zeta}f(\theta,r,\zeta;\rho)\in Y_s
	\end{eqnarray*}
	and
	\begin{eqnarray*}
		\mathcal{O}^s(\sigma,\mu)\ni (\theta,r,\zeta)\mapsto \nabla_{\zeta}^2f(\theta,r,\zeta;\rho)\in \mathcal{M}_{s,\beta}
	\end{eqnarray*}
are real holomorphic functions. Denote this set of functions by $\mathcal{T}^{s,\beta}(\sigma,\mu,\mathcal{D})$ and define the norm $[f]_{\sigma,\mu,\mathcal{D}}^{s,\beta}$ through
	\begin{eqnarray}\label{20240605-1}
		\sup_{\nu=0,1,\rho\in \mathcal{D}\atop (\theta,r,\zeta)\in\mathcal{O}^s(\sigma,\mu)}\max\big\{|\partial_{\rho}^{\nu}f(\theta,r,\zeta;\rho)|,\mu\|\partial_{\rho}^{\nu}\nabla_{\zeta}f(\theta,r,\zeta;\rho)\|_s,
\mu^2|\partial_{\rho}^{\nu}\nabla_{\zeta}^2f(\theta,r,\zeta;\rho)|_{s,\beta}\big\}.
	\end{eqnarray}
We remark that for $\beta>0$, only the Hessian of $f\in \mathcal{T}^{s,\beta}$ has a regularizing effect, so does not the gradient of $f$.

Introduce the orthogonal projection $\Pi$ defined on the $2\times2$ complex matrices
	\begin{eqnarray*}
		\Pi~:\mathrm{M}_{2\times 2}(\mathbb{C})\to \mathbb{C}I+\mathbb{C}J,
	\end{eqnarray*}
	where $I=	\begin{pmatrix}
			1 & 0\\ 0 & 1
		\end{pmatrix}$
    and $J=	\begin{pmatrix}
			0 & -1\\ 1 & 0
		\end{pmatrix}$.
		A matrix $A:~\mathfrak{L}\times\mathfrak{L}\to \mathrm{M}_{2\times 2}(\mathbb{C})$ is called on normal form and denoted by $A\in\mathcal{NF}$ if it is real valued, symmetric, i.e., $A={}^tA$, block diagonal, i.e., $A_i^j=0$ for $i\neq j$, and satisfies $\Pi A=A$.

Introduce the symplectic structure
	\begin{eqnarray*}
		\sum_{b=1}^{n}dr_b\wedge d\theta_b+\sum_{(j,\iota)\in\mathfrak{L}}dp_{j,\iota}\wedge dq_{j,\iota},
	\end{eqnarray*}
and a Hamiltonian is called on normal form  if it reads
	\begin{eqnarray*}
		h=\langle \omega(\rho),r\rangle+\frac{1}{2}\langle \zeta,A(\rho)\zeta\rangle,
	\end{eqnarray*}
	where $\omega\in \mathbb{R}^n$ is the tangential frequency,  $A$ is a matrix on normal form, and $\rho\in\mathcal{D}$ is an external parameter.

	\begin{thm}\label{2024371454}
Consider the perturbed Hamiltonian system
	\begin{eqnarray}\label{20246302006}
		\begin{aligned}				h+f=\langle\omega(\rho),r\rangle+\frac{1}{2}\langle\zeta,A(\rho)\zeta\rangle+f(\theta,r,\zeta;\rho)
		\end{aligned}	
	\end{eqnarray}
with the normal form $h$ and the perturbation $f$ satisfying the following assumptions:

		\textbf{(A1)} Both the map $\rho\mapsto \omega(\rho)$ and its inverse map are $C^1$ between $\mathcal{D}$ and its image. More precisely, there exist positive constants $M_{\omega}$ and $L$, such that
\begin{eqnarray}\label{20240709-1}
&&\sup_{\nu=0,1,\rho\in\mathcal{D}}|\partial_{\rho}^{\nu}\omega(\rho)|\leq M_{\omega},\\
&&\sup_{\rho\in\omega(\mathcal{D})}|\partial_{\rho}\omega^{-1}(\rho)|\leq L.\label{20240709-2}
\end{eqnarray}

			\textbf{(A2)} The normal form matrix $A(\rho)$ is block diagonal with its components
\begin{eqnarray}\label{20240709-5}
A_j^j(\rho)=\lambda_jI+\tilde{A}_j^j(\rho),~~~~~~\lambda_j=b_1j+b_0,
\end{eqnarray}
where $b_1, b_0$ are constants satisfying $b_1>0$ and $b_1+b_0>0$, $I$ is the unit matrix of the same order as $A_j^j$, and $\tilde{A}_j^j(\rho)$ is $C^1$ in the parameter $\rho\in\mathcal{D}$. Moreover, there exist positive constants $M_{\Omega}$ and $\tilde{\beta}$, such that for any $j\geq 1$,
			\begin{eqnarray}\label{20240709-3} \sup_{\nu=0,1,\rho\in\mathcal{D}}\|\partial_{\rho}^{\nu}\tilde{A}_j^j(\rho)\|\leq \frac{M_{\Omega}}{j^{\tilde{\beta}}}.
			\end{eqnarray}
Additionally we assume
			\begin{eqnarray}\label{20240721-3}
M_{\Omega}<\min\{\frac{b_1}{8},\frac{1}{12L}\}.
			\end{eqnarray}

		\textbf{(A3)} The perturbation $f$ is real analytic in the space coordinates $\mathcal{O}^s(\sigma,\mu)$ and $C^1$ in the parameter $\rho\in\mathcal{D}$, where $0<\sigma,\mu\leq1$. Moreover, there exist positive constants $s,\beta$ (without loss of generality, assume $\beta\leq\tilde{\beta}$) with $0<\beta\leq\min\{\frac{s}{2},\frac{1}{2}\}$ such that $f\in\mathcal{T}_{\sigma,\mu,\mathcal{D}}^{s,\beta}$.

Then there exists $\varepsilon^*>0$ depending on $n$, $s$, $b_0$, $b_1$, $c^*$, $d^*$, $M_{\omega}$, $L$, $M_{\Omega}$, $\beta$, $\sigma$, $\mu$ such that if
		\begin{eqnarray}
			[f]_{\sigma,\mu,\mathcal{D}}^{s,\beta}=\varepsilon<\varepsilon^*,
		\end{eqnarray}
then there exist:

(1) a Cantor set $\mathcal{D}_*\subset\mathcal{D}$ with
\begin{eqnarray}\label{20240721-4}
\text{meas}(\mathcal{D}\setminus\mathcal{D}_*)\leq\varepsilon^{\frac{1}{\alpha}},
\end{eqnarray}
where
\begin{eqnarray}\label{20240721-5}
\alpha=32(1+\frac{d^*}{\beta})(1+\frac{4d^*+2}{\beta});
\end{eqnarray}

(2) a family of real analytic symplectic diffeomorphism
$\Phi:\mathcal{O}^s(\sigma/2,\mu/2)\times\mathcal{D}_*\to\mathcal{O}^s(\sigma,\mu)$
satisfying
\begin{eqnarray}\label{20240721-6}
\|\Phi-id\|_s\leq\varepsilon^{\frac{1}{2}},
\end{eqnarray}
such that
		\begin{eqnarray}
			(h+f)\circ \Phi=\langle \omega_*(\rho),r\rangle+\frac{1}{2}\langle \zeta,A_*(\rho)\zeta\rangle+{f}_*(\theta,r,\zeta;\rho),
		\end{eqnarray}
where the norm $\|(\theta,r,\zeta)\|_s=\max\{|\theta|,|r|,\|\zeta\|_s\}$ for $(\theta,r,\zeta)\in\mathbb{C}^n\times \mathbb{C}^n\times Y_s$, the new tangential frequency  $\omega_*$ satisfies
\begin{eqnarray}\label{20240721-7}
\sup_{\nu=0,1,\rho\in\mathcal{D}_*}|\partial_{\rho}^{\nu}(\omega_*(\rho)-\omega(\rho))|
\leq\frac{2\varepsilon}{\mu^2},
\end{eqnarray}
the new normal form matrix $A_*$ satisfies
\begin{eqnarray}\label{20240721-8}
\sup_{\nu=0,1,\rho\in\mathcal{D}_*}\|\partial_{\rho}^{\nu}(A_*(\rho)-A(\rho))_j^j\|\leq \frac{2\varepsilon}{\mu^2} j^{-\frac{\beta}{2}},
\end{eqnarray}
and the new perturbation $f_*$ satisfies $\partial_r{f_*}=\partial_{\zeta}{f_*}=\partial_{\zeta}^2{f_*}=0$ for $r=\zeta=0$.

As a result, for each $\rho\in\mathcal{D}_*$,the map $\Phi$ restricted to $\mathbb{T}^n\times\{0\}\times\{0\}$ is an analytic embedding of a rotational torus with frequencies $\omega_*(\rho)$ for the perturbed Hamiltonian $h+f$ at $\rho$. In other words,
$$t\mapsto\Phi\big(\theta+t\omega_*(\rho),0,0;\rho\big),\quad t\in\mathbb{R}$$
is an analytic quasi-periodic solution for the Hamiltonian $h+f$ evaluated at $\rho$ for every $\theta\in\mathbb{T}^n$ and $\rho\in\mathcal{D}_*$.
\end{thm}

\begin{Remark}
	We remark that $\alpha$ in (\ref{20240721-5}) is not optimal. However, it could tell us that the digged set $\mathcal{D}\setminus\mathcal{D}_*$ will become larger if $d^*$ becomes larger or $\beta$ becomes smaller.
\end{Remark}

\subsection{Application to high-dimensional nonlinear quantum harmonic oscillator}

In this subsection, we shall apply Theorem \ref{2024371454} to the equation \eqref{2024341418}. The operator $T$ has eigenvalues $\{\lambda_j\}$ satisfying
	\begin{eqnarray*}
		\lambda_{j}=2j-2+d,~~~j\geq 1.
	\end{eqnarray*}
Let $X_j$ be the associated eigenspace, whose dimension denoted by $\tilde{d}_j$ is less than $j^{d-1}$. We denote by $\Psi_{j,\iota}$ the Hermitian function of degree $j$ and order $\iota$ so that we have
\begin{eqnarray*}
	X_j=\text{Span}\{\Psi_{j,\iota},~\iota=1,\cdots,\tilde{d}_j\}.
\end{eqnarray*}
Denoting
\begin{eqnarray*}
	\mathfrak{L}_0:=\{(j,\iota)~|~j\geq 1,~\iota=1,\cdots,\tilde{d}_j\},
\end{eqnarray*}
then $\{\Psi_{j,\iota}\}_{(j,\iota)\in\mathfrak{L}_0}$ forms\ a basis of $L^2(\mathbb{R}^d)$. Define the Hermitian multiplier on the basis as
  \begin{eqnarray*}
  	M_{\rho}\Psi_{j,\iota}=\rho_{j,\iota}\Psi_{j,\iota}~~~~~~\text{for}~~ (j,\iota)\in\mathfrak{L}_0,
  \end{eqnarray*}
  where $(\rho_{j,\iota})_{(j,\iota)\in\mathfrak{L}_0}$ is a bounded sequence of real numbers.
Writing $u$ as
  \begin{eqnarray*}
  	u=\sum_{(j,\iota)\in\mathfrak{L}_0}\xi_{j,\iota}\Psi_{j,\iota}(x),
  \end{eqnarray*}
then the equation \eqref{2024341418} turns into
	\begin{eqnarray}\label{20246291645}
			\dot{\xi}_{j,\iota}={\mi}\frac{\partial H}{\partial\bar{\xi}_{j,\iota}}
	\end{eqnarray}
	with the Hamiltonian function
	\begin{eqnarray}	H=\sum_{(j,\iota)\in\mathfrak{L}_0}(\lambda_{j,\iota}+\rho_{j,\iota})\xi_{j,\iota}\bar{\xi}_{j,\iota}+\frac{\varepsilon}{p+1}\int_{\mathbb{R}^d}\Big|\sum_{(j,\iota)\in\mathfrak{L}_0}\xi_{j,\iota}\Psi_{j,\iota}(x)\Big|^{2p+2}dx.
	\end{eqnarray}
Introduce the real variables $\zeta_{j,\iota}=\begin{pmatrix}p_{j,\iota}\\q_{j,\iota} \end{pmatrix}$ by letting $\xi_{j,\iota}=\frac{1}{\sqrt{2}}(p_{j,\iota}+{\mi}q_{j,\iota})$. Then  \eqref{20246291645} becomes
	\begin{eqnarray}
		\begin{cases}
			\dot{p}_{j,\iota}=\frac{-\partial H}{\partial q_{j,\iota}}\\\dot{q}_{j,\iota}=\frac{\partial H}{\partial p_{j,\iota}}
		\end{cases}
	\end{eqnarray}
	with
	\begin{eqnarray}\label{20246291705} H=\frac{1}{2}\sum_{(j,\iota)\in\mathfrak{L}_0}(\lambda_{j,\iota}+\rho_{j,\iota})(p_{j,\iota}^2+q_{j,\iota}^2)+\frac{\varepsilon}{p+1}\int_{\mathbb{R}^d}\Big|\sum_{(j,\iota)\in\mathfrak{L}_0}\frac{p_{j,\iota}+{\mi}q_{j,\iota}}{\sqrt{2}}\Psi_{j,\iota}(x)\Big|^{2p+2}dx.
	\end{eqnarray}
	Choose $\mathcal{A}:=\{(j_1,\iota_1),(j_2,\iota_2),\cdots,(j_n,\iota_n)\}\subset \mathfrak{L}_0$ as tangential sites and $\mathfrak{L}:=\mathfrak{L}_0\backslash \mathcal{A}$ as normal sites. The index set $\mathfrak{L}$ is of the form \eqref{20240721-11} with $d_j\leq j^{d-1}$, that is, $c^*=1$, $d^*=d-1$. We assume
	\begin{eqnarray}\label{20240723-1}
		\begin{cases}	\rho_b=\rho_{j_b,\iota_b}\in[0,1],~~~b=1,2,\cdots,n\\\rho_{j,\iota}=0,~~~(j,\iota)\in\mathfrak{L}
		\end{cases}
	\end{eqnarray}
	and take $\rho:=(\rho_1,\cdots,\rho_n)\in[0,1]^n:=\mathcal{D}$ as KAM parameters. Fix $I_b>0$ and introduce the action-angle variables $(\theta_b,r_b)$ for $(j_b,\iota_b)\in\mathcal{A}$:
	\begin{eqnarray*} p_{j_b,\iota_b}=\sqrt{2(I_b+r_b)}\cos(\theta_b),~~~q_{j_b,\iota_b}=\sqrt{2(I_b+r_b)}\sin(\theta_b).
	\end{eqnarray*}
	Then the Hamiltonian \eqref{20246291705} becomes $H=h+f$ with
	\begin{eqnarray}\label{20246291942}
		\begin{aligned}	&h=\langle\omega(\rho),r\rangle+\frac{1}{2}\langle\zeta,A\zeta\rangle,\\
&f=\frac{\varepsilon}{p+1}\int_{\mathbb{R}^d}\Big|\sum_{b=1}^n\sqrt{I_b+r_b}e^{\mi\theta_b}\Psi_{j_b,\iota_b}(x)+\sum_{(j,\iota)\in\mathfrak{L}}\frac{p_{j,\iota}+\mi q_{j,\iota}}{\sqrt{2}}\Psi_{j,\iota}(x)\Big|^{2p+2}dx,
		\end{aligned}
	\end{eqnarray}
	where $\zeta=(\zeta_{j,\iota})_{(j,\iota)\in\mathfrak{L}}$ and
	\begin{eqnarray}\label{20246291937} &&\omega(\rho)=(\omega_1(\rho),\cdots,\omega_n(\rho)),~~~~~~\omega_b(\rho)=2j_b-2+d+\rho_b,\\
&&A=\text{diag}(A_j^j,j\geq1),~~~~~~A_j^j=(2j-2+d)I_{2d_j\times2d_j}.\label{20246291938}
	\end{eqnarray}

In view of \eqref{20246291937}, the tangential frequency $\omega$ satisfies the assumption \textbf{(A1)} with
$$M_{\omega}=\max_{1\leq b\leq n}2j_b+d-1,~~~~~~L=1.$$
In view of \eqref{20246291938}, we know $b_1=2$, $b_0=d-2$ and $\tilde{A}_j^j(\rho)=0$, which implies that we can choose $M_{\Omega}=\frac{1}{100}$ and $\tilde{\beta}=1$ such that \eqref{20240709-3} and \eqref{20240721-3} in the assumption \textbf{(A2)} hold. It remains to check the perturbation $f$ in \eqref{20246291942} satisfies the assumption \textbf{(A3)}. Before that, introduce the space
	\begin{eqnarray*}	\mathcal{H}^s=\Big\{u(x)=\sum_{(j,\iota)\in\mathfrak{L}_0}\xi_{j,\iota}\Psi_{j,\iota}(x)~\big|~\sum_{(j,\iota)\in\mathfrak{L}_0}|\xi_{j,\iota}|^{2}j^{2s}<+\infty\Big\},
	\end{eqnarray*}
with the norm
 $\|u\|_{\mathcal{H}^s}^2:=\sum_{(j,\iota)\in\mathfrak{L}_0}|\xi_{j,\iota}|^2j^{2s}$.
	Clearly,
	\begin{eqnarray*}
		\mathcal{H}^s=D(T^{s}):= \{u\in L^2(\mathbb{R}^d):~T^{s}u\in L^2(\mathbb{R}^d)\},
	\end{eqnarray*}
Fix
$$\sigma=\frac{1}{2},~~~\mu=\frac{1}{2}\min\{1,I_1,\cdots,I_n\},~~~s>\max\{d-2,\frac{d}{4}\}.$$
The fact $(\theta,r,\zeta)\in\mathcal{O}^{s}(\sigma,\mu)$ implies $u\in\mathcal{H}^{s}$. Then $|u|^{2p+2}$ belongs to $\mathcal{H}^{s}$ since $\mathcal{H}^{s}$ is an algebra. Thus the perturbation $f$ is analytic in $\mathcal{O}^{s}(\sigma,\mu)$ with
		\begin{eqnarray}\label{2024721703}
\sup_{(\theta,r,\zeta)\in\mathcal{O}^{s}(\sigma,\mu)}|f(\theta,r,\zeta)|\leq C\varepsilon,
		\end{eqnarray}
		where $C$ is a positive constant depending on $s,p,d$, $\{j_b\}_{1\leq b\leq n}$, $\{I_b\}_{1\leq b\leq n}$. Furthermore, one has
		\begin{eqnarray*}
			\frac{\partial f}{\partial p_{j,\iota}}=\frac{\varepsilon}{\sqrt{2}}\Big(\int_{\mathbb{R}^d}u^p\bar{u}^{p+1}\Psi_{j,\iota}(x)dx+\int_{\mathbb{R}^d}u^{p+1}\bar{u}^{p}\Psi_{j,\iota}(x)dx\Big)
		\end{eqnarray*}
		and
		\begin{eqnarray*}
			\frac{\partial f}{\partial q_{j,\iota}}=\frac{\varepsilon{\mi}}{\sqrt{2}}\Big(\int_{\mathbb{R}^d}u^p\bar{u}^{p+1}\Psi_{j,\iota}(x)dx-\int_{\mathbb{R}^d}u^{p+1}\bar{u}^{p}\Psi_{j,\iota}(x)dx\Big).
		\end{eqnarray*}
Thus $\frac{\partial f}{\partial \zeta_{j,\iota}}={}^t(\frac{\partial f}{\partial p_{j,\iota}},\frac{\partial f}{\partial q_{j,\iota}})$ is the $(j,\iota)$-th coefficients of the decomposition of $u^p\bar{u}^{p+1}$ and $u^{p+1}\bar{u}^p$, and both terms are in $\mathcal{H}^{s}$. Hence, $\nabla_{\zeta}f\in Y_{s}$ and
		\begin{eqnarray} \sup_{(\theta,r,\zeta)\in\mathcal{O}^{s}(\sigma,\mu)}\|\nabla_{\zeta}f\|_{s}\leq C\varepsilon.
		\end{eqnarray}
		To estimate the Hessian of $f$, we have
		\begin{eqnarray*}
			\begin{aligned}
				\frac{\partial^2 f}{\partial p_{i,\iota_1}\partial p_{j,\iota_2}}=\frac{\varepsilon}{2}\int_{\mathbb{R}^d}(pu^{p-1}\bar{u}^{p+1}+2(p+1)|u|^{2p}+pu^{p+1}\bar{u}^{p-1})\Psi_{i,\iota_1}(x)\Psi_{j,\iota_2}(x)dx,\\
				\frac{\partial^2 f}{\partial p_{i,\iota_1}\partial q_{j,\iota_2}}=\frac{\partial^2 f}{\partial q_{i,\iota_1}\partial p_{j,\iota_2}}=\frac{{\mi}\varepsilon p}{2}\int_{\mathbb{R}^d}(u^{p-1}\bar{u}^{p+1}-u^{p+1}\bar{u}^{p-1})\Psi_{i,\iota_1}(x)\Psi_{j,\iota_2}(x)dx,\\
				\frac{\partial^2 f}{\partial q_{i,\iota_1}\partial q_{j,\iota_2}}=\frac{-\varepsilon}{2}\int_{\mathbb{R}^d}(pu^{p-1}\bar{u}^{p+1}-2(p+1)|u|^{2p}+pu^{p+1}\bar{u}^{p-1})\Psi_{i,\iota_1}(x)\Psi_{j,\iota_2}(x)dx.
			\end{aligned}
		\end{eqnarray*}
		Since $u^{p-1}\bar{u}^{p+1},|u|^{2p},u^{p+1}\bar{u}^{p-1}\in\mathcal{H}^{s}$, by Lemma \ref{20247212055}, there exists $\beta\in(0,\frac{1}{8})$ such that
		\begin{eqnarray*}
		\Big\|\frac{\partial^2f}{\partial p_i\partial p_j}\Big\|,~\Big\|\frac{\partial^2f}{\partial p_i\partial q_j}\Big\|,~\Big\|\frac{\partial^2f}{\partial q_i\partial q_j}\Big\|\leq \frac{C\varepsilon}{(ij)^{\frac{\beta}{2}}w^{s}(i,j)}\|u\|_{\mathcal{H}^{s}}^{2p}\leq \frac{C\varepsilon}{(i\wedge j)^{\beta}w^{s}(i,j)}\|u\|_{\mathcal{H}^{s}}^{2p}.
		\end{eqnarray*}
		Thus, $\nabla_{\zeta}^2f\in\mathcal{M}_{{s},\beta}$ and
		\begin{eqnarray}\label{2024721704}
			|\nabla_{\zeta}^2f|_{{s},\beta}\leq C\|u\|_{\mathcal{H}^{s}}^{2p}\varepsilon\leq C\varepsilon.
		\end{eqnarray}
Notice that $f$ does not depend on $\rho$. Therefore, by \eqref{2024721703}-\eqref{2024721704}, we have $f\in\mathcal{T}_{\sigma,\mu,\mathcal{D}}^{{s},\beta}$ and
		\begin{eqnarray}
			[f]_{\sigma,\mu,\mathcal{D}}^{s,\beta}\leq C\varepsilon.
		\end{eqnarray}

		Now applying Theorem \ref{2024371454}, we get the following result.
		\begin{thm}\label{20240726-6}
Consider the equation \eqref{2024341418} parameterized by Hermitian multiplier $M_{\rho}$ in \eqref{20240723-1}. There exists $\varepsilon_*>0$ depending on $n$, $s$, $d$, $p$, $\mathcal{A}$, $\{I_b\}_{1\leq b\leq n}$ such that if $0<\varepsilon<\varepsilon_*$, we obtain a subset $\mathcal{D}_*\subset\mathcal{D}$ with
			\begin{eqnarray}
\text{meas}(\mathcal{D}\backslash \mathcal{D}_*)\leq \varepsilon^{\frac{1}{\alpha}},~~~\alpha=32(1+\frac{d-1}{\beta})(1+\frac{4d-2}{\beta})
			\end{eqnarray}
			and for every $\rho\in\mathcal{D}_*$, the equation has a smooth quasi-periodic solution of the form
			\begin{eqnarray}	u(t,x)=\sum_{(j,\iota)\in\mathfrak{L}_0}\tilde{\xi}_{j,\iota}(t)\Psi_{j,\iota}(x),
			\end{eqnarray}
			where $\{\tilde{\xi}_{j,\iota}\}_{(j,\iota)\in\mathfrak{L}_0}$ are quasi-periodic functions with frequencies $\omega_*$. Moreover,
			\begin{eqnarray}	
|\omega_*-\omega|=O(\varepsilon)
			\end{eqnarray}
and
			\begin{eqnarray}	
\sum_{(j,\iota)\in\mathfrak{L}}\tilde{\xi}_{j,\iota}^2(t)j^{2s }=O(\varepsilon).
			\end{eqnarray}
		\end{thm}		

\begin{Remark}
Actually, the precise expression of $\beta$ can be given. For example, if $d\geq3$ and $s>d$ is an integer, then $\beta=\frac{1}{2(d+3)}$, seeing the proof of Lemma 3.2 in \cite{Gre3} for details.
\end{Remark}

\begin{Remark}
We also mention that for equation \eqref{2024341418}, Gr\'{e}bert, Imekraz and Paturelin have proved long time stability of solutions by Birkhoff normal form in \cite{Gre1}.
\end{Remark}

\section{Estimates on matrix norm}
In view of the infinite matrix space $\mathcal{M}_{s,\beta}$ defined by the norm $|\cdot|_{s,\beta}$ in \eqref{20240511-1}, we also introduce its subspace $\mathcal{M}_{s,\beta}^+$ of infinite matrices $A$ with the norm
\begin{eqnarray}\label{20240519-1}		|A|_{s,\beta+}:=\sup_{i,j\geq1}(i{\wedge}j)^{\beta}w(i,j)^s(1+|i-j|)\big\|A_i^j\big\|<+\infty.
\end{eqnarray}
On these two matrix norms, we will prove several basic estimates, which are crucial for establishing our KAM theorem.

\begin{lem}\label{20243141530-1}
		Let $A\in \mathcal{M}_{s,\beta}$ and $B\in \mathcal{M}_{s,\beta}^+$ with $0<\beta\leq\min\{\frac{s}{2},1\}$. Then both $AB$ and $BA$ belong to $\mathcal{M}_{s,\beta}$ with the following estimates
		\begin{eqnarray}\label{2024571512}
			|AB|_{s,\beta},~|BA|_{s,\beta}\leq \frac{2^{\frac{s}{2}+2}}{\beta}|A|_{s,\beta}|B|_{s,\beta+}.
		\end{eqnarray}
	\end{lem}

	\begin{proof}			
	The proof of $BA$ is similar to the proof of $AB$. Thus, we only prove the inequality of $AB$ in (\ref{2024571512}). In view of the definitions of matrix norms in \eqref{20240511-1} and \eqref{20240519-1}, it is sufficient to verify
		\begin{eqnarray}\label{2024571606}
		I:=\sum_{k\geq1} \frac{(i\wedge j)^{\beta}}{(i\wedge k)^{\beta}(k\wedge j)^{\beta}(1+|k-j|)}\Big(\frac{w(i,j)}{w(i,k)w(k,j)}\Big)^{{s}}\leq\frac{2^{\frac{s}{2}+2}}{\beta}
		\end{eqnarray}
for any $i,j\geq1$. Now fix $i,j\geq1$ with $i\leq j$, and decompose the sum in \eqref{2024571606} into three parts: $I_1=\sum_{1\leq k\leq \frac{i}{2}}$, $I_2=\sum_{\frac{i}{2}< k< 2j}$ and $I_3=\sum_{k\geq 2j}$.

For the first sum, we have $i\wedge k=k$, $k\wedge j=k$ and
		\begin{eqnarray*}\label{2024561955}
			w(i,k)= 1+\frac{i-k}{\sqrt{k}}\geq
1+\frac{{i}/{2}}{\sqrt{{i}/{2}}}>\sqrt{{i}/{2}}.
		\end{eqnarray*}
Moreover, we have $w(i,j)\leq w(k,j)$ by Lemma \ref{20243122043}. Thus,
		\begin{eqnarray}\label{2024562049}
				I_1<\sum_{1\leq k\leq\frac{i}{2}}\frac{i^{\beta}}{k^{2\beta}(1+|k-j|)}\big(\frac{2}{i}\big)^{\frac{s}{2}}\leq 2^{\frac{s}{2}}\sum_{1\leq k\leq\frac{i}{2}}\frac{1}{k^{2\beta}(1+|k-j|)},
		\end{eqnarray}
where the assumption $\beta\leq\frac{s}{2}$ is used in the last inequality.

For the second sum, we have $ (i\wedge k)(k\wedge j)> ik/2$. Moreover, we have $w(i,j)\leq w(i,k)w(k,j)$ by Lemma \ref{20243122043}. Thus,
		\begin{eqnarray}\label{20245620491}
			I_2<\sum_{\frac{i}{2}< k< 2j}\frac{i^{\beta}}{(ik/2)^{\beta}(1+|k-j|)}= 2^{\beta}\sum_{\frac{i}{2}< k< 2j}\frac{1}{k^{\beta}(1+|k-j|)}.
		\end{eqnarray}

For the last sum, we have $i\wedge k=i$, $k\wedge j=j$ and
		\begin{eqnarray*}\label{2024562040}
			w(k,j)=1+\frac{k-j}{\sqrt{j}}\geq 1+\frac{k/2}{\sqrt{k/2}}>\sqrt{k/2}.
		\end{eqnarray*}
Moreover, we have $w(i,j)\leq w(i,k)$ by Lemma \ref{20243122043}. Thus,
		\begin{eqnarray}\label{20245620492}
			I_3<\sum_{k\geq
2j}\frac{1}{1+|k-j|}\big(\frac{2}{k}\big)^{\frac{s}{2}}= 2^{\frac{s}{2}}\sum_{k\geq 2j}\frac{1}{k^{\frac{s}{2}}(1+|k-j|)}.
		\end{eqnarray}

Combining (\ref{2024562049})-(\ref{20245620492}), by $0<\beta\leq\min\{\frac{s}{2},1\}$ and Lemma \ref{20243122017}, we get
\begin{eqnarray}\label{20240520-2}
			I<2^{\frac{s}{2}}\sum_{k\geq 1}\frac{1}{k^{\beta}(1+|k-j|)}<2^{\frac{s}{2}+1}\big(1+\frac{1}{\beta}\big)\leq\frac{2^{\frac{s}{2}+2}}{\beta}.
		\end{eqnarray}
This completes the proof of (\ref{2024571606}) with $i\leq j$. We can prove (\ref{2024571606}) with $i>j$ in the same way.
	\end{proof}

\begin{lem}\label{20243141530-2}
		Let $A,B\in \mathcal{M}_{s,\beta}^+$ with $0<\beta\leq\min\{\frac{s}{2},1\}$. Then both $AB$ and $BA$ belong to $\mathcal{M}_{s,\beta}^+$ with the following estimates
		\begin{eqnarray}\label{20240521-1}
			|AB|_{s,\beta+},~|BA|_{s,\beta+}\leq \frac{2^{\frac{s}{2}+3}}{\beta}|A|_{s,\beta+}|B|_{s,\beta+}.
		\end{eqnarray}
	\end{lem}

	\begin{proof}			
	By symmetry, we only prove the inequality of $AB$ in (\ref{20240521-1}). In view of the definition of matrix norm in \eqref{20240519-1}, it is sufficient to verify
		\begin{eqnarray}\label{20240521-2}
		I:=\sum_{k\geq1} \frac{(i\wedge j)^{\beta}(1+|i-j|)}{(i\wedge k)^{\beta}(k\wedge j)^{\beta}(1+|i-k|)(1+|k-j|)}\Big(\frac{w(i,j)}{w(i,k)w(k,j)}\Big)^{{s}}\leq\frac{2^{\frac{s}{2}+3}}{\beta}
		\end{eqnarray}
for any $i,j\geq1$. By $1+|i-j|<1+|i-k|+1+|k-j|$, we have
        \begin{eqnarray}\label{20240521-3}
		\frac{1+|i-j|}{(1+|i-k|)(1+|k-j|)}<\frac{1}{1+|k-j|}+\frac{1}{1+|i-k|}.
		\end{eqnarray}
Thus, the proof of \eqref{20240521-2} is translated to the proof of Lemma \ref{20243141530-1}.
\end{proof}

\begin{lem}\label{20243141530-3}
		Let $A\in \mathcal{M}_{s,\beta}^+$ with $s\geq0$ and $0\leq\beta'<\beta\leq\frac{1}{2}$. Then $A\in \mathcal{L}(Y_s,Y_{s+\beta'})$ with the following estimate
		\begin{eqnarray}\label{2024581524}
			\|A\zeta\|_{s+\beta'}\leq \frac{2^{s+4}}{\beta-\beta'}|A|_{s,\beta+}\|\zeta\|_{s}.
		\end{eqnarray}
	\end{lem}

	\begin{proof}
	Write the block wise representation of $A\in\mathcal{M}_{s,\beta}^+$ and $\zeta\in Y_s$, i.e.,  $A=(A_i^j)_{i,j\in\mathbb{N}^+}$ and $\zeta=(\zeta_j)_{j\in\mathbb{N}^+}$, where $A_i^j\in\mathrm{M}_{2d_i\times 2d_j}(\mathbb{C})$ and $\zeta_j\in\mathbb{C}^{2d_j}$. Then we have
			\begin{eqnarray}
				\begin{aligned}					
\|A\zeta\|_{s+\beta'}^2&=\sum_{i\in\mathbb{N}^+}i^{2(s+\beta')}\Big|\sum_{ j\in\mathbb{N}^+}A_{i}^{j}\zeta_{j}\Big|^2\\
&\leq \sum_{i\in\mathbb{N}^+}i^{2(s+\beta')}\Big(\sum_{ j\in\mathbb{N}^+}\big\|A_{i}^{j}\big\||\zeta_{j}|\Big)^2\\
&\leq |A|_{s,\beta+}^2\sum_{i\in\mathbb{N}^+}i^{2(s+\beta')}\Big(\sum_{ j\in\mathbb{N}^+}(i\wedge j)^{-\beta}(1+|i-j|)^{-1}w(i,j)^{-s}|\zeta_{j}|\Big)^2\\
&\leq |A|_{s,\beta+}^2\|B\|_{\ell^2\rightarrow\ell^2}^2\|\zeta\|_s^2,
				\end{aligned}
			\end{eqnarray}	
where
		\begin{eqnarray}\label{2024582222}			B=\big(b_i^j\big)_{i,j\in\mathbb{N}^+}:=\Big(i^{s+\beta'}j^{-s}(i\wedge j)^{-\beta}(1+|i-j|)^{-1}w(i,j)^{-s}\Big)_{i,j\in\mathbb{N}^+}.
		\end{eqnarray}
Thus, it is sufficient to verify
\begin{eqnarray}\label{20240521-4}
			\|B\|_{\ell^2\rightarrow\ell^2}\leq \frac{2^{s+4}}{\beta-\beta'}.
		\end{eqnarray}

	To estimate the $\ell^2$-operator norm of $B$, we decompose $B=B_1+B_2+B_3$ with
\begin{eqnarray*}
B_1=\big(b_i^j\chi_{i\leq\frac{j}{2}}\big)_{i,j\in\mathbb{N}^+},~~~B_2=\big(b_i^j\chi_{\frac{j}{2}<i<2j}\big)_{i,j\in\mathbb{N}^+},~~~B_3=\big(b_i^j\chi_{i\geq2j}\big)_{i,j\in\mathbb{N}^+},
\end{eqnarray*}
where $\chi_{i\leq\frac{j}{2}}=1$ if $i\leq\frac{j}{2}$, and $\chi_{i\leq\frac{j}{2}}=0$ otherwise; it is similar for $\chi_{\frac{j}{2}<i<2j}$ and $\chi_{i\geq2j}$. In the following, we estimate $b_{i}^j$ in three cases respectively.

\textbf{Case 1: $i\leq\frac{j}{2}$}. By $i\wedge j=i$, $1+|i-j|>\frac{j}{2}$ and $w(i,j)>1$, we have
\begin{eqnarray}\label{202451022411}
			b_{i}^j<\frac{2}{i^{-s+\beta-\beta'}j^{s+1}}< \frac{2}{i^{\frac{1}{2}(1+\beta-\beta')}j^{\frac{1}{2}(1+\beta-\beta')}},
\end{eqnarray}
where the facts $j>i$ and $s+\frac{1}{2}(1-\beta+\beta')>0$ are used in the last inequality. Then by \eqref{202451022411} and Lemma \ref{20240527-2}, we get
\begin{eqnarray}\label{20240527-6}
\|B_1\|_{\ell^2\rightarrow\ell^2}\leq\Big(\sum_{i,j\in\mathbb{N}^+}\frac{4}{i^{1+\beta-\beta'}j^{1+\beta-\beta'}}\Big)^{\frac{1}{2}}<2+\frac{2}{\beta-\beta'},
\end{eqnarray}
where the last inequality follows from \eqref{20240527-1} in Lemma \ref{20243122017}.

\textbf{Case 2: $\frac{j}{2}<i<2j$}. By  $i\wedge j>\frac{i}{2}$, $j>\frac{i}{2}$ and $w(i,j)\geq1$, we have
\begin{eqnarray}\label{20245102241}
			b_{i}^j < \frac{2^{s+\beta}}{i^{\beta-\beta'}(1+|i-j|)},
\end{eqnarray}
and thus
\begin{eqnarray}\label{20240527-7}
\|B_2\|_{\ell^1\rightarrow\ell^1}\leq\sup_{j\in\mathbb{N}^+}\sum_{i\in\mathbb{N}^+}\frac{2^{s+\beta}}{i^{\beta-\beta'}(1+|i-j|)}<2^{s+\beta+1}\big(1+\frac{1}{\beta-\beta'}\big),
\end{eqnarray}
where the last inequality follows from \eqref{20240520-1} in Lemma \ref{20243122017}.
On the other hand, by  $i\wedge j>\frac{j}{2}$, $i<2j$ and $w(i,j)\geq1$, we have
\begin{eqnarray}\label{20240526-1}
			b_{i}^j < \frac{2^{s+\beta+\beta'}}{j^{\beta-\beta'}(1+|i-j|)},
\end{eqnarray}
and thus
\begin{eqnarray}\label{20240527-8}
\|B_2\|_{\ell^\infty\rightarrow\ell^\infty}\leq\sup_{i\in\mathbb{N}^+}\sum_{j\in\mathbb{N}^+}\frac{2^{s+\beta+\beta'}}{j^{\beta-\beta'}(1+|i-j|)}<2^{s+\beta+\beta'+1}\big(1+\frac{1}{\beta-\beta'}\big),
\end{eqnarray}
where the last inequality follows from \eqref{20240520-1} in Lemma \ref{20243122017}. By \eqref{20240527-7}, \eqref{20240527-8} and Lemma \ref{20240527-4}, we get
\begin{eqnarray}\label{20240527-9}
\|B_2\|_{\ell^2\rightarrow\ell^2}<2^{s+\beta+\frac{\beta'}{2}+1}\big(1+\frac{1}{\beta-\beta'}\big).
\end{eqnarray}

\textbf{Case 3: $i\geq2j$}. By  $i\wedge j=j$, $1+|i-j|>\frac{i}{2}$  and $w(i,j)>\frac{i}{2\sqrt{j}}$, we have
		\begin{eqnarray}\label{20245102240}
			b_{i}^j<\frac{i^{s+\beta'} }{j^{s+\beta}}\frac{2}{i}\Big(\frac{2\sqrt{j}}{i}\Big)^s=\frac{2^{s+1}}{i^{1-\beta'}j^{\frac{s}{2}+\beta}}\leq\frac{2^{s+1}}{i^{1-\beta'}j^{\beta}}<\frac{2^{s+1}}{i^{\frac{1}{2}(1+\beta-\beta')}j^{\frac{1}{2}(1+\beta-\beta')}},
		\end{eqnarray}
where the facts $i>j$ and $\frac{1}{2}(1-\beta-\beta')>0$ are used in the last inequality. Then by \eqref{20245102240} and Lemma \ref{20240527-2}, we get
\begin{eqnarray}\label{20240527-10}
\|B_3\|_{\ell^2\rightarrow\ell^2}\leq\Big(\sum_{i,j\in\mathbb{N}^+}\frac{2^{2s+2}}{i^{1+\beta-\beta'}j^{1+\beta-\beta'}}\Big)^{\frac{1}{2}}<2^{s+1}\big(1+\frac{1}{\beta-\beta'}\big),
\end{eqnarray}
where the last inequality follows from \eqref{20240527-1} in Lemma \ref{20243122017}.

Finally, summing the estimates \eqref{20240527-6}, \eqref{20240527-9} and \eqref{20240527-10}, we get
\begin{eqnarray}\label{20240527-11}
\|B\|_{\ell^2\rightarrow\ell^2}<\big(2+2^{s+\beta+\frac{\beta'}{2}+1}+2^{s+1}\big)\big(1+\frac{1}{\beta-\beta'}\big)<2^{s+3}\frac{2}{\beta-\beta'},
\end{eqnarray}
which is \eqref{20240521-4}. This completes the proof of the lemma.
	\end{proof}

	\begin{lem}\label{20243141530-4}
Let $\eta\in Y_{s+\beta}$ and $\zeta\in Y_{s}$ with $s\geq0$ and $\beta\geq0$. Denoting $A:=\eta\otimes \zeta$, then we have $A\in\mathcal{M}_{s,\beta}$ with the following estimate
\begin{eqnarray}\label{2024571512}
	|A|_{s,\beta}\leq \|\eta\|_{s+\beta}\|\zeta\|_{s}.
\end{eqnarray}
	\end{lem}
	\begin{proof}
	For $i,j\in\mathbb{N}^+$, we have
	\begin{eqnarray}
		\|A_i^j\|=|\eta_i||\zeta_j|\leq
i^{-(s+\beta)}j^{-s}\|\eta\|_{s+\beta}\|\zeta\|_{s}.
	\end{eqnarray}
Thus, it is sufficient to verify
\begin{eqnarray}
(i\wedge j)^{\beta}w(i,j)^{s}i^{-(s+\beta)}j^{-s}\leq1,
\end{eqnarray}
which follows from $i\wedge j\leq i$ and $w(i,j)\leq ij$. This completes the proof of the lemma.
	\end{proof}

\section{Poisson bracket and Hamiltonian flow}
In this section, we shall study Poisson brackets and Hamiltonian flows. Before that, we define the truncation functions: for function $f\in \mathcal{T}_{\sigma,\mu,\mathcal{D}}^{s,\beta}$, define its truncation functions $f^T$, which are the second order Taylor approximation of $f$ at $r=0$ and $\zeta=0$. That is,
	\begin{eqnarray}\label{2024531000}
		\begin{aligned}
			f^T&:=f_{\theta}+\langle f_{r},r\rangle+\langle f_{\zeta},\zeta\rangle+\frac{1}{2}\langle \zeta,f_{\zeta\zeta}\zeta\rangle\\&=f(\theta,0;\rho)+\langle \nabla_r f(\theta,0;\rho),r\rangle+\langle \nabla_{\zeta}f(\theta,0;\rho),\zeta\rangle+\frac{1}{2}\langle\zeta,\nabla_{\zeta}^2f(\theta,0;\rho)\zeta\rangle.
		\end{aligned}
	\end{eqnarray}	
	According to the definition of the norm $[f]_{\sigma,\mu,\mathcal{D}}^{s,\beta}$ in (\ref{20240605-1}), for $\theta\in \mathbb{T}_{\sigma}^n$ and $\rho\in \mathcal{D}$, we have
	\begin{eqnarray}\label{20243141517-1}
			\sup_{{\nu}=0,1}|\partial_{\rho}^{\nu}f_{\theta}(\theta;\rho)|\leq [f]_{\sigma,\mu,\mathcal{D}}^{s,\beta},
	\end{eqnarray}
	\begin{eqnarray}\label{20243141517-2}
			\sup_{{\nu}=0,1}\|\partial_{\rho}^{\nu}f_{\zeta}(\theta;\rho)\|_s\leq \mu^{-1}[f]_{\sigma,\mu,\mathcal{D}}^{s,\beta},
	\end{eqnarray}
	\begin{eqnarray}\label{20243141517-3} \sup_{{\nu}=0,1}|\partial_{\rho}^{\nu}f_{\zeta\zeta}(\theta;\rho)|_{s,\beta}\leq \mu^{-2}[f]_{\sigma,\mu,\mathcal{D}}^{s,\beta}.
	\end{eqnarray}
	Furthermore, using the Cauchy estimates, we obtain
	\begin{eqnarray}\label{202431415171}
		\sup_{{\nu}=0,1}|\partial_{\rho}^{\nu}f_{r}(\theta;\rho)|\leq \mu^{-2}[f]_{\sigma,\mu,\mathcal{D}}^{s,\beta},
	\end{eqnarray}
	and $(\partial_{\rho}^{\nu}f_{\zeta\zeta})_{{\nu}=0,1}$ are bounded linear operators from $Y_s$ to itself, satisfying
	\begin{eqnarray}\label{2024314151711}
		\begin{aligned}				\sup_{{\nu}=0,1}\|\partial_{\rho}^{\nu}f_{\zeta\zeta}(\theta;\rho)\|_{\mathcal{L}(Y_s,Y_s)}&=\sup_{{\nu}=0,1}\|\partial_{\rho}^{\nu}\nabla_{\zeta}^2f(\theta,0;\rho)\|_{\mathcal{L}(Y_s,Y_s)}\\&\leq \mu^{-1}\sup_{{\nu}=0,1\atop(\theta,r,\zeta)\in\mathcal{O}^s(\sigma,\mu)}\|\partial_{\rho}^{\nu}\nabla_{\zeta}f(\theta,r,\zeta;\rho)\|_{s}\\&\leq \mu^{-2}[f]_{\sigma,\mu,\mathcal{D}}^{s,\beta}.
		\end{aligned}
	\end{eqnarray}

\subsection{Poisson bracket}
With the symplectic structure $dr\wedge d\theta+dp\wedge dq$, the Poisson brackets of $f$ and $g$ are defined by
	\begin{eqnarray}\label{20243191701}		\{f,g\}=-\nabla_rf\cdot\nabla_{\theta}g+\nabla_{\theta}f\cdot\nabla_{r}g+\langle \nabla_{\zeta}f,J\nabla_{\zeta}g\rangle,
	\end{eqnarray}		
	where $\zeta=(p,q)$ and
    $J=	\begin{pmatrix}
			0 & -1\\ 1 & 0
		\end{pmatrix}$.
Since the space $\mathcal{T}_{\sigma,\mu,\mathcal{D}}^{s,\beta}$ is not closed under Poisson brackets, we introduce the subspace $\mathcal{T}_{\sigma,\mu,\mathcal{D}}^{s,\beta+}\subset \mathcal{T}_{\sigma,\mu,\mathcal{D}}^{s,\beta}$  defined by
	\begin{eqnarray*}
		\mathcal{T}_{\sigma,\mu,\mathcal{D}}^{s,\beta+}:=\big\{g\in \mathcal{T}_{\sigma,\mu,\mathcal{D}}^{s,\beta}~\big|~\partial_{\rho}^{\nu}\nabla_{\zeta}g(\theta,0;\rho)\in Y_{s+\beta},~\partial_{\rho}^{\nu}\nabla_{\zeta}^2g(\theta,0;\rho)\in \mathcal{M}_{s,\beta}^+,~{\nu}=0,1\big\}
	\end{eqnarray*}
	with the norm
	\begin{eqnarray}\label{20243192006}		[g]_{\sigma,\mu,\mathcal{D}}^{s,\beta+}=[g]_{\sigma,\mu,\mathcal{D}}^{s,\beta}+\sup_{{\nu}=0,1\atop \theta\in \mathbb{T}_{\sigma}^n, \rho\in\mathcal{D}}\Big(\mu\|\partial_{\rho}^{\nu}\nabla_{\zeta}g(\theta,0;\rho)\|_{s+\beta}+\mu^2|\partial_{\rho}^{\nu}\nabla_{\zeta}^2g(\theta,0;\rho)|_{s,\beta+}\Big).
	\end{eqnarray}
Based on above, we shall give the estimates of Poisson brackets in the following.

	\begin{lem}\label{20243212055}
		Assume $f=f^T\in \mathcal{T}_{\sigma,\mu,\mathcal{D}}^{s,\beta}$ and $g=g^T\in \mathcal{T}_{\sigma,\mu,\mathcal{D}}^{s,\beta+}$ with $0<\beta\leq \min\{\frac{s}{2},1\}$. Then for $0<\sigma'<\sigma$, we have $\{f,g\}\in \mathcal{T}_{\sigma',\mu,\mathcal{D}}^{s,\beta}$ and
		\begin{eqnarray}\label{20243141513}
			[\{f,g\}]_{\sigma',\mu,\mathcal{D}}^{s,\beta}\leq  \frac{C}{\beta(\sigma-\sigma')\mu^2}[f]_{\sigma,\mu,\mathcal{D}}^{s,\beta}[g]_{\sigma,\mu,\mathcal{D}}^{s,\beta+},
		\end{eqnarray}
where $C>0$ is a constant only depending on $n$ and $s$.
	\end{lem}

	\begin{proof}
According to \eqref{2024531000}, write
$$f=f_{\theta}+\langle f_{r},r\rangle+\langle f_{\zeta},\zeta\rangle+\frac{1}{2}\langle \zeta,f_{\zeta\zeta}\zeta\rangle,$$
$$g=g_{\theta}+\langle g_{r},r\rangle+\langle g_{\zeta},\zeta\rangle+\frac{1}{2}\langle \zeta,g_{\zeta\zeta}\zeta\rangle,$$
where $f_{\theta}$, $f_r$, $f_{\zeta}$, $f_{\zeta\zeta}$, $g_{\theta}$, $g_r$, $g_{\zeta}$, $g_{\zeta\zeta}$ are independent of $r,\zeta$. Then in view of \eqref{20243191701}, by direct calculation we have $\nabla_{r}f=f_r$, $\nabla_{r}g=g_r$, $\nabla_{\zeta}f=f_{\zeta}+f_{\zeta\zeta}\zeta$, $\nabla_{\zeta}g=g_{\zeta}+g_{\zeta\zeta}\zeta$ and thus
\begin{eqnarray}\label{20240612-1}
\{f,g\}=-\langle f_{r},\nabla_{\theta}g\rangle+\langle \nabla_{\theta}f,g_{r}\rangle+\langle f_{\zeta},Jg_{\zeta}\rangle-\langle {g_{\zeta\zeta}}Jf_{\zeta},\zeta\rangle+\langle \zeta,{f_{\zeta\zeta}}Jg_{\zeta}\rangle+\langle \zeta,{f_{\zeta\zeta}}Jg_{\zeta\zeta}\zeta\rangle.
\end{eqnarray}
For simplicity, let $h_1,h_2,h_3,h_4,h_5,h_6$ be the right side of (\ref{20240612-1}) successively.

Firstly, we consider $h_1$ and $h_2$. By (\ref{202431415171}) and Cauchy estimates, we get
\begin{eqnarray}\label{20240612-2}			[h_1]_{\sigma',\mu,\mathcal{D}}^{s,\beta},\,\,\, [h_2]_{\sigma',\mu,\mathcal{D}}^{s,\beta}\leq \frac{C}{(\sigma-\sigma')\mu^2}[f]_{\sigma,\mu,\mathcal{D}}^{s,\beta}[g]_{\sigma,\mu,\mathcal{D}}^{s,\beta}.
\end{eqnarray}

Secondly, we consider $h_3:=\langle f_{\zeta},Jg_{\zeta}\rangle$. Notice that $\nabla_{\zeta}h_3=\nabla_{\zeta}^2h_3=0$. Thus, by \eqref{20243141517-2}, we get
\begin{eqnarray}\label{20240612-3}
[h_3]_{\sigma,\mu,\mathcal{D}}^{s,\beta}=\sup_{{\nu}=0,1,\rho\in \mathcal{D}\atop (\theta,r,\zeta)\in\mathcal{O}^s(\sigma,\mu)}|\partial_{\rho}^{\nu}h_3|\leq \frac{C}{\mu^2}[f]_{\sigma,\mu,\mathcal{D}}^{s,\beta}[g]_{\sigma,\mu,\mathcal{D}}^{s,\beta}.
\end{eqnarray}

Thirdly, we consider $h_4:=-\langle {g_{\zeta\zeta}}Jf_{\zeta},\zeta\rangle$ and $h_5:=\langle \zeta,{f_{\zeta\zeta}}Jg_{\zeta}\rangle$. Notice that $\nabla_{\zeta}h_4=-{g_{\zeta\zeta}}Jf_{\zeta}$, $\nabla_{\zeta}h_5={f_{\zeta\zeta}}Jg_{\zeta}$ and $\nabla_{\zeta}^2h_4=\nabla_{\zeta}^2h_5=0$. Thus, by \eqref{20243141517-2}, \eqref{2024314151711} and $\|\zeta\|_s\leq\mu$, we get
\begin{eqnarray}\label{20240612-4}
[h_4]_{\sigma,\mu,\mathcal{D}}^{s,\beta}=\sup_{{\nu}=0,1,\rho\in \mathcal{D}\atop (\theta,r,\zeta)\in\mathcal{O}^s(\sigma,\mu)}\max\big\{|\partial_{\rho}^{\nu}h_4|,\mu\|\partial_{\rho}^{\nu}\nabla_{\zeta}h_4\|_s\big\}
\leq\frac{C}{\mu^2}[f]_{\sigma,\mu,\mathcal{D}}^{s,\beta}[g]_{\sigma,\mu,\mathcal{D}}^{s,\beta},
\end{eqnarray}
\begin{eqnarray}\label{20240612-5}
[h_5]_{\sigma,\mu,\mathcal{D}}^{s,\beta}=\sup_{{\nu}=0,1,\rho\in \mathcal{D}\atop (\theta,r,\zeta)\in\mathcal{O}^s(\sigma,\mu)}\max\big\{|\partial_{\rho}^{\nu}h_5|,\mu\|\partial_{\rho}^{\nu}\nabla_{\zeta}h_5\|_s\big\}
\leq\frac{C}{\mu^2}[f]_{\sigma,\mu,\mathcal{D}}^{s,\beta}[g]_{\sigma,\mu,\mathcal{D}}^{s,\beta}.
\end{eqnarray}

Finally, we consider $h_6:=\langle \zeta,{f_{\zeta\zeta}}Jg_{\zeta\zeta}\zeta\rangle$. Notice that $$\nabla_{\zeta}h_6={f_{\zeta\zeta}}Jg_{\zeta\zeta}\zeta-{g_{\zeta\zeta}}Jf_{\zeta\zeta}\zeta,$$ $$\nabla_{\zeta}^2h_6={f_{\zeta\zeta}}Jg_{\zeta\zeta}-{g_{\zeta\zeta}}Jf_{\zeta\zeta}.$$ %
By \eqref{2024314151711} and $\|\zeta\|_s\leq\mu$, we get
\begin{eqnarray}\label{20240612-6}
\sup_{{\nu}=0,1}|\partial_{\rho}^{\nu}h_6|\leq\frac{C}{\mu^2}[f]_{\sigma,\mu,\mathcal{D}}^{s,\beta}[g]_{\sigma,\mu,\mathcal{D}}^{s,\beta},
\end{eqnarray}
\begin{eqnarray}\label{20240612-7}
\sup_{{\nu}=0,1}\|\partial_{\rho}^{\nu}\nabla_{\zeta}h_6\|_s\leq\frac{C}{\mu^3}[f]_{\sigma,\mu,\mathcal{D}}^{s,\beta}[g]_{\sigma,\mu,\mathcal{D}}^{s,\beta};
\end{eqnarray}
by Lemma \ref{20243141530-1}, we get
\begin{eqnarray}\label{20240612-8}	
|\nabla_{\zeta}^2h_6|_{s,\beta}\leq \frac{C}{\beta}|f_{\zeta\zeta}|_{s,\beta}|g_{\zeta\zeta}|_{s,\beta+}\leq
\frac{C}{\beta\mu^4}[f]_{\sigma,\mu,\mathcal{D}}^{s,\beta}[g]_{\sigma,\mu,\mathcal{D}}^{s,\beta+},
\end{eqnarray}
\begin{eqnarray}\label{20240612-9}	|\partial_{\rho}\nabla_{\zeta}^2h_6|_{s,\beta}\leq  \frac{C}{\beta}\Big(|\partial_{\rho}f_{\zeta\zeta}|_{s,\beta}|g_{\zeta\zeta}|_{s,\beta+}+|f_{\zeta\zeta}|_{s,\beta}|\partial_{\rho}g_{\zeta\zeta}|_{s,\beta+}\Big)\leq
\frac{C}{\beta\mu^4}[f]_{\sigma,\mu,\mathcal{D}}^{s,\beta}[g]_{\sigma,\mu,\mathcal{D}}^{s,\beta+}.
\end{eqnarray}
Thus, by \eqref{20240612-6}-\eqref{20240612-9}, we get
\begin{eqnarray}\label{20240612-10}
[h_6]_{\sigma,\mu,\mathcal{D}}^{s,\beta}\leq\frac{C}{\beta\mu^2}[f]_{\sigma,\mu,\mathcal{D}}^{s,\beta}[g]_{\sigma,\mu,\mathcal{D}}^{s,\beta+}.
\end{eqnarray}

Summing the estimates \eqref{20240612-2}-\eqref{20240612-5} and \eqref{20240612-10}, we get \eqref{20243141513}.
	\end{proof}

\subsection{Hamiltonian flow}
In this subsection, we assume $g\equiv g^T$, namely,
	\begin{eqnarray}\label{2024351735}
		g=g_{\theta}(\theta;\rho)+\langle g_r(\theta;\rho), r\rangle+\langle g_{\zeta}(\theta;\rho),\zeta\rangle+\frac{1}{2}\langle \zeta,g_{\zeta\zeta}(\theta;\rho)\zeta\rangle.
	\end{eqnarray}
	Then the  Hamiltonian associated with $g$ have the following form:
	\begin{eqnarray}\label{2024351730}
		\begin{cases}			
\dot{\theta}(t)=g_r(\theta;\rho)\\
\dot{r}(t)=-\nabla_{\theta}g(\theta,r,\zeta;\rho)\\
\dot{\zeta}(t)=J(g_{\zeta}(\theta;\rho)+g_{\zeta\zeta}(\theta;\rho)\zeta).
		\end{cases}
	\end{eqnarray}
	In the following, we shall study the Hamiltonian flows $\Phi_g^t$ generated by (\ref{2024351730}).			

	\begin{lem}\label{20243181426}
Let $0<\beta'<\beta\leq\min\{\frac{s}{2},\frac{1}{2}\}$,
$0<\sigma'<\sigma\leq1$,
$0<\mu'<\mu\leq1$
and $g=g^T\in\mathcal{T}^{s,\beta+}_{\sigma,\mu,\mathcal{D}}$ satisfy
		\begin{eqnarray}\label{20243141942}
			[g]_{\sigma,\mu,\mathcal{D}}^{s,\beta+}	\leq \frac{(\beta-\beta')(\sigma-\sigma')(\mu-\mu')^2}{c_1},
		\end{eqnarray}
where $c_1>1$ is a properly large constant only depending on $n$ and $s$. Then for $0\leq t\leq 1$ the flow maps
$$\Phi_g^t: \mathcal{O}^s(\sigma',\mu')\ni (\theta^0,r^0,\zeta^0)\mapsto (\theta(t),r(t),\zeta(t))\in\mathcal{O}^s(\frac{\sigma+\sigma'}{2},\frac{\mu+\mu'}{2})$$ are of the form (here hide the parameter $\rho$):
		\begin{eqnarray}
			{\Phi}_g^t : \begin{pmatrix}
				\theta^0\\r^0\\\zeta^0
			\end{pmatrix}
			\to	\begin{pmatrix}				K(\theta^0;t)\\L(\theta^0,\zeta^0;t)+V(\theta^0;t)r^0\\T(\theta^0;t)+U(\theta^0;t)\zeta^0
			\end{pmatrix},
		\end{eqnarray}
where the mappings $K,T$ and the operators $V,U$ analytically depend on $\theta^0\in \mathbb{T}_{\sigma'}^n$, and the mapping $L$ analytically depends on $(\theta^0,\zeta^0)\in\mathbb{T}_{\sigma'}^n\times\mathcal{O}_{\mu'}(Y_s)$ with $\mathcal{O}_{\mu'}(Y_s):=\{\zeta\in Y_s : \|\zeta\|_s\leq \mu'\}$. More precisely, there hold the following estimates:

		(1)  For any $\theta^0\in\mathbb{T}_{\sigma'}^n$, we have
		\begin{eqnarray}			
\sup_{\nu=0,1}\|\partial_{\rho}^{\nu}(V(\theta^0;t)-I)\|_{\mathcal{L}(\mathbb{C}^n,\mathbb{C}^n)}&\leq&16(\sigma-\sigma')^{-1}\mu^{-2}[g]_{\sigma,\mu,\mathcal{D}}^{s,\beta},\label{20245201706}\\
\sup_{\nu=0,1}\|\partial_{\rho}^{\nu}({}^tU(\theta^0;t)-I)\|_{\mathcal{L}(Y_s,Y_{s+\beta'})}&\leq&\frac{2^{s+7}\mu^{-2}}{\beta-\beta'}[g]_{\sigma,\mu,\mathcal{D}}^{s,\beta+},\label{20245211529}\\
\sup_{\nu=0,1}\|\partial_{\rho}^{\nu}(U(\theta^0;t)-I)\|_{\mathcal{L}(Y_s,Y_{s+\beta'})}&\leq&\frac{2^{s+7}\mu^{-2}}{\beta-\beta'}[g]_{\sigma,\mu,\mathcal{D}}^{s,\beta+},\label{202431721021}\\
\sup_{\nu=0,1}|\partial_{\rho}^{\nu}(U(\theta^0;t)-I)|_{s,\beta+}&\leq&8\mu^{-2}[g]_{\sigma,\mu,\mathcal{D}}^{s,\beta+};\label{20245201521}
		\end{eqnarray}
for any $(\theta^0,\zeta^0)\in\mathbb{T}_{\sigma'}^n\times\mathcal{O}_{\mu'}(Y_s)$ and any component $L^b$ of $L$, $b=1,2,\cdots,n$, we have
		\begin{eqnarray}			\sup_{{\nu}=0,1}\|\partial_{\rho}^{\nu}\nabla_{\zeta^0}L^b(\theta^0,\zeta^0;t)\|_{s+\beta'}&\leq& \frac{2^{s+9}}{\beta-\beta'}(\sigma-\sigma')^{-1}{\mu}^{-1}[g]_{\sigma,\mu,\mathcal{D}}^{s,\beta+},\label{20243172102}\\
\sup_{{\nu}=0,1}|\partial_{\rho}^{\nu}\nabla_{\zeta^0}^2L^b(\theta^0,\zeta^0;t)|_{s,\beta+}&\leq& 32(\sigma-\sigma')^{-1}\mu^{-2}[g]_{\sigma,\mu,\mathcal{D}}^{s,\beta+}.\label{2024317210211}
		\end{eqnarray}

		(2) The flow maps $\Phi_g^t$ analytically extend to mappings
		\begin{eqnarray*}
\mathbb{T}_{\sigma'}^n\times\mathbb{C}^n\times Y_s\ni (\theta^0,r^0,\zeta^0)\mapsto(\theta(t),r(t),\zeta(t))\in \mathbb{T}_{\sigma}^n\times\mathbb{C}^n\times Y_s,
		\end{eqnarray*}
		which satisfy
		\begin{eqnarray}\label{20243172034}
			\begin{aligned}
&|\theta(t)-\theta^0|\leq\mu^{-2}[g]_{\sigma,\mu,\mathcal{D}}^{s,\beta},\\
&|r(t)-r^0|\leq 8({\sigma-\sigma'})^{-1}\big(1+\mu^{-1}\|\zeta^0\|_s+\mu^{-2}|r^0|+\mu^{-2}\|\zeta^0\|_{s}^2\big)[g]_{\sigma,\mu,\mathcal{D}}^{s,\beta},\\
&\|\zeta(t)-\zeta^0\|_{s+\beta'}\leq \Big(2\mu^{-1}+\frac{2^{s+5}\mu^{-2}}{\beta-\beta'}\|\zeta^0\|_s\Big)[g]_{\sigma,\mu,\mathcal{D}}^{s,\beta+}
			\end{aligned}
		\end{eqnarray}
		and
		\begin{eqnarray}\label{20245212009}
			\begin{aligned}
&|\partial_{\rho}\theta(t)|\leq 2\mu^{-2}[g]_{\sigma,\mu,\mathcal{D}}^{s,\beta},\\
&|\partial_{\rho}r(t)|\leq 32({\sigma-\sigma'})^{-1}\big(1+\mu^{-1}\|\zeta^0\|_s+\mu^{-2}|r^0|+\mu^{-2}\|\zeta^0\|_{s}^2\big)[g]_{\sigma,\mu,\mathcal{D}}^{s,\beta},\\
&\|\partial_{\rho}\zeta(t)\|_{s+\beta'}\leq \Big(8\mu^{-1}+\frac{2^{s+7}\mu^{-2}}{\beta-\beta'}\|\zeta^0\|_s\Big)[g]_{\sigma,\mu,\mathcal{D}}^{s,\beta+}.
			\end{aligned}
		\end{eqnarray}
	\end{lem}

	\begin{proof}
Since $g_r$ is independent of $r,\zeta$, we firstly consider the equation \eqref{2024351730}$_1$ for $\theta(t)$. By \eqref{202431415171} and \eqref{20243141942}, we get
\begin{eqnarray}\label{20240616-1}
\sup_{\theta\in\mathbb{T}_{\sigma}^n,\rho\in\mathcal{D}}|g_{r}(\theta;\rho)|
\leq \mu^{-2}[g]_{\sigma,\mu,\mathcal{D}}^{s,\beta}<\frac{\sigma-\sigma'}{2},
\end{eqnarray}
which implies $\mathbb{T}_{\sigma'}^n\ni\theta^0\mapsto\theta(t)\in\mathbb{T}_{\frac{\sigma+\sigma'}{2}}^n$
and the first estimate of \eqref{20243172034} for $\theta(t)-\theta^0$. Differentiating \eqref{2024351730}$_1$ with respect to $\rho$, we get
\begin{eqnarray}\label{20240618-1}
\partial_{\rho}\dot{\theta}(t)=\partial_{\rho}g_{r}(\theta;\rho)+	\nabla_{\theta}g_{r}(\theta;\rho)\partial_{\rho}\theta(t),~~~\partial_{\rho}\theta(0)=0.
\end{eqnarray}
By \eqref{202431415171} we get
\begin{eqnarray}\label{20240618-2}
		\sup_{\theta\in\mathbb{T}_{\sigma}^n,\rho\in \mathcal{D}}|\partial_{\rho}g_{r}(\theta;\rho)|\leq \mu^{-2}[g]_{\sigma,\mu,\mathcal{D}}^{s,\beta}.
\end{eqnarray}
In view of \eqref{20240616-1}, by Cauchy estimate, we get
\begin{eqnarray}\label{20240618-3}
\sup_{\theta\in\mathbb{T}_{\frac{\sigma+\sigma'}{2}}^n,\rho\in\mathcal{D}}|\nabla_{\theta}g_{r}(\theta;\rho)|
\leq2(\sigma-\sigma')^{-1}\mu^{-2}[g]_{\sigma,\mu,\mathcal{D}}^{s,\beta}.
\end{eqnarray}
Then, by \eqref{20243141942}, \eqref{20240618-1}-\eqref{20240618-3} and Gronwall inequality, we get the first estimate of \eqref{20245212009} for $\partial_{\rho}\theta(t)$.

Next, we consider the equation \eqref{2024351730}$_3$ for $\zeta(t)$. It reads
		\begin{eqnarray}\label{20243141955}
			\dot{\zeta}(t)=a(t)+B(t)\zeta(t),~~~\zeta(0)=\zeta^0\in \mathcal{O}_{\mu'}(Y_s),
		\end{eqnarray}
		where $a(t):=Jg_{\zeta}\big(\theta(t);\rho\big)$ and $B(t):=Jg_{\zeta\zeta}\big(\theta(t);\rho\big)$ analytically depend on $\theta^0$.
Since $g\in\mathcal{T}^{s,\beta+}_{\sigma,\mu,\mathcal{D}}$, we have $a(t)\in Y_s$ and $B(t)\in\mathcal{M}_{s,\beta}^+$.  By (\ref{20243192006}), we get
		\begin{eqnarray}\label{20243171704}
			\|a(t)\|_{s+\beta}\leq \mu^{-1}[g]_{\sigma,\mu,\mathcal{D}}^{s,\beta+},
		\end{eqnarray}
and by (\ref{20243192006}), (\ref{20243141942}) and Lemma \ref{20243141530-3}, we get
		\begin{eqnarray}\label{20243142009}
			\|B(t)\|_{\mathcal{L}(Y_s,Y_{s+\beta'})}\leq \frac{2^{s+4}}{\beta-\beta'}|B(t)|_{s,\beta+}\leq  \frac{2^{s+4}}{\beta-\beta'}\mu^{-2}[g]_{\sigma,\mu,\mathcal{D}}^{s,\beta+}\leq \frac{1}{2},
		\end{eqnarray}
which implies
		\begin{eqnarray}\label{20245202044}
			\|B(t)\|_{\mathcal{L}(Y_s,Y_{s})}\leq \frac{1}{2}.
		\end{eqnarray}
According to (\ref{20243141955}), we have  $\zeta(t)=\zeta^0+\int_{0}^{t}(a(t')+B(t')\zeta(t'))dt'$. Iterating this relation, we can get
		\begin{eqnarray}\label{20243172033}
			\zeta(t)=a^{\infty}(t)+(I+B^{\infty}(t))\zeta^0,
		\end{eqnarray}
		where
		\begin{eqnarray}\label{20243171707}
			a^{\infty}(t)=\int_{0}^{t}a(t_1)dt_1+\sum_{k\geq 2}\int_{0}^{t}\int_{0}^{t_1}\cdots\int_{0}^{t_{k-1}}\prod_{j=1}^{k-1}B(t_j)a(t_k)dt_k\cdots dt_2dt_1
		\end{eqnarray}
		and
		\begin{eqnarray}\label{20243171708}
			B^{\infty}(t)=\sum_{k\geq 1}\int_{0}^{t}\int_{0}^{t_1}\cdots\int_{0}^{t_{k-1}}\prod_{j=1}^{k}B(t_j)dt_k\cdots dt_2dt_1.
		\end{eqnarray}
		For $k\geq 1$ and $0\leq  t_k\leq\cdots\leq t_1 \leq 1$, by (\ref{20243142009}) and (\ref{20245202044}),  we have
		\begin{eqnarray}\label{2024531200}
			\|B(t_1)\cdots B(t_k)\|_{\mathcal{L}(Y_s,Y_{s+\beta'})}\leq \Big(\frac{1}{2}\Big)^{k-1} \frac{2^{s+4}}{\beta-\beta'}\mu^{-2}[g]_{\sigma,\mu,\mathcal{D}}^{s,\beta+}\leq \Big(\frac{1}{2}\Big)^{k}.
		\end{eqnarray}
		Then in view of (\ref{20243141942}),(\ref{20243171704}) and (\ref{2024531200}), the series (\ref{20243171707}) and (\ref{20243171708}) uniformly converge in $t\in [0,1]$ and $\theta^0\in \mathbb{T}_{\sigma'}^n$ with
		\begin{eqnarray}\label{20245211906}	
	    \begin{aligned}
\|a^{\infty}(t)\|_{s+\beta'}&\leq\mu^{-1}[g]_{\sigma,\mu,\mathcal{D}}^{s,\beta+}+\sum_{k\geq2}\Big(\frac{1}{2}\Big)^{k-1}\mu^{-1}[g]_{\sigma,\mu,\mathcal{D}}^{s,\beta+}\\
&=2\mu^{-1}[g]_{\sigma,\mu,\mathcal{D}}^{s,\beta+}
        \end{aligned}
		\end{eqnarray}
		and
		\begin{eqnarray}\label{20245211907}
        \begin{aligned}
\|B^{\infty}(t)\|_{\mathcal{L}(Y_s,Y_{s+\beta'})}&\leq\sum_{k\geq1}\Big(\frac{1}{2}\Big)^{k-1} \frac{2^{s+4}}{\beta-\beta'}\mu^{-2}[g]_{\sigma,\mu,\mathcal{D}}^{s,\beta+}\\
&=\frac{2^{s+5}}{\beta-\beta'}{\mu^{-2}}[g]_{\sigma,\mu,\mathcal{D}}^{s,\beta+}.
        \end{aligned}
		\end{eqnarray}
Putting (\ref{20245211906}) and (\ref{20245211907}) into (\ref{20243172033}), it is easy to verify $\zeta(t)-\zeta^0$ fulfills the third estimate of (\ref{20243172034}).

In the following, we estimate the derivative of $\zeta(t)$ with respect to $\rho$. In view of (\ref{20243192006}), (\ref{20243141942}) and \eqref{20245212009}$_1$, by Cauchy estimates and Lemma \ref{20243141530-3}, we get
		\begin{eqnarray}\label{20240620-1}
			\|\partial_{\rho}a(t)\|_{s+\beta}\leq 2\mu^{-1}[g]_{\sigma,\mu,\mathcal{D}}^{s,\beta+},
		\end{eqnarray}
		\begin{eqnarray}\label{20240620-2}
			\|\partial_{\rho}B(t)\|_{\mathcal{L}(Y_s,Y_{s+\beta'})}\leq \frac{2^{s+5}}{\beta-\beta'}\mu^{-2}[g]_{\sigma,\mu,\mathcal{D}}^{s,\beta+}.
		\end{eqnarray}
Then differentiating \eqref{20243171707} and \eqref{20243171708} with respect to $\rho$, we get
\begin{eqnarray}\label{20240525-1}	
	    \begin{aligned}
\|\partial_{\rho}a^{\infty}(t)\|_{s+\beta'}&\leq2\mu^{-1}[g]_{\sigma,\mu,\mathcal{D}}^{s,\beta+}+\sum_{k\geq2}k\Big(\frac{1}{2}\Big)^{k-1}2\mu^{-1}[g]_{\sigma,\mu,\mathcal{D}}^{s,\beta+}\\
&=8\mu^{-1}[g]_{\sigma,\mu,\mathcal{D}}^{s,\beta+}
        \end{aligned}
		\end{eqnarray}
and
		\begin{eqnarray}\label{20240625-2}
        \begin{aligned}
\|\partial_{\rho}B^{\infty}(t)\|_{\mathcal{L}(Y_s,Y_{s+\beta'})}&\leq\sum_{k\geq1}k\Big(\frac{1}{2}\Big)^{k-1} \frac{2^{s+5}}{\beta-\beta'}\mu^{-2}[g]_{\sigma,\mu,\mathcal{D}}^{s,\beta+}\\
&=\frac{2^{s+7}}{\beta-\beta'}{\mu^{-2}}[g]_{\sigma,\mu,\mathcal{D}}^{s,\beta+},
        \end{aligned}
		\end{eqnarray}
from which the third estimate of (\ref{20245212009}) for $\partial_{\rho}\zeta(t)$ is derived.

Moreover, by (\ref{20243192006}), we have
		\begin{eqnarray}\label{2024531336}
			|B(t)|_{s,\beta+}\leq \mu^{-2}[g]_{\sigma,\mu,\mathcal{D}}^{s,\beta+}.
		\end{eqnarray}
Then, by (\ref{20243141942}),(\ref{2024531336}) and lemma \ref{20243141530-2}, we have
		\begin{eqnarray}
			\begin{aligned}
				|B(t_1)\cdots B(t_k)|_{s,\beta+}&\leq \Big(2^{\frac{s}{2}+3}{\beta}^{-1}\mu^{-2}[g]_{\sigma,\mu,\mathcal{D}}^{s,\beta+}\Big)^{k-1} \mu^{-2}[g]_{\sigma,\mu,\mathcal{D}}^{s,\beta+}\\&\leq \Big(\frac{1}{2}\Big)^{k-1}\mu^{-2}[g]_{\sigma,\mu,\mathcal{D}}^{s,\beta+}
			\end{aligned}
		\end{eqnarray}
		and thus
		\begin{eqnarray}\label{20243172103}
			|B^{\infty}(t)|_{s,\beta+}\leq 2\mu^{-2}[g]_{\sigma,\mu,\mathcal{D}}^{s,\beta+}.
		\end{eqnarray}
In the same manner, we deduce 		
\begin{eqnarray}\label{2024522252}
				\begin{aligned}
				|\partial_{\rho}B^{\infty}(t)|_{s,\beta+}\leq 8\mu^{-2}[g]_{\sigma,\mu,\mathcal{D}}^{s,\beta+}.
				\end{aligned}
			\end{eqnarray}
Due to $U=I+B^{\infty}$, the estimates of $U$ in (\ref{20245211529})- (\ref{20245201521}) follow from (\ref{20245211907}), (\ref{20240625-2}), (\ref{20243172103}) and (\ref{2024522252}).

Finally, we consider the equation \eqref{2024351730}$_2$ for $r(t)$. It reads
			\begin{eqnarray}\label{20245212025}
				\dot{r}(t)=\alpha(t)+\Lambda(t)r(t),~~~r(0)=r^0\in \mathcal{O}_{{(\mu')}^2}(\mathbb{C}^n),
			\end{eqnarray}
where $\mathcal{O}_{(\mu')^2}(\mathbb{C}^n)=\{r\in\mathbb{C}^n:|r|\leq{(\mu')}^2\}$, $\Lambda(t)=-\nabla_{\theta}g_r\big(\theta(t);\rho\big)$, and
\begin{eqnarray*}
\alpha(t)=-\nabla_{\theta}g_{\theta}\big(\theta(t);\rho\big)-\langle \nabla_{\theta}g_{\zeta}\big(\theta(t);\rho\big),\zeta(t)\rangle-\frac{1}{2}\langle \zeta(t),\nabla_{\theta}g_{\zeta\zeta}\big(\theta(t);\rho\big)\zeta(t)\rangle.
\end{eqnarray*}
By (\ref{202431415171}), (\ref{20243141942}) and Cauchy estimate, we get
\begin{eqnarray}\label{20243172223}
\|\Lambda (t)\|_{\mathcal{L}(\mathbb{C}^n,\mathbb{C}^n)}\leq 2(\sigma-\sigma')^{-1}\mu^{-2}[g]_{\sigma,\mu,\mathcal{D}}^{s,\beta}\leq\frac{1}{2}.
\end{eqnarray}
In view of (\ref{20243141942}),(\ref{20243172034})$_3$ and Cauchy estimates, we get
\begin{eqnarray}\label{20243172139}
\begin{aligned}
|\alpha(t)|&\leq 2(\sigma-\sigma')^{-1}[g]_{\sigma,\mu,\mathcal{D}}^{s,\beta}\big(1+\mu^{-1}\|\zeta(t)\|_s+\frac{\mu^{-2}}{2}\|\zeta(t)\|_s^2\big)\\&\leq 4(\sigma-\sigma')^{-1}[g]_{\sigma,\mu,\mathcal{D}}^{s,\beta}\big(1+\mu^{-1}\|\zeta^0\|_s+\mu^{-2}\|\zeta^0\|_s^2\big).
\end{aligned}			
\end{eqnarray}
Similarly as $\zeta(t)$, we get
			\begin{eqnarray}\label{20243261953}
				r(t)=\alpha^{\infty}(t)+(1+\Lambda^{\infty}(t))r^0,
			\end{eqnarray}
			where
\begin{eqnarray}\label{20240625-3}
\Lambda^{\infty}(t)=\sum_{k\geq 1}\int_{0}^{t}\int_{0}^{t_1}\cdots\int_{0}^{t_{k-1}}\prod_{j=1}^k\Lambda(t_j)dt_k\cdots dt_2dt_1
\end{eqnarray}
and
\begin{eqnarray}\label{20245221309}
\alpha^{\infty}(t)=\int_{0}^{t}\alpha(t_1)dt_1+\sum_{k\geq 2}\int_{0}^{t}\int_{0}^{t_1}\cdots\int_{0}^{t_{k-1}}\prod_{j=1}^{k-1}\Lambda(t_j)\alpha(t_k)dt_k\cdots dt_2dt_1.
\end{eqnarray}
From (\ref{20243172223}) and (\ref{20243172139}), we obtain that
\begin{eqnarray}\label{2024521023}			\|\Lambda^{\infty}(t)\|_{\mathcal{L}(\mathbb{C}^n,\mathbb{C}^n)}\leq 4(\sigma-\sigma')^{-1}\mu^{-2}[g]_{\sigma,\mu,\mathcal{D}}^{s,\beta}
\end{eqnarray}
and
\begin{eqnarray}\label{2024521035}
|\alpha^{\infty}(t)|\leq 8(\sigma-\sigma')^{-1}\big(1+\mu^{-1}\|\zeta^0\|_s+\mu^{-2}\|\zeta^0\|_s^2\big)[g]_{\sigma,\mu,\mathcal{D}}^{s,\beta}.
\end{eqnarray}
Putting (\ref{2024521023}) and (\ref{2024521035}) into (\ref{20243261953}), it is easy to verify $r(t)-r^0$ fulfills the second estimate of (\ref{20243172034}).

In the following, we estimate the derivative of $r(t)$ with respect to $\rho$. By (\ref{202431415171}), (\ref{20243141942}), \eqref{20245212009}$_1$, (\ref{20243172034})$_3$ and Cauchy estimate, we get
\begin{eqnarray}\label{20245221325}
\|\partial_{\rho}\Lambda (t)\|_{\mathcal{L}(\mathbb{C}^n,\mathbb{C}^n)}\leq 4(\sigma-\sigma')^{-1}\mu^{-2}[g]_{\sigma,\mu,\mathcal{D}}^{s,\beta}
\end{eqnarray}
and
\begin{eqnarray}\label{20245221324}
\begin{aligned}
|\partial_{\rho}\alpha(t)|\leq 8(\sigma-\sigma')^{-1}[g]_{\sigma,\mu,\mathcal{D}}^{s,\beta}\big(1+\mu^{-1}\|\zeta^0\|_s+\mu^{-2}\|\zeta^0\|_s^2\big).
\end{aligned}			
\end{eqnarray}
Then differentiating \eqref{20240625-3} and \eqref{20245221309} with respect to $\rho$, we get
\begin{eqnarray}\label{20240625-4}			\|\partial_{\rho}\Lambda^{\infty}(t)\|_{\mathcal{L}(\mathbb{C}^n,\mathbb{C}^n)}\leq 16(\sigma-\sigma')^{-1}\mu^{-2}[g]_{\sigma,\mu,\mathcal{D}}^{s,\beta}
\end{eqnarray}
and
\begin{eqnarray}\label{20240625-5}
|\partial_{\rho}\alpha^{\infty}(t)|\leq 32(\sigma-\sigma')^{-1}\big(1+\mu^{-1}\|\zeta^0\|_s+\mu^{-2}\|\zeta^0\|_s^2\big)[g]_{\sigma,\mu,\mathcal{D}}^{s,\beta}
\end{eqnarray}
from which the second estimate of (\ref{20245212009}) for $\partial_{\rho}r(t)$ is derived.

Moreover, due to $V=I+\Lambda^{\infty}$, by (\ref{2024521023}) and (\ref{20240625-4}), the estimate (\ref{20245201706}) holds.
By \eqref{20243192006} and Cauchy estimates, we have
\begin{eqnarray}\label{20240701-1}		
\sup_{\nu=0,1}\|\partial_{\rho}^{\nu}\nabla_{\theta}g_{\zeta}\big(\theta(t);\rho\big)\|_{s+\beta}
\leq2(\sigma-\sigma')^{-1}{\mu}^{-1}[g]_{\sigma,\mu,\mathcal{D}}^{s,\beta+},
\end{eqnarray}
\begin{eqnarray}\label{20240701-2}
\sup_{\nu=0,1}|\partial_{\rho}^{\nu}\nabla_{\theta}g_{\zeta\zeta}\big(\theta(t);\rho\big)|_{s,\beta+}
\leq2(\sigma-\sigma')^{-1}{\mu}^{-2}[g]_{\sigma,\mu,\mathcal{D}}^{s,\beta+},
\end{eqnarray}
and thus by Lemma \ref{20243141530-3},
\begin{eqnarray}\label{20240701-3}
\sup_{\nu=0,1}|\partial_{\rho}^{\nu}\nabla_{\theta}g_{\zeta\zeta}\big(\theta(t);\rho\big)|_{\mathcal{L}(Y_s,Y_{s+\beta'})}
\leq\frac{2^{s+5}}{\beta-\beta'}(\sigma-\sigma')^{-1}{\mu}^{-2}[g]_{\sigma,\mu,\mathcal{D}}^{s,\beta+}.
\end{eqnarray}
By direct calculation, we have
\begin{eqnarray*}
\nabla_{\zeta^0}\alpha(t)=-{}^tU\nabla_{\theta}g_{\zeta}\big(\theta(t);\rho\big)-{}^tU\nabla_{\theta}g_{\zeta\zeta}\big(\theta(t);\rho\big)\zeta(t)
\end{eqnarray*}
and
\begin{eqnarray*}
\nabla_{\zeta^0}^2\alpha(t)=-{}^tU\nabla_{\theta}g_{\zeta\zeta}\big(\theta(t);\rho\big)U.
\end{eqnarray*}
Then by \eqref{20243141942}, \eqref{20245211529}, \eqref{20240701-1}, \eqref{20240701-3} and $\|\zeta(t)\|_s\leq\frac{\mu+\mu'}{2}$, we get
\begin{eqnarray}\label{20240701-4}
\|\nabla_{\zeta^0}\alpha(t)\|_{s+\beta'}
\leq\frac{2^{s+6}}{\beta-\beta'}(\sigma-\sigma')^{-1}{\mu}^{-1}[g]_{\sigma,\mu,\mathcal{D}}^{s,\beta+};
\end{eqnarray}
by \eqref{20243141942}, \eqref{20245201521}, \eqref{20240701-2} and Lemma \ref{20243141530-2}, we get
\begin{eqnarray}\label{20240701-5}
|\nabla_{\zeta^0}^2\alpha(t)|_{s,\beta+}
\leq4(\sigma-\sigma')^{-1}{\mu}^{-2}[g]_{\sigma,\mu,\mathcal{D}}^{s,\beta+}.
\end{eqnarray}
Furthermore, we have
\begin{eqnarray*}
\begin{aligned}
\partial_{\rho}\nabla_{\zeta^0}\alpha(t)&=-(\partial_{\rho}{}^tU)\nabla_{\theta}g_{\zeta}\big(\theta(t);\rho\big)-{}^tU\partial_{\rho}\nabla_{\theta}g_{\zeta}\big(\theta(t);\rho\big)-(\partial_{\rho}{}^tU)\nabla_{\theta}g_{\zeta\zeta}\big(\theta(t);\rho\big)\zeta(t)\\
&-{}^tU\partial_{\rho}\nabla_{\theta}g_{\zeta\zeta}\big(\theta(t);\rho\big)\zeta(t)-{}^tU\nabla_{\theta}g_{\zeta\zeta}\big(\theta(t);\rho\big)\partial_{\rho}\zeta(t)
\end{aligned}			
\end{eqnarray*}
and
\begin{eqnarray*}
\partial_{\rho}\nabla_{\zeta^0}^2\alpha(t)=-(\partial_{\rho}{}^tU)\nabla_{\theta}g_{\zeta\zeta}\big(\theta(t);\rho\big)U-{}^tU\partial_{\rho}\nabla_{\theta}g_{\zeta\zeta}\big(\theta(t);\rho\big)U-{}^tU\nabla_{\theta}g_{\zeta\zeta}\big(\theta(t);\rho\big)(\partial_{\rho}U).
\end{eqnarray*}
Then in the same manner, we deduce
\begin{eqnarray}\label{20240701-6}
\|\partial_{\rho}\nabla_{\zeta^0}\alpha(t)\|_{s+\beta'}
\leq\frac{2^{s+7}}{\beta-\beta'}(\sigma-\sigma')^{-1}{\mu}^{-1}[g]_{\sigma,\mu,\mathcal{D}}^{s,\beta+},
\end{eqnarray}
\begin{eqnarray}\label{20240701-7}
|\partial_{\rho}\nabla_{\zeta^0}^2\alpha(t)|_{s,\beta+}
\leq8(\sigma-\sigma')^{-1}{\mu}^{-2}[g]_{\sigma,\mu,\mathcal{D}}^{s,\beta+}.
\end{eqnarray}
Since $L(\theta^0,\zeta^0;t)=\alpha^{\infty}(t)$ with the formula \eqref{20245221309} and $\Lambda(t)$ is independent of $\zeta^0$, then the estimates of $L$ in (\ref{20243172102}) and (\ref{2024317210211}) follow from (\ref{20240701-4})-(\ref{20240701-7}).
\end{proof}

The next lemma indicates the preservation of functions $h\in\mathcal{T}^{s,\beta}_{\sigma,\mu,\mathcal{D}}$ under the flow maps $\Phi_g^t$.
\begin{lem}\label{20243212033}
		Let $0<\beta'<\beta\leq\min\{\frac{s}{2},\frac{1}{2}\}$,
$0<\sigma'<\sigma\leq1$,
$0<\mu'<\mu\leq1$,
$h\in \mathcal{T}^{s,\beta}_{\sigma,\mu,\mathcal{D}}$, and $g=g^T\in\mathcal{T}^{s,\beta+}_{\sigma,\mu,\mathcal{D}}$ satisfy (\ref{20243141942}) with $c_1>1$ being a properly large constant only depending on $n$ and $s$. Then for $0\leq t\leq 1$, we have
$h_t:=h\big(\Phi_g^t(\theta,r,\zeta;\rho);\rho\big)\in\mathcal{T}^{s,\beta'}_{\sigma',\mu',\mathcal{D}}$
with the following estimate
\begin{eqnarray}
[h_t]_{\sigma',\mu',\mathcal{D}}^{s,\beta'}\leq2[h]_{\sigma,\mu,\mathcal{D}}^{s,\beta}.
\end{eqnarray}
\end{lem}

\begin{proof}			
Write the flow map $\Phi_g^t$ as
		\begin{eqnarray*}
			(\theta^0,r^0,\zeta^0)\mapsto (\theta(t),r(t),\zeta(t)).
		\end{eqnarray*}
By Lemma \ref{20243181426}, $h_t(\theta^0,r^0,\zeta^0;\rho)$ is analytical in $(\theta^0,r^0,\zeta^0)\in \mathcal{O}^s(\sigma',\mu')$ with
		\begin{eqnarray}\label{20245221511}
			|h_t(\theta^0,r^0,\zeta^0;\rho)|\leq [h]_{\sigma,\mu,\mathcal{D}}^{s,\beta}.
		\end{eqnarray}
Moreover, by (\ref{20243141942}), (\ref{20245212009}) and Cauchy estimates, we can deduce
		\begin{eqnarray}\label{20240702-1}
			|\partial_{\rho}h_t(\theta^0,r^0,\zeta^0;\rho)|\leq 2[h]_{\sigma,\mu,\mathcal{D}}^{s,\beta}.
		\end{eqnarray}

Now we estimate the gradient of $h_t(\theta^0,r^0,\zeta^0)$. Since $\theta(t)$ does not depend on $\zeta^0$, we have
		\begin{eqnarray*}
			\nabla_{\zeta^0}h_t=\sum_{b=1}^{n}\frac{\partial h}{\partial r_b}\cdot\frac{\partial r_b}{\partial\zeta^0}+{}^tU\nabla_{\zeta}h.
		\end{eqnarray*}
Since $(\theta(t),r(t),\zeta(t))\in\mathcal{O}^s(\frac{\sigma+\sigma'}{2},\frac{\mu+\mu'}{2})$, by Cauchy estimates, we get
		\begin{eqnarray}\label{20245221521}
			\Big|\frac{\partial h}{\partial r_b}\Big|\leq 4(\mu-\mu')^{-2}[h]_{\sigma,\mu,\mathcal{D}}^{s,\beta}.
		\end{eqnarray}
By (\ref{20243141942}), (\ref{20243172102}) and (\ref{20245221521}), we get
		\begin{eqnarray}\label{20240701-8}
			\Big\|\sum_{b=1}^{n}\frac{\partial h}{\partial r_b}\cdot\frac{\partial r_b}{\partial\zeta^0}\Big\|_{s}
&\leq&\frac{2^{s+11}n}{\beta-\beta'}(\sigma-\sigma')^{-1}{\mu}^{-1}(\mu-\mu')^{-2}[h]_{\sigma,\mu,\mathcal{D}}^{s,\beta}{[g]_{\sigma,\mu,\mathcal{D}}^{s,\beta+}}\nonumber\\ &\leq&\frac{1}{2}\mu^{-1}[h]_{\sigma,\mu,\mathcal{D}}^{s,\beta};
		\end{eqnarray}
by (\ref{20243141942}) and (\ref{20245211529}), we get
		\begin{eqnarray}\label{20240701-9}
			\|{}^tU\nabla_{\zeta}h\|_{s}
&\leq&\Big(1+\frac{2^{s+7}{\mu}^{-2}}{\beta-\beta'}[g]_{\sigma,\mu,\mathcal{D}}^{s,\beta+}\Big){\mu}^{-1}[h]_{\sigma,\mu,\mathcal{D}}^{s,\beta}\nonumber\\ &\leq&\frac{3}{2}\mu^{-1}[h]_{\sigma,\mu,\mathcal{D}}^{s,\beta}.
		\end{eqnarray}
Summing the estimates \eqref{20240701-8} and \eqref{20240701-9}, we get
\begin{eqnarray}\label{20240701-10}
			\|\nabla_{\zeta^0}h_t\|_{s}
\leq2\mu^{-1}[h]_{\sigma,\mu,\mathcal{D}}^{s,\beta}.
		\end{eqnarray}
Similarly, we can estimate the derivative of $\nabla_{\zeta^0}h_t$ with respect to $\rho$ and get
		\begin{eqnarray}\label{20240701-11}
			\|\partial_{\rho}\nabla_{\zeta^0}h_t\|_{s}
\leq2\mu^{-1}[h]_{\sigma,\mu,\mathcal{D}}^{s,\beta}.
		\end{eqnarray}

Next we estimate the Hessian of $h_t(\theta^0,r^0,\zeta^0)$. By direct calculation, we have
		\begin{eqnarray}\label{20243251525}
			\begin{aligned}
				\nabla_{\zeta^0}^2h_t
&=\sum_{b,b'=1}^{n}\frac{\partial^2 h}{\partial r_{b}\partial r_{b'}}\frac{\partial r_{b}}{\partial \zeta^0}\otimes\frac{\partial r_{b'}}{\partial \zeta^0}
+\sum_{b=1}^{n}\frac{\partial h}{\partial r_b}\frac{\partial^2r_b}{(\partial\zeta^0)^2}\\
&+\sum_{b=1}^{n}\frac{\partial r_b}{\partial \zeta^0}\otimes \big({}^tU\frac{\partial^2h}{\partial r_b\partial\zeta}\big)
+\sum_{b=1}^{n}\big({}^tU\frac{\partial^2 h}{\partial\zeta\partial r_b}\big)\otimes\frac{\partial r_b}{\partial \zeta^0}
+{}^tU\frac{\partial^2 h}{\partial\zeta^2}U.
			\end{aligned}
		\end{eqnarray}
For simplicity, let $\textrm{I}_{1}$, $\textrm{I}_{2}$, $\textrm{I}_{3}$, $\textrm{I}_{4}$, $\textrm{I}_{5}$ be the right side of (\ref{20243251525}) successively. In the following, we will estimate them separately.

		(\rmnum{1})
		By Cauchy estimates, we get
		\begin{eqnarray}\label{202431419422}
			\Big|\frac{\partial^2 h}{\partial r_{b}\partial r_{b'}}\Big|\leq 32(\mu-\mu')^{-4}[h]_{\sigma,\mu,\mathcal{D}}^{s,\beta}.
		\end{eqnarray}
		Then by (\ref{20243141942}), (\ref{20243172102}), (\ref{202431419422}) and Lemma \ref{20243141530-4}, one has
		\begin{eqnarray}\label{20240702-2}
			\begin{aligned}
				|\textrm{I}_{1}|_{s,\beta'}
&\leq2^{2s+23}n^2(\beta-\beta')^{-2}(\sigma-\sigma')^{-2}(\mu-\mu')^{-4}\mu^{-2}\big({[g]_{\sigma,\mu,\mathcal{D}}^{s,\beta+}}\big)^2[h]_{\sigma,\mu,\mathcal{D}}^{s,\beta}\\
&\leq\frac{1}{5}\mu^{-2}[h]_{\sigma,\mu,\mathcal{D}}^{s,\beta}.
			\end{aligned}
		\end{eqnarray}

        (\rmnum{2})
        By (\ref{20243141942}), (\ref{2024317210211}) and (\ref{20245221521}), one has
		\begin{eqnarray}\label{20240702-3}
			\begin{aligned}
				|\textrm{I}_{2}|_{s,\beta+}
&\leq2^{7}n(\sigma-\sigma')^{-1}(\mu-\mu')^{-2}\mu^{-2}{[g]_{\sigma,\mu,\mathcal{D}}^{s,\beta+}}[h]_{\sigma,\mu,\mathcal{D}}^{s,\beta}\\
&\leq\frac{1}{5}\mu^{-2}[h]_{\sigma,\mu,\mathcal{D}}^{s,\beta}.
			\end{aligned}
		\end{eqnarray}

		(\rmnum{3})
        By Cauchy estimates, we get
		\begin{eqnarray}\label{20243202025}
			\Big\|\frac{\partial^2h}{\partial r_b\partial\zeta}\Big\|_{s}\leq 4(\mu-\mu')^{-2}\mu^{-1}[h]_{\sigma,\mu,\mathcal{D}}^{s,\beta}.
		\end{eqnarray}
       By (\ref{20243141942}), (\ref{20245211529}) and (\ref{20243202025}), we get
		\begin{eqnarray}\label{20240702-4}
			\Big\|{}^tU\frac{\partial^2 h}{\partial r_b\partial \zeta}\Big\|_s
&\leq&4(\mu-\mu')^{-2}{\mu}^{-1}\Big(1+\frac{2^{s+7}{\mu}^{-2}}{\beta-\beta'}[g]_{\sigma,\mu,\mathcal{D}}^{s,\beta+}\Big)[h]_{\sigma,\mu,\mathcal{D}}^{s,\beta}\nonumber\\ &\leq&8(\mu-\mu')^{-2}{\mu}^{-1}[h]_{\sigma,\mu,\mathcal{D}}^{s,\beta}.
		\end{eqnarray}
		Then by (\ref{20243141942}), (\ref{20243172102}), (\ref{20240702-4}) and Lemma \ref{20243141530-4}, one has
		\begin{eqnarray}\label{20240702-5}
			\begin{aligned}
				|\textrm{I}_{3}|_{s,\beta'}
&\leq\sum_{b=1}^{n}\Big\|\frac{\partial r_b}{\partial \zeta^0}\Big\|_{s+\beta'}	\Big\|{}^tU\frac{\partial^2 h}{\partial r_b\partial \zeta}\Big\|_s\\
&\leq2^{s+12}n(\beta-\beta')^{-1}(\sigma-\sigma')^{-1}(\mu-\mu')^{-2}\mu^{-2}{[g]_{\sigma,\mu,\mathcal{D}}^{s,\beta+}}[h]_{\sigma,\mu,\mathcal{D}}^{s,\beta}\\
&\leq\frac{1}{5}\mu^{-2}[h]_{\sigma,\mu,\mathcal{D}}^{s,\beta}.
			\end{aligned}
		\end{eqnarray}

		(\rmnum{4})
        The same as the estimate of $\textrm{I}_{3}$, one has
        \begin{eqnarray}\label{20240702-6}
				|\textrm{I}_{4}|_{s,\beta'}
\leq\frac{1}{5}\mu^{-2}[h]_{\sigma,\mu,\mathcal{D}}^{s,\beta}.
		\end{eqnarray}

		(\rmnum{5})
        In view of (\ref{20240605-1}), we have
        \begin{eqnarray}\label{20240702-7}
				\Big|\frac{\partial^2 h}{\partial\zeta^2}\Big|_{s,\beta}
\leq\mu^{-2}[h]_{\sigma,\mu,\mathcal{D}}^{s,\beta}.
		\end{eqnarray}
         Then by (\ref{20243141942}), (\ref{20245201521}), (\ref{20240702-7}) and Lemma \ref{20243141530-1}, one has
         \begin{eqnarray}\label{20240702-8}
			\begin{aligned}
				|\textrm{I}_{5}|_{s,\beta}
&\leq\Big(1+2^{\frac{s}{2}+2}\beta^{-1}|U-I|_{s,\beta+}\Big)^2
\Big|\frac{\partial^2 h}{\partial\zeta^2}\Big|_{s,\beta}\\
&\leq\Big(1+2^{\frac{s}{2}+5}\beta^{-1}\mu^{-2}{[g]_{\sigma,\mu,\mathcal{D}}^{s,\beta+}}\Big)^2\mu^{-2}[h]_{\sigma,\mu,\mathcal{D}}^{s,\beta}\\
&\leq\frac{6}{5}\mu^{-2}[h]_{\sigma,\mu,\mathcal{D}}^{s,\beta}.
			\end{aligned}
		\end{eqnarray}
Summing the estimates \eqref{20240702-2}, \eqref{20240702-3}, \eqref{20240702-5}, \eqref{20240702-6} and \eqref{20240702-8}, we get
\begin{eqnarray}\label{20240702-9}
			|\nabla_{\zeta^0}^2h_t|_{s,\beta'}
\leq2\mu^{-2}[h]_{\sigma,\mu,\mathcal{D}}^{s,\beta}.
		\end{eqnarray}
Similarly, we can estimate the derivative of $\nabla_{\zeta^0}^2h_t$ with respect to $\rho$ and get
		\begin{eqnarray}\label{20240702-10}
			|\partial_{\rho}\nabla_{\zeta^0}^2h_t|_{s,\beta'}
\leq2\mu^{-2}[h]_{\sigma,\mu,\mathcal{D}}^{s,\beta}.
		\end{eqnarray}

		 Finally, by combining the estimates \eqref{20245221511}, \eqref{20240702-1}, \eqref{20240701-10}, \eqref{20240701-11}, \eqref{20240702-9} and \eqref{20240702-10}, we complete the proof of this lemma. 		
	\end{proof}

\section{The KAM step}
\subsection{Homological equation}

	At each KAM step, we look for a transformation such that the perturbation of the transformed Hamiltonian becomes smaller. Roughly speaking, for the $m$-th step, the Hamiltonian
	\begin{eqnarray*}
		H_m=h_m+f_m
	\end{eqnarray*}
	is considered as a small perturbation of the normal form $h_m$, where $f_m$ is of order $\varepsilon_m$. A transformation $\Phi_m$ is set up so that
	\begin{eqnarray*}
		H_m\circ\Phi_m=h_{m+1}+f_{m+1}
	\end{eqnarray*}
	with another normal form $h_{m+1}$ and a much smaller perturbation $f_{m+1}$ of size $\varepsilon_m^2$. We drop the index of $H_m,h_m,f_m,\Phi_m$ and shorten the index $m+1$ as $+$.
	
	Let $h$ be a normal form
	\begin{eqnarray*}
		h(r,\zeta;\rho)=\langle \omega(\rho),r\rangle+\frac{1}{2}\langle \zeta,A(\rho)\zeta\rangle
	\end{eqnarray*}
and $f^T$ be the $2$-order Taylor polynomial truncation of $f$, that is,
	\begin{eqnarray*}
		f^T=f_{\theta}+\langle f_r,r\rangle+\langle f_{\zeta
		},\zeta\rangle+\frac{1}{2}\langle \zeta,f_{\zeta\zeta}\zeta\rangle,
	\end{eqnarray*}
	where $f_{\theta},f_r,f_{\zeta}$ and $f_{\zeta\zeta}$ depend on $\theta$ and $\rho$. The coordinate transformation $\Phi$ is obtained as the time-$1$-map $X_{S}^t|_{t=1}$ of a Hamiltonian vector field $X_S$, where $S$ is of the same form as $f^T$:
	\begin{eqnarray}
		S(\theta,r,\zeta;\rho)=S_{\theta}(\theta;\rho)+\langle S_r(\theta;\rho),r\rangle+\langle S_{\zeta}(\theta;\rho),\zeta\rangle+\frac{1}{2}\langle \zeta,S_{\zeta\zeta}(\theta;\rho)\zeta\rangle.
	\end{eqnarray}
	Thus, it follows that
	\begin{eqnarray}\label{20243211944}
		\begin{aligned}
			H\circ \Phi&=(h+f^T)\circ X_S^1+(f-f^T)\circ X_S^1\\&=h+\{h,S\}+f^T+\int_{0}^{1}\{(1-t)\{h,S\}+f^T,S\}\circ X_S^tdt+(f-f^T)\circ X_S^1.
		\end{aligned}
	\end{eqnarray}

Introduce the cut-off operator $\Gamma_N$ with $N\in\mathbb{N}^+$. For a function $g$ defined on $\mathbb{T}_{\sigma}^n$, define
	\begin{eqnarray*}
		\Gamma_Ng(\theta):=\sum_{|k|\leq N}\hat{g}(k)e^{\textbf{i}\langle k,\theta\rangle},
	\end{eqnarray*}
	where $\hat{g}(k)$ is the $k$- Fourier coefficient of $g$. For $f\in\mathcal{T}_{\sigma,\mu,\mathcal{D}}^{s,\beta}$, define $\Gamma_Nf$ in the same way, and choose large enough $N$ such that $f-\Gamma_Nf$ is of order $\varepsilon^2$. We will solve the homological equation
	\begin{eqnarray}\label{20240709-4}
			\{h,S\}+\Gamma_Nf^T=\hat{h}
	\end{eqnarray}
with the shift term
\begin{eqnarray}\label{2024661530}
		\hat{h}:=[\![f_{\theta}]\!]+\langle [\![f_{r}]\!],r\rangle+\frac{1}{2}\langle {\zeta},B_{\zeta\zeta}\zeta\rangle,
	\end{eqnarray}
where $[\![u]\!] $ denotes the averaging of a function $u$ in $\theta\in \mathbb{T}^n$, and $B_{\zeta\zeta}\in \mathcal{NF}$ will be chosen later. More precisely, the homological equation \eqref{20240709-4} can be written as	
	\begin{eqnarray}\label{20245231310}
		-\partial_{\omega}S_{\theta}+\Gamma_Nf_{\theta}-[\![f_{\theta}]\!]&=&0,\\	-\partial_{\omega}S_{r}+\Gamma_Nf_{r}-[\![f_{r}]\!]&=&0,\label{20245231513}\\	-\partial_{\omega}S_{\zeta}+AJS_{\zeta}+\Gamma_Nf_{\zeta}&=&0,\label{20245231307}\\
-\partial_{\omega}S_{\zeta\zeta}+AJS_{\zeta\zeta}-S_{\zeta\zeta}JA-B_{\zeta\zeta}+\Gamma_Nf_{\zeta\zeta}&=&0.\label{20245231311}
	\end{eqnarray}
We are now in a position to solve these homological equations respectively. In what follows the notation $a\lessdot b$ stands for ``there exists a positive constant $c$ such that $a\leq cb$, where $c$ only depends on $n,s,b_0,b_1,c^*,d^*,M_{\omega},M_{\Omega},\beta$.''

$\bullet$~~~The first two equations

Assume the tangential frequency $\omega(\rho)$ satisfies
\begin{eqnarray}\label{20245262015}
\sup_{\nu=0,1,\rho\in\mathcal{D}}|\partial_{\rho}^{\nu}\omega(\rho)|\leq M_{\omega},
\end{eqnarray}
and normal form matrix $A(\rho)$ of the form \eqref{20240709-5} satisfies
	\begin{eqnarray}\label{20245271500}
\sup_{\nu=0,1,\rho\in\mathcal{D}}\|\partial_{\rho}^{\nu}\tilde{A}_j^j(\rho)\|\leq \frac{M_{\Omega}}{j^{\beta}}.
	\end{eqnarray}
Under the following non-resonant condition
	\begin{eqnarray}
		|\langle k,\omega(\rho)\rangle|\geq \kappa~~~\text{for}~~~0<|k|\leq N,
	\end{eqnarray}	
where $0<\kappa<1$, the solutions of the homological equations (\ref{20245231310}) and (\ref{20245231513}) can be obtained by the standard procedures and the details are omitted.
$\bullet$~~~The third equation

Introduce the complex variables ${}^t(\xi,\eta)$ by
	\begin{eqnarray}\label{20245262056}
		\zeta_{(j,\iota)}=	\begin{pmatrix}
			p_{(j,\iota)}\\ q_{(j,\iota)}
		\end{pmatrix}=U\begin{pmatrix}
			\xi_{(j,\iota)}\\ \eta_{(j,\iota)}
		\end{pmatrix},~~~\text{for}~~~(j,\iota)\in\mathfrak{L},
	\end{eqnarray}
where
\begin{eqnarray*}
		U=\begin{pmatrix}
			\frac{1}{\sqrt{2}} & \frac{1}{\sqrt{2}}\\ \frac{-\textbf{i}}{\sqrt{2}} & \frac{\textbf{i}}{\sqrt{2}}
		\end{pmatrix}.
\end{eqnarray*}
Let $Q:\mathfrak{L}\times\mathfrak{L}\to \mathbb{C}$ be the complex valued normal form matrix associated to $A$, i.e., $$\frac{1}{2}\langle\zeta,A\zeta\rangle=\langle\xi,Q\eta\rangle.$$
Since $Q$ is Hermitian and block diagonal, i.e., $Q=\text{diag}(Q_j^j)$, then for every $j\geq1$, there exists a unitary matrix $P_j$ such that
	\begin{eqnarray}\label{20240710-1}
	{}^tP_jQ_j^jP_j=D_j:=\text{diag}(\Omega_{j,\iota},\iota=1,\cdots,d_j).
	\end{eqnarray}

Next, we solve the homological equation (\ref{20245231307}) under the following non-resonant condition
	\begin{eqnarray}\label{20245231546}
		|\langle k,\omega(\rho)\rangle+\Omega_{j,\iota}(\rho)|\geq \kappa j~~~\text{for}~~~|k|\leq N.
	\end{eqnarray}
Multiplying by ${}^tUJ$ in (\ref{20245231307}), we have
	\begin{eqnarray}\label{20246221328}	-\partial_{\omega}{}^tUJS_{\zeta}+{}^tUJAJS_{\zeta}+\Gamma_N{}^tUJf_{\zeta}=0.
	\end{eqnarray}
By the substitution ${}^tUJS_{\zeta}=S$, ${}^tUJf_{\zeta}=F$ and the fact ${}^tUJU=\textbf{i}J$,	equation (\ref{20246221328}) becomes
	\begin{eqnarray}\label{20246221509}
		-\partial_{\omega}S+\textbf{i}J{}^tUAUS+\Gamma_NF=0.
	\end{eqnarray}
Denote
	\begin{eqnarray} S={}^t(S_j^{\xi},S_j^{\eta})_{j\geq1},~~~~~~F={}^t(F_j^{\xi},F_j^{\eta})_{j\geq1}.
	\end{eqnarray}
Then by the relationship
	\begin{eqnarray}\label{20246221623}
		({}^tUAU)_j^j=\begin{pmatrix}
			0  & Q_j^j \\ {}^tQ_j^j & 0
		\end{pmatrix},
	\end{eqnarray}
the system (\ref{20246221509}) is decoupled into
	\begin{eqnarray}\label{20246221610}	-\partial_{\omega}S_{j}^{\xi}-\textbf{i}({}^tQ_{j}^j)S_{j}^{\xi}+\Gamma_NF_{j}^{\xi}&=&0,\\	-\partial_{\omega}S_{j}^{\eta}
		+\textbf{i}(Q_{j}^j)S_{j}^{\eta}+\Gamma_NF_{j}^{\eta}&=&0.
	\end{eqnarray}
Since the two equations above are conjugated, we just consider (\ref{20246221610}). Expanding the functions $S_j^{\xi},F_j^{\xi}$ with respect to $\theta$ into Fourier series, it becomes
	\begin{eqnarray}\label{20245231543}
		-\textbf{i}(\langle k,\omega\rangle+{}^tQ_j^j)\hat{S}_j(k)+\hat{F}_j(k)=0,~~~|k|\leq N,
	\end{eqnarray}
	where we have suppressed the upper index $\xi$.
	Multiplied by ${}^tP_j$ in the left side of (\ref{20245231543}), the equation turns into
	\begin{eqnarray}\label{20240710-3}
		-\textbf{i}(\langle k,\omega\rangle+D_j)\hat{S}_j'(k)+\hat{F}_j'(k)=0,~~~|k|\leq N
	\end{eqnarray}
	with $\hat{S}_j'(k)={}^tP_j\hat{S}_j(k)$ and $\hat{F}_j'(k)={}^tP_j\hat{F}_j(k)$.
By (\ref{20245231546}) and \eqref{20240710-3}, we have
	\begin{eqnarray}\label{20245262016}
		|\hat{S}_j(k)|=|\hat{S}_j'(k)|\leq \frac{1}{\kappa j}|\hat{F}_j(k)|.
	\end{eqnarray}
Then in view of the relationship between $S_{\zeta}$ and ${}^t(S^{\xi},S^{\eta})$, we have
	\begin{eqnarray}
		\begin{aligned}
			\|\hat{S}_{\zeta}(k;\rho)\|_{s+1}^2&=\sum_{j\geq 1}|(\hat{S}_{\zeta}(k;\rho))_j|^2j^{2s+2}\\&=\sum_{j\geq 1}(|\hat{S}_j^{\xi}(k;\rho)|^2+|\hat{S}_j^{\eta}(k;\rho)|^2)j^{2s+2}\\&\leq \frac{1}{\kappa^2}\sum_{j\geq 1}|(\hat{F}_{\zeta}(k;\rho))_j|^2j^{2s}\\&= \frac{1}{\kappa^2}\|\hat{f}_{\zeta}(k;\rho)\|_s^2.
		\end{aligned}
	\end{eqnarray}
	Furthermore, for any $\theta\in\mathbb{T}_{\sigma'}^n$,
	\begin{eqnarray}\label{20240712-1}
		\|S_{\zeta}(\theta;\rho)\|_{s+1}\lessdot \frac{1}{\kappa(\sigma-\sigma')^n}\sup_{\theta\in\mathbb{T}_{\sigma}^n}\|f_{\zeta}(\theta;\rho)\|_s.
	\end{eqnarray}

In the following, we give the estimate of $\partial_{\rho}\hat{S}_{\zeta}(k;\rho)$. Differentiating (\ref{20245231543}), we get
	\begin{eqnarray}\label{20245262038}
		-\textbf{i}(\langle k,\omega\rangle+Q_j^j)\partial_{\rho}\hat{S}_j(k;\rho)+\hat{R}_j(k;\rho)=0
	\end{eqnarray}
	with
$$\hat{R}_j(k;\rho):=\partial_{\rho}\hat{F}_j(k;\rho)-\textbf{i}(\langle k,\partial_{\rho}\omega\rangle+\partial_{\rho}Q_j^j)\hat{S}_j(k;\rho).$$
By (\ref{20245262015}), (\ref{20245271500}) and (\ref{20246221623}), we know
\begin{eqnarray}\label{2024672045}
|\partial_{\rho}\omega|\leq M_{\omega},~~~~~~\|\partial_{\rho}Q_j^j\|\leq M_{\Omega}.
\end{eqnarray}
Then due to (\ref{20245262016}), we have
	\begin{eqnarray}
		|\hat{R}_j(k;\rho)|\leq |\partial_{\rho}\hat{F}_j(k;\rho)|+(| k|M_{\omega}+M_{\Omega})\frac{1}{\kappa j}|\hat{F}_j(k;\rho)|.
	\end{eqnarray}
	Similar to (\ref{20245231543}), for (\ref{20245262038}), we obtain
	\begin{eqnarray}
		\begin{aligned}
			|\partial_{\rho}\hat{S}_j(k;\rho)|&\leq \frac{1}{\kappa j}|\hat{R}_j(k;\rho)|\\&\leq \frac{1}{\kappa j}|\partial_{\rho}\hat{F}_j(k;\rho)|+\frac{|k|M_{\omega}+M_{\Omega}}{\kappa^2j^2}|\hat{F}_j(k;\rho)|.
		\end{aligned}
	\end{eqnarray}
	Then it follows that
	\begin{eqnarray}
		\begin{aligned}
			\|\partial_{\rho}\hat{S}_{\zeta}(k;\rho)\|_{s+1}^2&=\sum_{j\geq 1}|(\partial_{\rho}\hat{S}_{\zeta}(k;\rho))_j|^2j^{2s+2}\\&=\sum_{j\geq 1}(|\partial_{\rho}\hat{S}_j^{\xi}(k;\rho)|^2+|\partial_{\rho}\hat{S}_j^{\eta}(k;\rho)|^2)j^{2s+2}\\&\lessdot \frac{1}{\kappa^2}\sum_{j\geq 1}|(\partial_{\rho}\hat{F}_{\zeta}(k;\rho))_j|^2j^{2s}+\frac{(|k|+1)^2}{\kappa^4}\sum_{j\geq 1}|(\hat{F}_{\zeta}(k;\rho))_j|^2j^{2s}\\&= \frac{1}{\kappa^2}\|\partial_{\rho}\hat{f}_{\zeta}(k;\rho)\|_s^2+\frac{(|k|+1)^2}{\kappa^4}\|\hat{f}_{\zeta}(k;\rho)\|_s^2.
		\end{aligned}
	\end{eqnarray}
Therefore, for any $\theta\in\mathbb{T}_{\sigma'}^n$,
	\begin{eqnarray}\label{20240712-2}
			\|\partial_{\rho} S_{\zeta}(\theta;\rho)\|_{s+1}\lessdot \frac{1}{\kappa(\sigma-\sigma')^n}\sup_{\theta\in\mathbb{T}_{\sigma}^n}\|\partial_{\rho} f_{\zeta}(\theta;\rho)\|_{s}+\frac{1}{\kappa^2(\sigma-\sigma')^{n+1}}\sup_{\theta\in\mathbb{T}_{\sigma}^n}\| f_{\zeta}(\theta;\rho)\|_{s}.
	\end{eqnarray}
Summing the estimates \eqref{20240712-1} and \eqref{20240712-2}, for any $\theta\in\mathbb{T}_{\sigma'}^n$, we get
	\begin{eqnarray}\label{20240712-3}
			\sup_{\nu=0,1}\|\partial_{\rho}^{\nu} S_{\zeta}(\theta;\rho)\|_{s+1}\lessdot \frac{1}{\kappa^2(\sigma-\sigma')^{n+1}}{\mu^{-1}}[f]_{\sigma,\mu,\mathcal{D}}^{s,\beta}.
	\end{eqnarray}

$\bullet$~~~The fourth equation
	
	Now, we solve the homological equation (\ref{20245231311}). By the substitution
	\begin{eqnarray}\label{20245271514}
		S:={}^tUS_{\zeta\zeta}U,~B:={}^tUB_{\zeta\zeta}U,~F:={}^tUf_{\zeta\zeta}U
	\end{eqnarray}
and the fact ${}^tUJU=\textbf{i}J$, the equation (\ref{20245231311}) becomes
	\begin{eqnarray}\label{20245262120}
	-\partial_{\omega}S+\textbf{i}{}^tUAUJS-\textbf{i}SJ{}^tUAU-B+\Gamma_NF=0.
	\end{eqnarray}
Denote
	\begin{eqnarray*}
		S=\begin{pmatrix}
		(S_i^j)^{\xi\xi}  & (S_i^j)^{\xi\eta}\\
        (S_i^j)^{\eta\xi} & (S_i^j)^{\eta\eta}
		\end{pmatrix}_{i,j\geq1}
	\end{eqnarray*}
and $B,F$ similarly. Then by the relationship \eqref{20246221623}, the system (\ref{20245262120}) is decoupled into
	\begin{eqnarray} -\partial_{\omega}(S_i^j)^{\xi\xi}+\textbf{i}(Q_i^i)(S_i^j)^{\xi\xi}+\textbf{i}(S_i^j)^{\xi\xi}({}^tQ_j^j)-(B_i^j)^{\xi\xi}+\Gamma_N(F_i^j)^{\xi\xi}=0,\label{20245262205}\\
-\partial_{\omega}(S_i^j)^{\xi\eta}+\textbf{i}(Q_i^i)(S_i^j)^{\xi\eta}-\textbf{i}(S_i^j)^{\xi\eta}(Q_j^j)-(B_i^j)^{\xi\eta}+\Gamma_N(F_i^j)^{\xi\eta}=0,\label{20245262206}\\
-\partial_{\omega}(S_i^j)^{\eta\xi}-\textbf{i}({}^tQ_i^i)(S_i^j)^{\eta\xi}+\textbf{i}(S_i^j)^{\eta\xi}({}^tQ_j^j)-(B_i^j)^{\eta\xi}+\Gamma_N(F_i^j)^{\eta\xi}=0,\label{20245262207}\\
-\partial_{\omega}(S_i^j)^{\eta\eta}-\textbf{i}({}^tQ_i^i)(S_i^j)^{\eta\eta}-\textbf{i}(S_i^j)^{\eta\eta}(Q_j^j)-(B_i^j)^{\eta\eta}+\Gamma_N(F_i^j)^{\eta\eta}=0.\label{20245262208}
	\end{eqnarray}
Under the following non-resonant condition
	\begin{eqnarray}
		|\langle k,\omega(\rho)\rangle+\Omega_{i,\iota_1}(\rho)+\Omega_{j,\iota_2}(\rho)|\geq \kappa (i+j)~~~\text{for}~~~|k|\leq N,
	\end{eqnarray}
the solutions of the homological equations (\ref{20245262205}) and (\ref{20245262208}) can be obtained with
$$(B_i^j)^{\xi\xi}=(B_i^j)^{\eta\eta}=0.$$
Since they are much simpler than (\ref{20245262206}) and (\ref{20245262207}), the details are omitted.
Notice that (\ref{20245262206}) and (\ref{20245262207}) are conjugated, so in the following we only give the procedure of solving (\ref{20245262206}) under the following non-resonant condition
	\begin{eqnarray}\label{20245271501}
		|\langle k,\omega(\rho)\rangle-\Omega_{i,\iota_1}(\rho)+\Omega_{j,\iota_2}(\rho)|\geq \kappa (1+|i-j|)
	\end{eqnarray}
for $|k|\leq N$	and $|k|+|i-j|\neq 0$.
Writing in Fourier coefficients, (\ref{20245262206}) reads
	\begin{eqnarray}\label{20245271444}
		-\textbf{i}\langle k,\omega\rangle\hat{S}_i^j(k)+\textbf{i}Q_i^i\hat{S}_i^j(k)-\textbf{i}\hat{S}_i^j(k)Q_j^j-B_i^j+\hat{F}_i^j(k)=0,~~~|k|\leq N,
	\end{eqnarray}
	where we have suppressed the upper index $\xi\eta$.

For the case $k=0$ and $i=j$, set
	\begin{eqnarray}\label{2024652052}
		B_j^j=\hat{F}_j^j(0) ~~~\text{and}~~~ \hat{S}_j^j(0)=0;
	\end{eqnarray}
otherwise, $|k|+|i-j|\neq 0$, set $B_i^j=0$. Thus, the matrix $B_{\zeta\zeta}$ is on normal form and for $\nu=0,1$, we have $\partial_{\rho}^{\nu}B_{\zeta\zeta}\in\mathcal{M}_{s,\beta}$ with the estimates
	\begin{eqnarray}\label{2024661532}
		|\partial_{\rho}^{\nu}B_{\zeta\zeta}|_{s,\beta}\leq |\partial_{\rho}^{\nu}\hat{f}_{\zeta\zeta}(0)|_{s,\beta}.
	\end{eqnarray}

For the case $|k|+|i-j|\neq 0$, in view of \eqref{20240710-1}, multiplied by ${}^tP_i$ and $P_j$ in the left and right sides of (\ref{20245271444}) respectively, the equation turns into
	\begin{eqnarray}\label{20240710-2}
		-\langle k,\omega\rangle\hat{S'}_i^j(k)+ D_i\hat{S'}_i^j(k)-\hat{S'}_i^j(k)D_j=\textbf{i}\hat{F'}_i^j(k)
	\end{eqnarray}
with
\begin{eqnarray}\label{20245271517}
\hat{S'}_i^j(k)={}^tP_i\hat{S}_i^j(k)P_j,~~~~~~ \hat{F'}_i^j(k)={}^tP_i\hat{F}_i^j(k)P_j.
	\end{eqnarray}
The equation \eqref{20240710-2} can be formally solved term by term
	\begin{eqnarray}\label{20245271502}	\hat{S'}_{(i,\iota_1)}^{(j,\iota_2)}(k)=\textbf{i}\frac{\hat{F'}_{(i,\iota_1)}^{(j,\iota_2)}(k)}{-\langle k,\omega\rangle+\Omega_{i,\iota_1}-\Omega_{j,\iota_2}},~~~~~~|k|\leq N,~~\iota_1=1,2,\cdots,d_i,~~\iota_2=1,2,\cdots,d_j.
	\end{eqnarray}
	In view of (\ref{20245271500}) and (\ref{20245271501}), applying Lemma \ref{20246251433} in Appendix to (\ref{20245271502}), we get $\hat{S'}(k)\in\mathcal{M}_{s,\beta}^+$ and
	\begin{eqnarray}\label{20245271536}
		|\hat{S'}(k)|_{s,\beta+}\lessdot \frac{N^{\frac{d^*}{2}}}{\kappa^{\frac{d^*}{2\beta}+1}}|\hat{F'}(k)|_{s,\beta}.
	\end{eqnarray}
Recalling (\ref{20245271514}) and (\ref{20245271517}), we get
	\begin{eqnarray}
		|\hat{S}_{\zeta\zeta}(k)|_{s,\beta+}\lessdot \frac{N^{\frac{d^*}{2}}}{\kappa^{\frac{d^*}{2\beta}+1}}|\hat{f}_{\zeta\zeta}(k)|_{s,\beta}.
	\end{eqnarray}
Furthermore, for any $\theta\in\mathbb{T}_{\sigma'}^n$,
	\begin{eqnarray}\label{20240712-4} |S_{\zeta\zeta}(\theta;\rho)|_{s,\beta+}\lessdot\frac{N^{\frac{d^*}{2}}}{\kappa^{\frac{d^*}{2\beta}+1}(\sigma-\sigma')^n}		\sup_{\theta\in\mathbb{T}_{\sigma}^n}|{f}_{\zeta\zeta}(\theta)|_{s,\beta}.
	\end{eqnarray}

For the estimate of $\partial_{\rho}S_{\zeta\zeta}$, we differentiate (\ref{20245271444}) and obtain for $|k|+|i-j|\neq0$,
	\begin{eqnarray}
		-\textbf{i}\langle k,\omega\rangle\partial_{\rho}\hat{S}_i^j(k;\rho)+\textbf{i} Q_i^i\partial_{\rho}\hat{S}_i^j(k;\rho)-\textbf{i}\partial_{\rho}\hat{S}_i^j(k;\rho)Q_j^j+\hat{R}_i^j(k;\rho)=0
	\end{eqnarray}
	with
$$\hat{R}_i^j(k;\rho):=\partial_{\rho}\hat{F}_i^j(k;\rho)-\textbf{i}\langle k,\partial_{\rho}\omega\rangle\hat{S}_i^j(k;\rho)+ \textbf{i}\partial_{\rho}Q_i^i\hat{S}_i^j(k;\rho)-\textbf{i}\hat{S}_i^j(k;\rho)\partial_{\rho}Q_j^j.$$
By (\ref{2024672045}), we have
	\begin{eqnarray}
		\|\hat{R}_i^j(k;\rho)\|\leq \|\partial_{\rho}\hat{F}_i^j(k;\rho)\|+(|k|M_{\omega}+2M_{\Omega})\|\hat{S}_i^j(k;\rho)\|,
	\end{eqnarray}
which implies
$\hat{R}(k;\rho)\in\mathcal{M}_{s,\beta}$ and
	\begin{eqnarray}\label{20245271537}
		|\hat{R}(k;\rho)|_{s,\beta}\leq |\partial_{\rho}\hat{F}(k;\rho)|_{s,\beta}+(|k|M_{\omega}+2M_{\Omega})|\hat{S}(k;\rho)|_{s,\beta}.
	\end{eqnarray}
In the same manner,	by Lemma \ref{20246251433} in Appendix, one has
	\begin{eqnarray}
		\begin{aligned}
			|\partial_{\rho}\hat{S}(k;\rho)|_{s,\beta+}&\lessdot \frac{N^{\frac{d^*}{2}}}{\kappa^{\frac{d^*}{2\beta}+1}}|\hat{R}(k;\rho)|_{s,\beta}\\
&\lessdot\frac{N^{\frac{d^*}{2}}}{\kappa^{\frac{d^*}{2\beta}+1}}|\partial_{\rho}\hat{F}(k;\rho)|_{s,\beta}
+\frac{N^{{d^*}}(1+|k|)}{\kappa^{\frac{d^*}{\beta}+2}}|\hat{F}(k;\rho)|_{s,\beta},
		\end{aligned}
	\end{eqnarray}
where (\ref{20245271517}), (\ref{20245271536}), (\ref{20245271537}) are used in the second inequality.
Therefore, for any $\theta\in\mathbb{T}_{\sigma'}^n$,
	\begin{eqnarray}\label{20240712-5}
    \begin{aligned}
		|\partial_{\rho} S_{\zeta\zeta}(\theta;\rho)|_{s,\beta+}\lessdot& \frac{N^{\frac{d^*}{2}}}{\kappa^{\frac{d^*}{2\beta}+1}(\sigma-\sigma')^{n}}\sup_{\theta\in\mathbb{T}_{\sigma}^n}|\partial_{\rho}{f}_{\zeta\zeta}(\theta)|_{s,\beta}\\
&+\frac{N^{{d^*}}}{\kappa^{\frac{d^*}{\beta}+2}(\sigma-\sigma')^{n+1}}\sup_{\theta\in\mathbb{T}_{\sigma}^n}|{f}_{\zeta\zeta}(\theta)|_{s,\beta}.
	\end{aligned}
    \end{eqnarray}
Summing the estimates \eqref{20240712-4} and \eqref{20240712-5}, for any $\theta\in\mathbb{T}_{\sigma'}^n$, we get
	\begin{eqnarray}\label{20240712-6}
			\sup_{\nu=0,1}|\partial_{\rho}^{\nu} S_{\zeta\zeta}(\theta;\rho)|_{s,\beta+}\lessdot \frac{N^{{d^*}}}{\kappa^{\frac{d^*}{\beta}+2}(\sigma-\sigma')^{n+1}}{\mu^{-2}}[f]_{\sigma,\mu,\mathcal{D}}^{s,\beta}.
	\end{eqnarray}

Finally, to sum up, we conclude that
\begin{eqnarray}\label{2024652148}
{[S]}_{\sigma',\mu,\mathcal{D}}^{s,\beta+}\lessdot
\frac{N^{{d^*}}}{\kappa^{\frac{d^*}{\beta}+2}(\sigma-\sigma')^{n+1}}[f]_{\sigma,\mu,\mathcal{D}}^{s,\beta}.
\end{eqnarray}

\subsection{New Hamiltonian}
Denoting the new Hamiltonian $H\circ\Phi$ in (\ref{20243211944}) still by $H$, we have
	\begin{eqnarray*}
		H=h_++f_+
	\end{eqnarray*}
with the new normal form $h_+:=h+\hat{h}$ and the new perturbation
	\begin{eqnarray}\label{2024652123}
	f_+:=(1-\Gamma_N)f^T+\int_{0}^{1}\{(1-t)(\hat{h}-\Gamma_Nf^T)+f^T,S\}\circ X_S^tdt+(f-f^T)\circ X_S^1.
	\end{eqnarray}
This subsection is devoted to estimating $\hat{h}$ and $f_+$.

	$\bullet$~~~The new normal form
	
	In view of (\ref{2024661530}), since the term $[\![f_{\theta}]\!]$ is a constant and does not affect the dynamics, we omit it and rewrite
	\begin{eqnarray}
		\hat{h}=\langle [\![f_{r}]\!],r\rangle+\frac{1}{2}\langle \zeta,B_{\zeta\zeta}\zeta\rangle.
	\end{eqnarray}
By (\ref{2024661532}) and the fact that $B_{\zeta\zeta}$ is block diagonal, for $\nu=0,1$, we get
	\begin{eqnarray*}
		\begin{aligned}
			\|\partial_{\rho}^{\nu}B_{\zeta\zeta}\zeta\|_s^2&=\sum_{j\geq 1}|j^{2s}(\partial_{\rho}^{\nu}B_{\zeta\zeta}\zeta)_j|^2\\
&=\sum_{j\geq 1}j^{2s}\|\partial_{\rho}^{\nu}(B_{\zeta\zeta})_j^j\|^2|\zeta_j|^2\\
&\leq |\partial_{\rho}^{\nu} \hat{f}_{\zeta\zeta}(0)|_{s,\beta}^2\|\zeta\|_s^2\\
&\leq \mu^{-4}([f^T]_{\sigma,\mu,\mathcal{D}}^{s,\beta})^2\|\zeta\|_s^2
		\end{aligned}
	\end{eqnarray*}
and thus
	\begin{eqnarray}\label{202466947}	[\hat{h}]_{\sigma,\mu,\mathcal{D}}^{s,\beta}\leq[f^T]_{\sigma,\mu,\mathcal{D}}^{s,\beta}\lessdot[f]_{\sigma,\mu,\mathcal{D}}^{s,\beta},
	\end{eqnarray}
where Lemma \ref{20243212004} is used in the last inequality.
	
	$\bullet$~~~The new perturbation
	
	Denote the three terms in the right hand side of (\ref{2024652123}) by $f_{1+},f_{2+},f_{3+}$. Assume $0<\sigma'<\sigma$ and $0<\mu'<\frac{\mu}{2}$. By the definition of the norm $[f]_{\sigma,\mu,\mathcal{D}}^{s,\beta}$ and Lemma \ref{20243212004} in Appendix, we have
	\begin{eqnarray}\label{2024672220}
		[f_{1+}]_{\sigma',\mu,\mathcal{D}}^{s,\beta}\lessdot \frac{e^{-\frac{1}{2}(\sigma-\sigma')N}}{(\sigma-\sigma')^n}[f^T]_{\sigma,\mu,\mathcal{D}}^{s,\beta}\lessdot \frac{e^{-\frac{1}{2}(\sigma-\sigma')N}}{(\sigma-\sigma')^n}[f]_{\sigma,\mu,\mathcal{D}}^{s,\beta}.
	\end{eqnarray}

Substituting $\sigma'$ by $\frac{2\sigma+\sigma'}{3}$ in the last subsection, the final conclusion (\ref{2024652148}) becomes
\begin{eqnarray}\label{20240712-8}
{[S]}_{\frac{2\sigma+\sigma'}{3},\mu,\mathcal{D}}^{s,\beta+}\lessdot
\frac{N^{{d^*}}}{\kappa^{\frac{d^*}{\beta}+2}(\sigma-\sigma')^{n+1}}[f]_{\sigma,\mu,\mathcal{D}}^{s,\beta}.
\end{eqnarray}
For the estimates of $f_{2+}$ and $f_{3+}$, we assume that
	\begin{eqnarray}\label{2024652149}
		[f]_{\sigma,\mu,\mathcal{D}}^{s,\beta}\leq \frac{(\beta-\beta')(\sigma-\sigma')^{n+2}(\mu')^2}{\tilde{c}_1}\cdot\frac{\kappa^{2+\frac{d^*}{\beta}}}{N^{d^*}}
	\end{eqnarray}
with a sufficiently large constant $\tilde{c}_1$, such that the inequalities  (\ref{20240712-8}) and (\ref{2024652149}) imply
	\begin{eqnarray}\label{202466938}
		[S]_{\frac{2\sigma+\sigma'}{3},\mu,\mathcal{D}}^{s,\beta+}\leq \frac{(\mu')^2(\sigma-\sigma')(\beta-\beta')}{c_1},
	\end{eqnarray}
where $c_1$ is the constant in \eqref{20243141942}. Denote
$$g_t:=(1-t)(\hat{h}-\Gamma_N f^T)+f^T=(1-t)\hat{h}+(1-t)(1-\Gamma_N)f^T+tf^T.$$
By (\ref{202466947}) and Lemma \ref{20243212004} in Appendix, we have
	\begin{eqnarray}\label{202466955}		[g_t]_{\frac{2\sigma+\sigma'}{3},\mu,\mathcal{D}}^{s,\beta}\lessdot\Big(1+\frac{e^{-\frac{1}{6}(\sigma-\sigma')N}}{(\sigma-\sigma')^n}\Big)[f^T]_{\sigma,\mu,\mathcal{D}}^{s,\beta}\lessdot [f]_{\sigma,\mu,\mathcal{D}}^{s,\beta},
	\end{eqnarray}
where in the second inequality we assume
	\begin{eqnarray}\label{20240712-7}		N>\frac{6n}{\sigma-\sigma'}\ln\frac{1}{\sigma-\sigma'}.
	\end{eqnarray}
By (\ref{20240712-8}), (\ref{202466955}) and Lemma \ref{20243212055}, we have
\begin{eqnarray}\label{20240712-9} [\{g_t,S\}]_{\frac{\sigma+2\sigma'}{3},\mu,\mathcal{D}}^{s,\beta}
&\lessdot&\frac{1}{(\sigma-\sigma')}{\mu}^{-2}[g_t]_{\frac{2\sigma+\sigma'}{3},\mu,\mathcal{D}}^{s,\beta}[S]_{\frac{2\sigma+\sigma'}{3},\mu,\mathcal{D}}^{s,\beta+}\nonumber\\
&\lessdot&\frac{N^{{d^*}}\mu^{-2}}{\kappa^{\frac{d^*}{\beta}+2}(\sigma-\sigma')^{n+2}}([f]_{\sigma,\mu,\mathcal{D}}^{s,\beta})^2.		\end{eqnarray}
By Lemma \ref{20243212004} in Appendix, we have
	\begin{eqnarray}\label{2024672201}
		[f-f^T]_{\sigma,2\mu',\mathcal{D}}^{s,\beta}\lessdot \Big(\frac{\mu'}{\mu}\Big)^3[f]_{\sigma,\mu,\mathcal{D}}^{s,\beta}.
	\end{eqnarray}
By \eqref{202466938}, \eqref{20240712-9}, \eqref{2024672201} and Lemma \ref{20243212033}, we get
	\begin{eqnarray}\label{2024672222}
[f_{2+}]_{\sigma',\mu',\mathcal{D}}^{s,\beta'}&\lessdot& \frac{N^{{d^*}}\mu^{-2}}{\kappa^{\frac{d^*}{\beta}+2}(\sigma-\sigma')^{n+2}}([f]_{\sigma,\mu,\mathcal{D}}^{s,\beta})^2,\\
~[f_{3+}]_{\sigma',\mu',\mathcal{D}}^{s,\beta'}&\lessdot& \Big(\frac{\mu'}{\mu}\Big)^3[f]_{\sigma,\mu,\mathcal{D}}^{s,\beta}.\label{2024672221}
	\end{eqnarray}

Finally, by (\ref{2024672220}), (\ref{2024672222}) and (\ref{2024672221}), we get
	\begin{eqnarray}\label{20240712-10}
			[f_{+}]_{\sigma',\mu',\mathcal{D}}^{s,\beta'}\lessdot \Big(\frac{e^{-\frac{1}{2}(\sigma-\sigma')N}}{(\sigma-\sigma')^n}+ \frac{N^{{d^*}}\mu^{-2}}{\kappa^{\frac{d^*}{\beta}+2}(\sigma-\sigma')^{n+2}}[f]_{\sigma,\mu,\mathcal{D}}^{s,\beta}+\Big(\frac{\mu'}{\mu}\Big)^3\Big)[f]_{\sigma,\mu,\mathcal{D}}^{s,\beta}.
	\end{eqnarray}

\section{Iteration and Convergence}
	Now we give the precise set-up of iteration parameters. Let $m \geq 0$ be the $m$-th KAM step.
	\begin{itemize}
\item[]		
$M_{\omega,m}={M_{\omega,0}}(2-2^{-m})$, which is used to dominate the norm of tangential frequencies,

\item[]
$L_m=L_{0}(2-2^{-m})$, which is used to dominate the inverse of tangential frequencies,

\item[]
$M_{\Omega,m}=M_{\Omega,0}(2-2^{-m})$, which is used to dominate the norm of normal frequencies,

\item[]
$\sigma_{m}=\frac{\sigma_0}{2}(1+2^{-m})$, which is used to dominate the width of the angle variable $\theta$,

\item[]
$\beta_{m}=\frac{\beta_0}{2}(1+2^{-m})$, which is used to quantify the regularizing effect of Hessian matrix.
\end{itemize}
Denote $E_{m}^0=c_1c_2(\beta_m-\beta_{m+1})^{-1}(\sigma_m-\sigma_{m+1})^{-\alpha(n+d^*+2)}$ and $E_{m}=c_2^3E_{m}^0$, where $c_1$ is the constant in \eqref{20243141942}, and $c_2$ triples the largest of all the constants being implicit in the notation $\lessdot$ in the last section. For fixed $0<\mu_0\leq1$ and $0<\varepsilon_0<1$, define $\mu_m$ and $\varepsilon_m$ inductively
	\begin{itemize}
\item[]		
$\mu_{m+1}=\mu_m\big(\frac{\varepsilon_m}{\mu_m^2}\big)^{\frac{1}{4}}(E_m^0)^{\frac{1}{3}}$, which is used to dominate the size of $r$ and $\zeta$,

\item[]
$\frac{\varepsilon_{m+1}}{\mu_{m+1}^2}=\big(\frac{\varepsilon_m}{\mu_m^2}\big)^{\frac{5}{4}}(E_m)^{\frac{1}{3}}$, which is used to dominate the size of perturbation.
\end{itemize}
Denote $\mathcal{O}_m=\mathcal{O}^s(\sigma_{m},\mu_{m})$. Moreover,
\begin{itemize}
        \item[]
        $N_{m}=\frac{2\ln\big(\frac{\mu_m^2}{\varepsilon_m}\big)}{\sigma_{m}-\sigma_{m+1}}$, which is the length of the truncation of Fourier series,
        \item[]
        $\kappa_{m}=\big(\frac{\varepsilon_m}{\mu_m^2}\big)^{\frac{1}{8(2+\frac{d^*}{\beta_m})}}$, which is used to dominate the measure of removed parameters.
\end{itemize}

\subsection{Iteration lemma}	
	\begin{lem}\label{20246101129}
		Assume that
		\begin{eqnarray}\label{20246101028}
			\frac{\varepsilon_0}{\mu_0^2}\leq \frac{\gamma_0}{16}\prod_{k=0}^{\infty}E_{k}^{-\frac{1}{3(\frac{5}{4})^{k+1}}},
		\end{eqnarray}
		where $\gamma_0$ is small and depends on $n,d^*,M_{\omega,0},L_0,M_{\Omega,0},\beta_0,\sigma_0,\mu_0$. Suppose the Hamiltonian $h_m+f_m$ is regular on $\mathcal{O}_m\times\mathcal{D}_m$, where
$h_m=\langle \omega_m,r\rangle+\frac{1}{2}\langle \zeta,A_m\zeta\rangle$
is on normal form with $\omega_m$ and
$$(A_m)_j^j(\rho):=\lambda_jI+(\tilde{A}_m)_j^j(\rho),~~~~~~j\geq1$$
satisfying
		\begin{eqnarray}\label{20246111949}				&&\sup_{\nu=0,1,\rho\in\mathcal{D}_m}|\partial_{\rho}^{\nu}\omega_m(\rho)|\leq M_{\omega,m},\\
&&\sup_{\rho\in\omega_m(\mathcal{D}_m)}|\partial_{\rho}\omega_m^{-1}(\rho)|\leq L_m,\label{20240714-1}\\
&&\sup_{\nu=0,1,\rho\in\mathcal{D}_m}\|\partial_{\rho}^{\nu}(\tilde{A}_m)_j^j(\rho)\|\leq \frac{M_{\Omega,m}}{j^{\beta_m}},\label{20246111950}
		\end{eqnarray}
and $f_m\in\mathcal{T}_{\sigma_m,\mu_m,\mathcal{D}_m}^{s,\beta_m}$ satisfying
		\begin{eqnarray}\label{20240719-6}
			[f_m]_{\sigma_m,\mu_m,\mathcal{D}_m}^{s,\beta_m}\leq \varepsilon_m.
		\end{eqnarray}
Let $Q_m$ be the complex valued normal form matrix associated to $A_m$, and for every $j\geq1$, denote the spectra of $(Q_m)_j^j$ by $\Omega_{m,j,\iota}(\rho)$, $1\leq\iota\leq d_j$. Then there exist a closed subset
		\begin{eqnarray}\label{20246191530}
			\mathcal{D}_{m+1}=\mathcal{D}_m\backslash \bigcup_{|k|\leq N_{m}}\mathfrak{R}_{k}^{m}(\kappa_{m}),
		\end{eqnarray}
		 where		
		\begin{eqnarray*}
			\mathfrak{R}_{k}^{m}(\kappa_{m})=\mathfrak{R}_{k}^{m,0}\cup \mathfrak{R}_{k}^{m,1}\cup \mathfrak{R}_{k}^{m,20}\cup \mathfrak{R}_{k}^{m,11}
		\end{eqnarray*}
		with		
		\begin{eqnarray}\label{20246191517}
			\begin{aligned}
				&\mathfrak{R}_{k}^{m,0}=\Big\{\rho\in\mathcal{D}_m:~|\langle k,\omega_{m}(\rho)\rangle|<\kappa_{m}\Big\},\\
&\mathfrak{R}_{k}^{m,1}=\bigcup_{(j,\iota)\in\mathfrak{L}}\Big\{\rho\in\mathcal{D}_m:~|\langle k,\omega_{m}(\rho)\rangle+\Omega_{m,j,\iota}(\rho)|< \kappa_{m} j\Big\},\\
&\mathfrak{R}_{k}^{m,20}=\bigcup_{(i,\iota_1),(j,\iota_2)\in\mathfrak{L}}\Big\{\rho\in\mathcal{D}_m:~|\langle k,\omega_{m}(\rho)\rangle+\Omega_{m,i,\iota_1}(\rho)+\Omega_{m,j,\iota_2}(\rho)|<\kappa_{m}(i+j)\Big\},\\
&\mathfrak{R}_{k}^{m,11}=\bigcup_{(i,\iota_1),(j,\iota_2)\in\mathfrak{L}\atop
|k|+|i-j|\neq 0}\Big\{\rho\in\mathcal{D}_m:~|\langle k,\omega_{m}(\rho)\rangle-\Omega_{m,i,\iota_1}(\rho)+\Omega_{m,j,\iota_2}(\rho)|<\kappa_{m}(1+|i-j|)\Big\},
			\end{aligned}
		\end{eqnarray}
and a $C^1$-Whitney smooth family of real symplectic transformations $\Phi_{m+1}:\mathcal{O}_{m+1}\times\mathcal{D}_{m+1}\to \mathcal{O}_m$,	such that %
		\begin{eqnarray}
			(h_m+f_m)\circ\Phi_{m+1}=h_{m+1}+f_{m+1}=\langle \omega_{m+1},r\rangle+\frac{1}{2}\langle \zeta,A_{m+1}\zeta\rangle +f_{m+1},
		\end{eqnarray}
where the estimate		
		\begin{eqnarray}\label{20246101052}
			[h_{m+1}-h_m]_{\sigma_m,\mu_m,\mathcal{D}_m}^{s,\beta_m}\leq c_2\varepsilon_m
		\end{eqnarray}
		holds true and the same assumptions as above are satisfied with $m+1$ in place of $m$.
	\end{lem}

	\begin{proof}
According to the induction formula of $\varepsilon_m$, namely, $\frac{\varepsilon_{m+1}}{\mu_{m+1}^2}=\big(\frac{\varepsilon_m}{\mu_m^2}\big)^{\frac{5}{4}}(E_m)^{\frac{1}{3}}$, we get for $m\geq1$,
\begin{eqnarray}\label{20240715-1}
\frac{\varepsilon_m}{\mu_m^2}=\Big(\frac{\varepsilon_0}{\mu_0^2}\prod_{k=0}^{m-1}E_k^{\frac{1}{3(\frac{5}{4})^{k+1}}}\Big)^{(\frac{5}{4})^m}.
\end{eqnarray}
Then by (\ref{20246101028}), (\ref{20240715-1}) and the fact $E_m$ is increasing with respect to $m$, we get
\begin{eqnarray}\label{20240715-2}
\begin{aligned}
\frac{\varepsilon_m}{\mu_m^2}&\leq
\Big(\frac{\gamma_0}{16}\prod_{k=m}^{\infty}E_k^{-\frac{1}{3(\frac{5}{4})^{k+1}}}\Big)^{(\frac{5}{4})^m}\\
&\leq\big(\frac{\gamma_0}{16}\big)^{(\frac{5}{4})^m}
\Big(\prod_{k=m}^{\infty}E_m^{-\frac{1}{3(\frac{5}{4})^{k+1}}}\Big)^{(\frac{5}{4})^m}\\
&=\big(\frac{\gamma_0}{16}\big)^{(\frac{5}{4})^m}E_m^{-\frac{4}{3}},
\end{aligned}
\end{eqnarray}
which implies for $\gamma_0$ small enough,		
		\begin{eqnarray}\label{20246101040} \frac{\varepsilon_m}{\mu_m^2}E_m^{\frac{4}{3}}\leq\frac{\gamma_0}{2^{4m+4}}.
		\end{eqnarray}
Furthermore, by \eqref{20246101040}, we know
\begin{eqnarray}\label{20240719-1}
2\ln\big(\frac{\mu_m^2}{\varepsilon_m}\big)<\big(\frac{\varepsilon_m}{\mu_m^2}\big)^{-\frac{1}{4\alpha(n+d^*+2)}}
\end{eqnarray}
and
\begin{eqnarray}\label{20240719-2}
(\sigma_{m}-\sigma_{m+1})^{-1}<\big(E_m\big)^{\frac{1}{\alpha(n+d^*+2)}}
<\big(\frac{\varepsilon_m}{\mu_m^2}\big)^{-\frac{3}{4\alpha(n+d^*+2)}}.
\end{eqnarray}
By \eqref{20240719-1} and \eqref{20240719-2}, we get
		\begin{eqnarray}\label{20246132026}	N_m^{n+d^*+2}=\Big(\frac{2\ln(\frac{\mu_m^2}{\varepsilon_m})}{\sigma_{m}-\sigma_{m+1}}\Big)^{n+d^*+2}
<\big(\frac{\varepsilon_m}{\mu_m^2}\big)^{-\frac{1}{\alpha}},
		\end{eqnarray}
which imply
		\begin{eqnarray}\label{20240719-3}	N_m^{d^*}<\big(\frac{\varepsilon_m}{\mu_m^2}\big)^{-\frac{1}{8}}.
		\end{eqnarray}
Then, by (\ref{20246101040}), (\ref{20240719-3}) and the facts
\begin{eqnarray}
&&\frac{\varepsilon_m}{\mu_{m+1}^2}E_m^0=\big(\frac{\varepsilon_m}{\mu_m^2}\big)^{\frac{1}{2}}(E_m^0)^{\frac{1}{3}},\\
&&\kappa_{m}^{-2-\frac{d^*}{\beta_{m}}}=\big(\frac{\varepsilon_m}{\mu_m^2}\big)^{-\frac{1}{8}},\label{20240721-1}
\end{eqnarray}
we get
		\begin{eqnarray}\label{20246121113}	\frac{\varepsilon_m}{\mu_{m+1}^2}\kappa_{m}^{-2-\frac{d^*}{\beta_{m}}}N_m^{d^*}E_m^0
\leq\big(\frac{\varepsilon_m}{\mu_m^2}\big)^{\frac{1}{4}}(E_m^0)^{\frac{1}{3}}<1,
		\end{eqnarray}
which implies that for each $m\geq 0$, the smallness condition \eqref{2024652149} is satisfied. Moreover, by (\ref{20246101040}), we get
\begin{eqnarray}\label{20240719-4}
N_m(\sigma_{m}-\sigma_{m+1})=2\ln\big(\frac{\mu_m^2}{\varepsilon_m}\big)>\frac{8}{3}\ln{E_m}>6n\ln{\frac{1}{\sigma_{m}-\sigma_{m+1}}},
\end{eqnarray}
which implies that for each $m\geq 0$, the condition \eqref{20240712-7} is satisfied.

Therefore, there exists a transformation $\Phi_{m+1}:\mathcal{O}_{m+1}\times\mathcal{D}_{m+1}\to \mathcal{O}_m$, taking $h_m+f_m$ into $h_{m+1}+f_{m+1}$. Firstly, by \eqref{202466947}, we get the estimate (\ref{20246101052}) and
		\begin{eqnarray}\label{20246111948}	&&\sup_{\nu=0,1,\rho\in\mathcal{D}_m}|\partial_{\rho}^{\nu}(\omega_{m+1}(\rho)-\omega_{m}(\rho))|
\leq\frac{\varepsilon_m}{\mu_m^2},\\
&&\sup_{\nu=0,1,\rho\in\mathcal{D}_m}\|\partial_{\rho}^{\nu}(A_{m+1}(\rho)-A_m(\rho))_j^j\|\leq \frac{\varepsilon_m}{\mu_m^2 j^{\beta_{m}}}.\label{20240719-5}
		\end{eqnarray}
		From (\ref{20246101040}), (\ref{20246111948}) and \eqref{20240719-5}, for $\gamma_0$ small enough, the estimates (\ref{20246111949})-(\ref{20246111950})  are satisfied with $m+1$ replacing $m$.
Next, we give the estimate of the new perturbation. In view of (\ref{20240712-10}), by the following facts
\begin{eqnarray}\label{20247151940}
\begin{aligned}	&\frac{e^{-\frac{1}{2}(\sigma_m-\sigma_{m+1})N_m}}{(\sigma_m-\sigma_{m+1})^n}
=\frac{1}{(\sigma_m-\sigma_{m+1})^n}\cdot\frac{\varepsilon_{m}}{\mu_{m}^2}\leq\frac{\varepsilon_{m}}{\mu_{m}^2}E_m^0,\\ &\Big(\frac{\mu_{m+1}}{\mu_{m}}\Big)^3
=\Big(\frac{\varepsilon_m}{\mu_m^2}\Big)^{\frac{3}{4}}E_m^0,\\ &\frac{N_m^{{d^*}}\mu_m^{-2}}{\kappa_m^{2+\frac{d^*}{\beta_m}}(\sigma_m-\sigma_{m+1})^{n+2}}[f_m]_{\sigma_m,\mu_m,\mathcal{D}_m}^{s,\beta_m}
\leq\Big(\frac{\varepsilon_m}{\mu_ m^2}\Big)^{\frac{3}{4}}E_m^0,
\end{aligned}
\end{eqnarray}
we have
		\begin{eqnarray}
			\begin{aligned}
~[f_{m+1}]_{\sigma_{m+1},\mu_{m+1},\mathcal{D}_{m+1}}^{s,\beta_{m+1}}
&\leq\frac{c_2}{3}\Big(\frac{\varepsilon_m}{\mu_m^2}E_m^0+\Big(\frac{\varepsilon_m}{\mu_m^2}\Big)^{\frac{3}{4}}E_m^0+\Big(\frac{\varepsilon_m}{\mu_m^2}\Big)^{\frac{3}{4}}E_m^0\Big)\cdot\varepsilon_m\\
&\leq c_2\Big(\frac{\varepsilon_m}{\mu_m^2}\Big)^{\frac{3}{4}}E_m^0\varepsilon_m\\
&=\varepsilon_{m+1},
			\end{aligned}
		\end{eqnarray}
which is \eqref{20240719-6} with $m+1$ replacing $m$.
	\end{proof}

\subsection{Convergence}
	Now we are in a position to prove the KAM theorem. To apply the
	iterative lemma with $m=0$, set
	\begin{eqnarray}
		h_0=h,~f_0=f,~M_{\omega,0}=M_{\omega},~L_0=L,~M_{\Omega,0}=M_{\Omega},\\
~\sigma_0=\sigma,~\mu_0=\mu,~\beta_0=\beta,~\mathcal{D}_0=\mathcal{D},~\varepsilon_0=\varepsilon=[f]_{\sigma_0,\mu_0,\mathcal{D}_0}^{s_0,\beta_0}.
	\end{eqnarray}
	In view of \eqref{20240709-1},\eqref{20240709-2} and \eqref{20240709-3}, we know that (\ref{20246111949})-(\ref{20246111950}) with $m=0$ are satisfied. There exists $\varepsilon^*>0$ depending on $n$, $s$, $b_0$, $b_1$, $c^*$, $d^*$, $M_{\omega}$, $L$, $M_{\Omega}$, $\beta$, $\sigma$, $\mu$ such that if $0<\varepsilon<\varepsilon^*$, the assumption (\ref{20246101028}) is satisfied. Hence, the iterative lemma applies, and we obtain a decreasing sequence of domain $\mathcal{O}_m\times\mathcal{D}_m$ and a sequence of transformations $\Phi_{m+1}$ defined on $\mathcal{O}^s(\sigma_{m+1},\mu_{m+1})$ for each $m\geq 0$. Denoting
$$\Phi_{m+1}=\Phi_{S_{m+1}}^1:~(\theta_{m+1},r_{m+1},\zeta_{m+1})\mapsto(\theta_{m},r_{m},\zeta_{m}),$$ %
then by \eqref{20240712-8} we have
\begin{eqnarray}\label{20240720-1}
\begin{aligned}
~[S_{m+1}]_{\frac{2\sigma_{m}+\sigma_{m+1}}{3},\mu_{m},\mathcal{D}_{m+1}}^{s,\beta_{m}+}
&\leq\frac{c_2N_m^{{d^*}}}{3\kappa_m^{2+\frac{d^*}{\beta_m}}(\sigma_m-\sigma_{m+1})^{n+1}}\varepsilon_m\\
&\leq(\sigma_m-\sigma_{m+1})\Big(\frac{\varepsilon_m}{\mu_m^2}\Big)^{-\frac{1}{4}}\varepsilon_mE_m^{\frac{1}{4}},
\end{aligned}
\end{eqnarray}
where in the last inequality we use \eqref{20240719-3}, \eqref{20240721-1} and the fact
$$\frac{c_2}{3}(\sigma_m-\sigma_{m+1})^{-n-2}<E_m^{\frac{1}{4}}.$$
Furthermore, by \eqref{20243172034} and \eqref{20240720-1}, the transformation
$\Phi_{m+1}$ has an analytic extension to $\mathcal{O}^s(\sigma_{m+1},\mu_0)$ and verifies on this set
		\begin{eqnarray}\label{20246121114}
			\begin{aligned}
&|\theta_{m+1}-\theta_m|\leq (\sigma_m-\sigma_{m+1})\Big(\frac{\varepsilon_m}{\mu_m^2}\Big)^{\frac{3}{4}}E_m^{\frac{1}{4}},\\
&|r_{m+1}-r_m|\leq12\Big(1+\frac{\mu_0}{\mu_m}+2\Big(\frac{\mu_0}{\mu_m}\Big)^2\Big)\Big(\frac{\varepsilon_m}{\mu_m^2}\Big)^{-\frac{1}{4}}\varepsilon_mE_m^{\frac{1}{4}},\\
&\|\zeta_{m+1}-\zeta_m\|_{s}\leq (\sigma_{m}-\sigma_{m+1})\Big(2\mu_m+\frac{2^{s+5}}{\beta_m-\beta_{m+1}}\mu_0\Big)\Big(\frac{\varepsilon_m}{\mu_m^2}\Big)^{\frac{3}{4}}E_m^{\frac{1}{4}}.
			\end{aligned}
		\end{eqnarray}
By (\ref{20246121114}), there exists a constant $C>0$ only depending on $s$, $\beta_0$ such that
	\begin{eqnarray}\label{20240721-2}
\sup_{\mathcal{O}^s(\sigma_{m+1},\mu_0)}\|\Phi_{m+1}-id\|_s
\leq C\Big(\frac{\varepsilon_m}{\mu_m^2}\Big)^{\frac{3}{4}}E_m^{\frac{1}{4}},
	\end{eqnarray}
which implies
	\begin{eqnarray}\label{20246131658}
\sup_{\mathcal{O}^s(\sigma_{m+1},\mu_0)}\|\Phi_{m+1}-id\|_s
\leq C\Big(\frac{\varepsilon_m}{\mu_m^2}E_m^{\frac{4}{3}}\Big)^{\frac{3}{4}}
<\frac{\gamma_0^{\frac{1}{2}}}{2^{3m+3}},
	\end{eqnarray}
where in the last inequality we use \eqref{20246101040} and the fact $\gamma_0$ is sufficiently small. By Cauchy estimate, we have
	\begin{eqnarray}\label{20246131702}	\sup_{\mathcal{O}^s(\frac{\sigma_{m+1}+\sigma_{m+2}}{2},\frac{3\mu_0}{4})}\|D\Phi_{m+1}-I\|_{\mathcal{L}(\mathbb{C}^{2n}\times Y_s,\mathbb{C}^{2n}\times Y_s)}
<\frac{\gamma_0^{\frac{1}{3}}}{2^{2m+2}}.
	\end{eqnarray}

Now for every $m\geq 1$, let us denote $\Phi^m=\Phi_1\circ\Phi_2\circ\cdots\circ\Phi_m$. By (\ref{20246131658}) and the fact $\gamma_0$ is sufficiently small, the transformations $\Phi_m$ are from $\mathcal{O}^s(\sigma_{m+1},\frac{\mu_0}{2}(1+2^{-m-1}))$ to $\mathcal{O}^s(\sigma_{m},\frac{\mu_0}{2}(1+2^{-m}))$ and $\Phi^m$ are from $\mathcal{O}^s(\sigma_{m+1},\frac{\mu_0}{2}(1+2^{-m-1}))$ to $\mathcal{O}^s(\sigma_{0},\frac{3\mu_0}{4})$. By construction, the map $\Phi^m$ transforms the original Hamiltonian
$H=\langle\omega,r\rangle+\frac{1}{2}\langle\zeta,A\zeta\rangle+f$
into
	\begin{eqnarray}
		H_m=\langle\omega_m,r\rangle+\frac{1}{2}\langle\zeta,A_m\zeta\rangle+f_m
	\end{eqnarray}
	with
	\begin{eqnarray}\label{20246131722}	\omega_m=\omega+[\![\nabla_{r}f_0(0;\rho)]\!]+\cdots+[\![\nabla_{r}f_{m-1}(0;\rho)]\!],
	\end{eqnarray}
	\begin{eqnarray}\label{20246131723}
		A_m=A+(B_{\zeta\zeta})_0+\cdots+(B_{\zeta\zeta})_{m-1},
	\end{eqnarray}
	and
	\begin{eqnarray}\label{20246131724}
		[f_m]_{\sigma_{m+1},\mu_{m},\mathcal{D}_m}^{s,\beta_{m}}\leq \varepsilon_{m}.
	\end{eqnarray}
	By (\ref{20246131702}) and the chain rule, ones obtain
	\begin{eqnarray}\label{20240720-2}
		\begin{aligned}	\sup_{\mathcal{O}^s(\sigma_{m+1},\frac{\mu_0}{2}(1+2^{-m-1}))}\|D\Phi^m\|_{\mathcal{L}(\mathbb{C}^{2n}\times Y_s,\mathbb{C}^{2n}\times Y_s)}
&\leq\prod_{k=1}^{m}\sup_{\mathcal{O}^s(\sigma_{k+1},\frac{3\mu_0}{4})}\|D\Phi_k\|_{\mathcal{L}(\mathbb{C}^{2n}\times Y_s,\mathbb{C}^{2n}\times Y_s)}\\
&<\prod_{k=1}^{\infty}(1+\frac{\gamma_0^{\frac{1}{3}}}{2^{2k}})\\
&<2.		
		\end{aligned}
	\end{eqnarray}
Furthermore, by (\ref{20240721-2}), \eqref{20240720-2} and the mean value theorem, we get
	\begin{eqnarray}\label{20246131718}
		\begin{aligned} &\sup_{\mathcal{O}^s(\sigma_{m+2},\frac{\mu_0}{2}(1+2^{-m-2}))}\|\Phi^{m+1}-\Phi^{m}\|_s\\
&\leq\sup_{\mathcal{O}^s(\sigma_{m+1},\frac{\mu_0}{2}(1+2^{-m-1}))}\|D\Phi^m\|_{\mathcal{L}(\mathbb{C}^{2n}\times Y_s,\mathbb{C}^{2n}\times Y_s)}\sup_{\mathcal{O}^s(\sigma_{m+2},\frac{\mu_0}{2}(1+2^{-m-2}))}\|\Phi_{m+1}-id\|_s\\
&\leq2C\Big(\frac{\varepsilon_m}{\mu_m^2}\Big)^{\frac{3}{4}}E_m^{\frac{1}{4}}\\
&\leq\frac{\varepsilon_0^{\frac{1}{2}}}{2^{m+1}},			
		\end{aligned}
	\end{eqnarray}
where we denote $\Phi^0=id$ and in the last inequality we use the facts
$$2C\Big(\frac{\varepsilon_m}{\mu_m^2}E_m\Big)^{\frac{1}{4}}<\frac{\mu_0}{2^{m+1}}, ~~~~~~\Big(\frac{\varepsilon_m}{\mu_m^2}\Big)^{\frac{1}{2}}\leq\Big(\frac{\varepsilon_0}{\mu_0^2}\Big)^{\frac{1}{2}}.$$ %
This shows that $\Phi^m$ converge uniformly on $\mathcal{O}^s(\frac{\sigma_{0}}{2},\frac{\mu_0}{2})\times\mathcal{D}_*$, where $\mathcal{D}_*=\cap_{m\geq0}\mathcal{D}_m$. Let $\Phi=\lim_{m\to+\infty}\Phi^m$. Summing in \eqref{20246131718} with respect to $m\geq0$, we get $\|\Phi-id\|_s\leq\varepsilon_0^{\frac{1}{2}}$. Since the original Hamiltonian $H$ is analytic in $\mathcal{O}^s(\sigma_{0},\mu_0)$,
	\begin{eqnarray*}
		H\circ\Phi&=&\lim_{m\to+\infty} H\circ\Phi^m\\
&=&\lim_{m\to+\infty}\langle\omega_m,r\rangle+\lim_{m\to+\infty}\frac{1}{2}\langle\zeta,A_m\zeta\rangle+\lim_{m\to+\infty}f_m\\
&=&\langle\omega_*,r\rangle+\frac{1}{2}\langle\zeta,A_*\zeta\rangle+f_*.
	\end{eqnarray*}
Summing in \eqref{20246111948}, \eqref{20240719-5} with respect to $m\geq0$, we get the estimates of $\omega_*$, $A_*$ in \eqref{20240721-7} and \eqref{20240721-8}. Moreover, for $r=\zeta=0$ and $|\Im\theta|<\frac{\sigma_{0}}{2}$, by \eqref{20246101040}, (\ref{20246131724}) and Cauchy estimates, we have
	\begin{eqnarray*}	|\partial_{r}f_m|,~~\|\partial_{\zeta}f_m\|_s,~~|\nabla^2_{\zeta}f_m|_{s,\beta_{m}}
\leq\frac{\varepsilon_{m}}{\mu_{m}^2}\cdot\frac{2^{m+2}}{\sigma_0}
<\frac{\gamma_0}{2^{3m+2}\sigma_0},
	\end{eqnarray*}
which leads to $\partial_r{f_*}=\partial_{\zeta}{f_*}=\partial^2_{\zeta}{f_*}=0$ by letting $m\to+\infty$.

\section{Measure estimate}
	From (\ref{20246191530}), we know that
	\begin{eqnarray}
		\mathcal{D}\backslash \mathcal{D}_*=\bigcup_{m\geq 0}\bigcup_{k\in\mathbb{Z}^n\atop |k|\leq N_m}\big(\mathfrak{R}_{k}^{m,0}\bigcup \mathfrak{R}_{k}^{m,1}\bigcup \mathfrak{R}_{k}^{m,20}\bigcup \mathfrak{R}_{k}^{m,11}\big),
	\end{eqnarray}
where $\mathfrak{R}_{k}^{m,0}$, $\mathfrak{R}_{k}^{m,1}$, $\mathfrak{R}_{k}^{m,20}$ and $\mathfrak{R}_{k}^{m,11}$ are defined in \eqref{20246191517}. In the following, we only estimate the set
	\begin{eqnarray}\label{20240716-1}
		\Theta:=\bigcup_{m\geq 0}\bigcup_{k\in\mathbb{Z}^n\atop |k|\leq N_m}{\mathfrak{R}}_k^{m,11}
	\end{eqnarray}
in detail, since the other three cases are much simpler.

Fix $m\geq 0$, $0<|k|\leq N_m$ and $\kappa>0$, we firstly study the set
	\begin{eqnarray}\label{20240716-2}
		\begin{aligned}	\bar{\mathfrak{R}}_k^m(\kappa)&:=\bigcup_{i,j\geq1}\bar{\mathfrak{R}}_{k,i,j}^{m}(\kappa)\\
&=\bigcup_{i,j\geq1}\big\{\rho\in\mathcal{D}_{m}:~|\langle k,\omega_{m}(\rho)\rangle-\lambda_i+\lambda_j|<\kappa(1+|i-j|)\big\}
		\end{aligned}
	\end{eqnarray}
	with $\lambda_i=b_1i+b_0$ and $\lambda_j=b_1j+b_0$.
	\begin{lem}\label{20246152222}		
		Assuming $\kappa<\frac{b_1}{3}$, then we have		
		\begin{eqnarray}\label{20247132036}
			\text{meas}\big(\bar{\mathfrak{R}}_k^m(\kappa)\big)\leq C|k|\kappa,
		\end{eqnarray}		
		where $C>0$ is a constant only depending on $n,b_1,M_{\omega,0},L_0$.	
	\end{lem}
	\begin{proof}
	We introduce the perturbed frequencies $\omega=\omega_{m}(\rho)$ as parameters over the domain $\Delta:=\omega_{m}(\mathcal{D}_{m})$ and consider the resonance zones $\dot{\bar{\mathfrak{R}}}_{k,i,j}^{m}(\kappa)=\omega_{m}({\bar{\mathfrak{R}}}_{k,i,j}^{m}(\kappa))$. From the iterative lemma, we know
	\begin{eqnarray}\label{20246151439}
		\sup_{\omega\in\Delta}|\omega|\leq M_{\omega,m}\leq 2M_{\omega,0},~~~~~~\sup_{\omega\in\Delta}|\partial_{\omega}\rho|\leq L_m\leq 2L_0.
	\end{eqnarray}

Now we consider a fixed $\dot{\bar{\mathfrak{R}}}_{k,i,j}^{m}(\kappa)$.
If $0<|k|\leq \frac{b_1}{6M_{\omega,0}}|i-j|$, then
$$|\langle k,\omega\rangle|\leq2M_{\omega,0}|k|\leq\frac{b_1}{3}|i-j|$$
and thus
$$|\langle k,\omega\rangle-\lambda_i+\lambda_j|\geq\frac{2b_1}{3}|i-j|>\kappa(1+|i-j|),$$
which implies $\dot{\bar{\mathfrak{R}}}_{k,i,j}^{m}(\kappa)=\emptyset$;
otherwise, $|k|> \frac{b_1}{6M_{\omega,0}}|i-j|$, fixing $\vec{v}\in\{-1,1\}^n$ such that $k\cdot\vec{v}=|k|$, then it is clear that
	\begin{eqnarray*}
		\partial_{\vec{v}}\big(\langle k,\omega\rangle-\lambda_i+\lambda_j\big)=|k|,
	\end{eqnarray*}
	which leads to
	\begin{eqnarray}\label{20240716-3}
		\begin{aligned}
			\text{meas}\big(\dot{\bar{\mathfrak{R}}}_{k,i,j}^{m}(\kappa)\big)&\leq (\text{diag}\Delta)^{n-1}\frac{\kappa(1+|i-j|)}{|k|}\\
&<(4M_{\omega,0})^{n-1}\Big(1+\frac{6M_{\omega,0}}{b_1}\Big)\kappa.
		\end{aligned}
	\end{eqnarray}

Summing \eqref{20240716-3} with respect to $|i-j|<\frac{6M_{\omega,0}}{b_1}|k|$, one has
	\begin{eqnarray}		\text{meas}\big(\dot{\bar{\mathfrak{R}}}_k^m(\kappa)\big)<(4M_{\omega,0})^{n-1}\Big(1+\frac{6M_{\omega,0}}{b_1}\Big)\Big(1+\frac{12M_{\omega,0}}{b_1}|k|\Big)\kappa.
	\end{eqnarray}
	Going back to the original parameter domain $\mathcal{D}_{m}$ by the inverse frequency map $\omega_{m}^{-1}$, we finally get
\begin{eqnarray}		
\text{meas}\big(\bar{\mathfrak{R}}_k^m(\kappa)\big)\leq L_m^n\text{meas}\big(\dot{\bar{\mathfrak{R}}}_k^m(\kappa)\big)<(2L_0)^n(4M_{\omega,0})^{n-1}\Big(1+\frac{6M_{\omega,0}}{b_1}\Big)\Big(1+\frac{12M_{\omega,0}}{b_1}|k|\Big)\kappa,
	\end{eqnarray}
which is \eqref{20247132036} in the lemma.
	\end{proof}

Next, we give the estimate of $\Theta$ in \eqref{20240716-1}.
\begin{lem}\label{20240726-7}
We have
		\begin{eqnarray}\label{20240717-10}
			\text{meas}(\Theta)<\varepsilon_0^{\frac{1}{\alpha}}.
		\end{eqnarray}
	\end{lem}
	\begin{proof}
		If $k=0$, then we know $i\neq j$. Using (\ref{20246111950}), the facts $M_{\Omega,0}<\frac{b_1}{8}$ and $\kappa_{m}\leq\kappa_0<\frac{b_1}{4}$, we obtain
		\begin{eqnarray}\label{20246161509}
			\begin{aligned}
				&|\Omega_{m,j,\iota_2}(\rho)-\Omega_{m,i,\iota_1}(\rho)|\\&\geq b_1|i-j|-\frac{M_{\Omega,m}}{j^{\beta_m}}-\frac{M_{\Omega,m}}{i^{\beta_m}}\\&\geq b_1|i-j|-4M_{\Omega,0}\\&>\kappa_{m}(1+|i-j|),
			\end{aligned}
		\end{eqnarray}
which implies
\begin{eqnarray}\label{20240717-1}
 \mathfrak{R}_0^{m,11}=\emptyset.
\end{eqnarray}

It remains to consider the case $0<|k|\leq N_m$. Similarly, we introduce the perturbed frequencies $\omega=\omega_{m}(\rho)$ as parameters over the domain $\Delta:=\omega_{m}(\mathcal{D}_{m})$ and consider the resonance zones $\dot{\mathfrak{R}}_{k,i,j,\iota_1,\iota_2}^{m,11}(\kappa_m)$. The eigenvalues $\Omega_{m,j,\iota}$ can be seen as functions of $\omega$. Since $(Q_m)_j^j$ is Hermitian, by the variation principle of eigenvalue of matrix,
		\begin{eqnarray}\label{202461615099}
			|\partial_{\vec{v}}\Omega_{m,j,\iota}(\omega)|\leq M_{\Omega,m}L_m\leq 4M_{\Omega,0}\cdot L_0< \frac{1}{3},
		\end{eqnarray}
		where $\vec{v}\in\{-1,1\}^n$ such that $k\cdot\vec{v}=|k|$.

For a fixed $\dot{{\mathfrak{R}}}_{k,i,j,\iota_1,\iota_2}^{m,11}$, if $0<|k|\leq \frac{b_1}{6M_{\omega,0}}|i-j|$, then by (\ref{20246151439}), (\ref{20246161509}), the facts $M_{\Omega,0}<\frac{b_1}{8}$ and $\kappa_{m}\leq \kappa_0<\frac{b_1}{12}$, we know
		\begin{eqnarray}
			\begin{aligned}
				&|\langle k,\omega\rangle-\Omega_{m,i,\iota_1}(\rho)+\Omega_{m,j,\iota_2}(\rho)|\\
&\geq|\Omega_{m,j,\iota_2}(\rho)-\Omega_{m,i,\iota_1}(\rho)|-|\langle k,\omega\rangle|\\
&\geq b_1|i-j|-4M_{\Omega,0}-2M_{\omega,0}|k|\\
&>\kappa_{m}(1+|i-j|),
			\end{aligned}
		\end{eqnarray}
which implies $\dot{\mathfrak{R}}_{k,i,j,\iota_1,\iota_2}^{m,11}=\emptyset$; otherwise, $|k|>\frac{b_1}{6M_{\omega,0}}|i-j|$, then by (\ref{202461615099}), we have
		\begin{eqnarray*}
			\partial_{\vec{v}}(\langle k,\omega\rangle-\Omega_{m,i,\iota_1}(\rho)+\Omega_{m,j,\iota_2}(\rho))> |k|-\frac{2}{3}\geq \frac{|k|}{3},
		\end{eqnarray*}
		which leads to
\begin{eqnarray}\label{20240716-4} \text{meas}(\dot{\mathfrak{R}}_{k,i,j,\iota_1,\iota_2}^{m,11})\leq \frac{3\kappa_{m}(1+|i-j|)}{|k|}<\frac{36M_{\omega,0}}{b_1}\kappa_{m}.
\end{eqnarray}
Going back to the original parameter domain $\mathcal{D}_{m}$, we get
		\begin{eqnarray}\label{2024616150911}
\text{meas}(\mathfrak{R}_{k,i,j,\iota_1,\iota_2}^{m,11})<(2L_0)^n\frac{36M_{\omega,0}}{b_1}\kappa_{m}.
		\end{eqnarray}
In the following, we estimate the measure of
\begin{eqnarray}\label{20240717-2}
\mathfrak{R}_{k}^{m,11}=\bigcup_{|i-j|<\frac{6M_{\omega,0}}{b_1}|k|}\mathfrak{R}_{k,i,j,\iota_1,\iota_2}^{m,11}
:=\Big(\bigcup_{\mathfrak{B}_1}\mathfrak{R}_{k,i,j,\iota_1,\iota_2}^{m,11}\Big)\cup\Big(\bigcup_{\mathfrak{B}_2}\mathfrak{R}_{k,i,j,\iota_1,\iota_2}^{m,11}\Big)
\end{eqnarray}
with the index sets
\begin{eqnarray*}
\begin{aligned} \mathfrak{B}_1&=\Big\{((i,\iota_1),(j,\iota_2))\in\mathfrak{L}\times\mathfrak{L}~:~|i-j|<\frac{6M_{\omega,0}}{b_1}|k|,~\min\{i,j\}{\geq}J_m\Big\},\\
\mathfrak{B}_2&=\Big\{((i,\iota_1),(j,\iota_2))\in\mathfrak{L}\times\mathfrak{L}~:~|i-j|<\frac{6M_{\omega,0}}{b_1}|k|,~\min\{i,j\}<J_m\Big\},
\end{aligned}
\end{eqnarray*}
where $J_m=(\frac{4M_{\Omega,0}}{\kappa_{m}^{2\alpha^*}})^{\frac{1}{\beta_m}}$ and $\alpha^*=\frac{\beta_0}{2\beta_0+8d^*+4}$. By the fact $\beta_m>\frac{\beta_0}{2}$, we have
		\begin{eqnarray}\label{20240721-9}	\kappa_{m}^{2\alpha^*}<\big(\frac{\varepsilon_m}{\mu_m^2}\big)^{\frac{2}{\alpha}}.
		\end{eqnarray}
Then by \eqref{20240721-9} and the fact $\beta_m<1$, we have
		\begin{eqnarray}\label{20240721-10}	J_m>4M_{\Omega,0}\big(\frac{\varepsilon_m}{\mu_m^2}\big)^{-\frac{2}{\alpha}}>\frac{6M_{\omega,0}}{b_1}N_m,
		\end{eqnarray}
where the last inequality follows from \eqref{20246132026} and $\gamma_0$ small enough.

Firstly, consider the case $((i,\iota_1),(j,\iota_2))\in\mathfrak{B}_1$. If
\begin{eqnarray*}
|\langle k,\omega(\rho)\rangle-\lambda_i+\lambda_j|\geq 2\kappa_{m}^{2\alpha^*}(1+|i-j|),
\end{eqnarray*}
then we have
		\begin{eqnarray}
			\begin{aligned}
				&|\langle k,\omega(\rho)\rangle-\Omega_{m,i,\iota_1}(\rho)+\Omega_{m,j,\iota_2}(\rho)|\\
&\geq|\langle k,\omega(\rho)\rangle-\lambda_i+\lambda_j|-\frac{2M_{\Omega,0}}{i^{\beta_m}}-\frac{2M_{\Omega,0}}{j^{\beta_m}}\\
&\geq2\kappa_{m}^{2\alpha^*}(1+|i-j|)-\kappa_{m}^{2\alpha^*}\\
&\geq \kappa_{m}^{2\alpha^*}(1+|i-j|)\\
&\geq \kappa_{m}(1+|i-j|),
			\end{aligned}
		\end{eqnarray}
which means
\begin{eqnarray}\label{20240717-3}
\mathfrak{R}_{k,i,j,\iota_1,\iota_2}^{m,11}\subset\bar{\mathfrak{R}}_{k,i,j}^{m}(2\kappa_{m}^{2\alpha^*})\subset\bar{\mathfrak{R}}_{k}^{m}(2\kappa_{m}^{2\alpha^*}).
\end{eqnarray}
Thus, by \eqref{20240717-3} and Lemma \ref{20246152222}, we get
\begin{eqnarray}\label{20240717-7}
\text{meas}\Big(\bigcup_{\mathfrak{B}_1}\mathfrak{R}_{k,i,j,\iota_1,\iota_2}^{m,11}\Big)
\leq\text{meas}\Big(\bar{\mathfrak{R}}_{k}^{m}(2\kappa_{m}^{2\alpha^*})\Big)
\leq C|k|\kappa_{m}^{2\alpha^*},
\end{eqnarray}
where $C>0$ is a constant only depending on $n,b_1,M_{\omega,0},L_0$.

Next, consider the case $((i,\iota_1),(j,\iota_2))\in\mathfrak{B}_2$.
We will count the number of the set $\mathfrak{B}_2$, denoted by $\#\mathfrak{B}_2$. By \eqref{20240721-10}, we know
\begin{eqnarray}\label{20240717-4}
|i-j|<\frac{6M_{\omega,0}}{b_1}|k|\leq\frac{6M_{\omega,0}}{b_1}N_m<J_m,
\end{eqnarray}
Assuming $i\leq j$, then by \eqref{20240717-4} and $\text{min}\{i,j\}< J_m$, we know $i<J_m$ and $j<2J_m$. Thus,
\begin{eqnarray}\label{20240717-5}
\begin{aligned}
\#\mathfrak{B}_2&\leq2\sum_{i<J_m,\atop i\leq j<\frac{6M_{\omega,0}}{b_1}|k|+i}c^*i^{d^*}c^*j^{d^*}\\				
&\leq2(c^*)^2J_m^{d^*}(2J_m)^{d^*}\frac{6M_{\omega,0}}{b_1}|k|J_m\\
&=2^{d^*+1}(c^*)^2\frac{6M_{\omega,0}}{b_1}|k|J_m^{2d^*+1}.
\end{aligned}
\end{eqnarray}
By \eqref{2024616150911}, \eqref{20240717-5} and the fact $J_m\leq(\frac{4M_{\Omega,0}}{\kappa_{m}^{2\alpha^*}})^{\frac{2}{\beta_0}}$, we get
\begin{eqnarray}\label{20240717-6}
\text{meas}\Big(\bigcup_{\mathfrak{B}_2}\mathfrak{R}_{k,i,j,\iota_1,\iota_2}^{m,11}\Big)
\leq C|k|\kappa_{m}^{2\alpha^*},
\end{eqnarray}
where $C>0$ is a constant only depending on $n,b_1,c^*,d^*,M_{\omega,0},L_0,M_{\Omega,0},\beta_0$.

Summing \eqref{20240717-7} and \eqref{20240717-6}, we get
\begin{eqnarray}\label{20240717-8}
\text{meas}\Big(\mathfrak{R}_{k}^{m,11}\Big)
\leq C|k|\kappa_{m}^{2\alpha^*},
\end{eqnarray}
and further taking sum with respect to $k=0$ in \eqref{20240717-1} and $0<|k|\leq N_m$ in \eqref{20240717-8},
\begin{eqnarray}\label{20240717-9}
\begin{aligned}
\text{meas}\Big(\bigcup_{|k|\leq N_m}\mathfrak{R}_k^{m,11}\Big)
&\leq CN_m^{n+1}\kappa_{m}^{2\alpha^*},\\
&\leq C\big(\frac{\varepsilon_m}{\mu_m^2}\big)^{\frac{1}{\alpha}\big(1+\frac{d^*+1}{n+d^*+2}\big)},
\end{aligned}
\end{eqnarray}
where $C>0$ is a constant only depending on $n,b_1,c^*,d^*,M_{\omega,0},L_0,M_{\Omega,0},\beta_0$, and the last inequality follows from \eqref{20246132026} and \eqref{20240721-9}.
Finally, taking sum in \eqref{20240717-9} with respect to $m\geq0$, this lemma is proved.
	\end{proof}

	\section{Appendix}
	\begin{lem}\label{20243122043}
	Assuming $j_1,j_2,j_3\in \mathbb{N}^+$ with $j_1\leq j_2\leq j_3$, then we have
	\begin{eqnarray}\label{20243121618}
		 \max\{w(j_1,j_2),w(j_2,j_3)\}\leq w(j_1,j_3)\leq w(j_1,j_2)w(j_2,j_3).
	\end{eqnarray}
\end{lem}
\begin{proof}
See Lemma A.1 and its proof in \cite{Gre3}.
\end{proof}

\begin{lem}\label{20243122017}
	Assuming $\beta>0$ and $j\in \mathbb{N}^+$, then we have
\begin{eqnarray}\label{20240527-1}
\sum_{k\in \mathbb{N}^+}\frac{1}{k^{1+\beta}}<1+\frac{1}{\beta},
\end{eqnarray}
	\begin{eqnarray}\label{20240520-1}
		\sum_{k\in \mathbb{N}^+}\frac{1}{k^{\beta}(1+|k-j|)}<2+\frac{2}{\beta}.
	\end{eqnarray}
\end{lem}
\begin{proof}
Firstly, \eqref{20240527-1} follows from the fact
\begin{eqnarray*}
\sum_{k\in \mathbb{N}^+}\frac{1}{k^{1+\beta}}<1+\int_1^{+\infty}\frac{1}{x^{1+\beta}}dx=1+\frac{1}{\beta}.
\end{eqnarray*}
Then, by H\"{o}lder inequality and \eqref{20240527-1},
	\begin{eqnarray*}
		\sum_{k\in \mathbb{N}^+}\frac{1}{k^{\beta}(1+|k-j|)}&\leq&\Big(\sum_{k\in \mathbb{N}^+}\frac{1}{k^{1+\beta}}\Big)^{\frac{\beta}{1+\beta}}\Big(\sum_{k\in \mathbb{N}^+}\frac{1}{(1+|k-j|)^{1+\beta}}\Big)^{\frac{1}{1+\beta}}\\
&<&\Big(\sum_{k\in \mathbb{N}^+}\frac{1}{k^{1+\beta}}\Big)^{\frac{\beta}{1+\beta}}\Big(2\sum_{k\in \mathbb{N}^+}\frac{1}{k^{1+\beta}}\Big)^{\frac{1}{1+\beta}}\\
&<&2\sum_{k\in \mathbb{N}^+}\frac{1}{k^{1+\beta}}\\
&<&2+\frac{2}{\beta}.
	\end{eqnarray*}
\end{proof}

\begin{lem}\label{20240527-2}
	Let $B=\big(b_i^j\big)_{i,j\in\mathbb{N}^+}$ be a complex infinite-dimensional matrix. If its Frobenius norm
$$\|B\|_F:=\Big(\sum_{i,j\in\mathbb{N}^+}|b_i^j|^2\Big)^{\frac{1}{2}}<+\infty,$$
then $B$ is a bounded operator on $\ell^2(\mathbb{N}^+)$ with its $\ell^2$-operator norm satisfying
	\begin{eqnarray}\label{20240527-3}
		 \|B\|_{\ell^2\rightarrow\ell^2}\leq\|B\|_F.
	\end{eqnarray}
\end{lem}
\begin{proof}
For $z=(z_j)_{j\in\mathbb{N}^+}\in\ell^2(\mathbb{N}^+)$, by Schwarz inequality, we get
			\begin{eqnarray*}
				\begin{aligned}					
\|Bz\|_{\ell^2}^2&=\sum_{i\in\mathbb{N}^+}\Big|\sum_{ j\in\mathbb{N}^+}b_{i}^{j}z_{j}\Big|^2\\
&\leq \sum_{i\in\mathbb{N}^+}\Big(\sum_{ j\in\mathbb{N}^+}\big|b_{i}^{j}\big|^2\Big)\Big(\sum_{ j\in\mathbb{N}^+}|z_j|^2\Big)\\
&=\|B\|_F^2\|z\|_{\ell^2}^2.
				\end{aligned}
			\end{eqnarray*}
From this the estimate \eqref{20240527-3} follows.
\end{proof}

\begin{lem}\label{20240527-4}
	Let $B=\big(b_i^j\big)_{i,j\in\mathbb{N}^+}$ be a complex infinite-dimensional matrix. If both $\ell^{\infty}$-operator norm and $\ell^1$-operator norm are bounded, i.e.,
$$\|B\|_{\ell^\infty\rightarrow\ell^\infty}:=\sup_{i\in\mathbb{N}^+}\sum_{j\in\mathbb{N}^+}|b_i^j|<+\infty,$$
$$\|B\|_{\ell^1\rightarrow\ell^1}:=\sup_{j\in\mathbb{N}^+}\sum_{i\in\mathbb{N}^+}|b_i^j|<+\infty,$$
then $B$ is a bounded operator on $\ell^2(\mathbb{N}^+)$ with its $\ell^2$-operator norm satisfying
	\begin{eqnarray}\label{20240527-5}
		 \|B\|_{\ell^2\rightarrow\ell^2}\leq\|B\|_{\ell^\infty\rightarrow\ell^\infty}^{\frac{1}{2}}\|B\|_{\ell^1\rightarrow\ell^1}^{\frac{1}{2}}.
	\end{eqnarray}
\end{lem}
\begin{proof}
For $z=(z_j)_{j\in\mathbb{N}^+}\in\ell^2(\mathbb{N}^+)$, by Schwarz inequality, we get
			\begin{eqnarray*}
				\begin{aligned}					
\|Bz\|_{\ell^2}^2&=\sum_{i\in\mathbb{N}^+}\Big|\sum_{ j\in\mathbb{N}^+}b_{i}^{j}z_{j}\Big|^2\\
&\leq \sum_{i\in\mathbb{N}^+}\Big(\sum_{ j\in\mathbb{N}^+}\big|b_{i}^{j}\big|\Big)\Big(\sum_{ j\in\mathbb{N}^+}\big|b_{i}^{j}\big||z_j|^2\Big)\\
&\leq\Big(\sup_{i\in\mathbb{N}^+}\sum_{ j\in\mathbb{N}^+}\big|b_{i}^{j}\big|\Big)\Big(\sum_{i, j\in\mathbb{N}^+}\big|b_{i}^{j}\big||z_j|^2\Big)\\
&=\Big(\sup_{i\in\mathbb{N}^+}\sum_{ j\in\mathbb{N}^+}\big|b_{i}^{j}\big|\Big)\sum_{j\in\mathbb{N}^+}\Big(\sum_{i\in\mathbb{N}^+}\big|b_{i}^{j}\big|\Big)|z_j|^2\\
&\leq\Big(\sup_{i\in\mathbb{N}^+}\sum_{ j\in\mathbb{N}^+}\big|b_{i}^{j}\big|\Big)\Big(\sup_{j\in\mathbb{N}^+}\sum_{i\in\mathbb{N}^+}\big|b_{i}^{j}\big|\Big)\sum_{j\in\mathbb{N}^+}|z_j|^2.
				\end{aligned}
			\end{eqnarray*}
From this the estimate \eqref{20240527-5} follows.
\end{proof}

\begin{lem}\label{20246251433}
Assume $A\in\mathcal{M}_{0,0}$. For any $k\in\mathbb{Z}^n$ with $|k|\leq N$, define $B(k)$ by
	\begin{eqnarray*}		B_{i,\iota_1}^{j,\iota_2}(k)=\frac{A_{i,\iota_1}^{j,\iota_2}}{\langle k,\omega\rangle-\Omega_{i,\iota_1}+\Omega_{j,\iota_2}},~~~\text{for}~~~|k|+|i-j|\neq0,
	\end{eqnarray*}
and $B_{j,\iota_1}^{j,\iota_2}(0)=0$, where $\omega\in\mathbb{R}^n$, $\{\Omega_{j,\iota}\}_{(j,\iota)\in \mathfrak{L}}$ is a sequence of real numbers satisfying
	\begin{eqnarray*}
		|\Omega_{j,\iota}-b_1j-b_0|\leq \frac{M_{\Omega}}{j^{\beta}},~~~\text{for}~~~(j,\iota)\in \mathfrak{L}
	\end{eqnarray*}
	with $b_1,b_0$ in \eqref{20240709-5} and  $M_{\Omega}, \beta$ being given positive real numbers. Moreover, assume
	\begin{eqnarray*}
		|\langle k,\omega\rangle-\Omega_{i,\iota_1}+\Omega_{j,\iota_2}|\geq \kappa(1+|i-j|)
	\end{eqnarray*}
for $(i,\iota_1),(j,\iota_2)\in\mathfrak{L}$ with $|k|+|i-j|\neq0$.	
Then $B(k)\in\mathcal{M}_{0,0}$, and there exists a constant $C>0$ depending on $n,b_0,b_1,c^*,d^*,|\omega|,M_{\Omega},\beta$ such that
	\begin{eqnarray} \|B_i^j(k)\|\leq\frac{CN^{\frac{d^*}{2}}}{\kappa^{\frac{d^*}{2\beta}+1}(1+|i-j|)}\|A_i^j\|.
	\end{eqnarray}
\end{lem}
\begin{proof}
See Lemma $4.3$ and its proof in \cite{Gre3}.
\end{proof}

\begin{lem}\label{20243212004}
		Assume $f\in \mathcal{T}_{\sigma,\mu,\mathcal{D}}^{s,\beta}$. Then we have $f^T\in \mathcal{T}_{\sigma,\mu,\mathcal{D}}^{s,\beta}$,
		\begin{eqnarray}
			[f^T]_{\sigma,\mu,\mathcal{D}}^{s,\beta}\leq C[f]_{\sigma,\mu,\mathcal{D}}^{s,\beta},
		\end{eqnarray}
and for any $0<\mu'<\mu$,
		\begin{eqnarray}
			[f-f^T]_{\sigma,\mu',\mathcal{D}}^{s,\beta}\leq C\Big(\frac{\mu'}{\mu}\Big)^3[f]_{\sigma,\mu,\mathcal{D}}^{s,\beta},
		\end{eqnarray}
where $C>0$ is an absolute constant.
\end{lem}
\begin{proof}
See Proposition 4.2 and its proof in \cite{Gre2}.
\end{proof}
\begin{lem}\label{20247212055}
Assume $s>d-2\geq0$ and $V(x)\in\mathcal{H}^s$. Define the matrix $Q$ with its element
$$Q_{i,\iota_1}^{j,\iota_2}=\int_{\mathbb{R}^d}V(x)\Psi_{i,\iota_1}(x)\Psi_{j,\iota_2}(x)dx,~~~~~~(i,\iota_1),(j,\iota_2)\in\mathfrak{L}.$$
	Then there exists $\beta\in(0,\frac{1}{8})$ only depending on $d,s$ such that for any $i,j\geq 1$,
	\begin{eqnarray}\label{20247211752}
		\big\|Q_i^j\big\|\leq \frac{C}{(ij)^{\frac{\beta}{2}}w^{s}(i,j)}\|V\|_s,
	\end{eqnarray}
where $C$ is a positive constant only depending on $d,s$.
\end{lem}
	\begin{proof}
See Lemma 3.2 and its proof in \cite{Gre3}.
	\end{proof}

\end{document}